\documentclass{article}

\usepackage{amsmath,amsfonts,amssymb,amsthm,amscd}
\usepackage{dsfont, mathrsfs,enumerate,eucal,fontenc}
\usepackage[varg]{txfonts}
\usepackage{xspace}
\usepackage[english]{babel}

\usepackage[utf8]{inputenc}
\usepackage{amsmath}
\usepackage{amsthm}
\usepackage{amssymb}
\usepackage{amscd}
\usepackage[all]{xy} 
\usepackage{latexsym}
\usepackage{extpfeil}
\usepackage{graphicx}
\usepackage{changebar,color}
\usepackage{tikz-cd}
\usepackage{bbold}
\usepackage[T1]{fontenc}
\usepackage[english]{babel}
\usepackage{tabularx}
\usepackage{comment}
\usepackage{enumitem}

\usepackage[backend=biber,style=alphabetic,maxnames=99]{biblatex} 
\usepackage[colorlinks=true, linkcolor=blue, citecolor=blue, urlcolor=purple]{hyperref}

\newtheorem{introtheorem}{Theorem}[section]

\newtheorem{introconjecture}{Conjecture}[section]

\newtheorem{theoreme}{Theorem}[section]          
\newtheorem{definition}[theoreme]{Definition}
\newtheorem{proposition}[theoreme]{Proposition}

\newtheorem{lemme}[theoreme]{Lemma}
\newtheorem{example}[theoreme]{Example}

\newtheorem{corollaire}[theoreme]{Corollary}
\newtheorem{remarque}[theoreme]{Remark}
\newtheorem{conjecture}[theoreme]{Conjecture}

\def\G{\mathbb{G}}
\def\Q{\mathbb{Q}}
\def\C{\mathbb{C}}
\def\Z{\mathbb{Z}}

\def\d{\mathrm{d}}
\def\exp{\mathrm{exp}}

\def\aref{\textcolor{red}{à ref}}

\def\id{\mathrm{id}}
\def\Ind{\textrm{Ind}}
\def\vac{|0\rangle}
\def\vaclambda{| \lambda \rangle}
\def\Lie{\mathrm{Lie}}

\def\Der{\mathrm{Der}}
\def\Ind{\mathrm{Ind}}
\def\Hom{\mathrm{Hom}}
\def\End{\mathrm{End}}

\def\Ad{\mathrm{Ad}}
\def\ad{\mathrm{ad}}
\def\GL{\mathrm{GL}}
\def\IndKacMoody{\underset{\Lgkappa \leftarrow \Lg[[t]] \ra \Lg}{\mathrm{Ind}}}

\def\Zhu{\mathrm{Zhu}}

\def\Spec{\mathrm{Spec}}

\def\cO{\mathcal{O}}

\def\cF{\mathcal{F}}
\def\cG{\mathcal{G}}

\def\cD{\mathcal{D}}
\def\cU{\mathcal{U}}

\def\cSF{\mathcal{S}\mathcal{F}}

\def\La{\mathfrak{a}}
\def\Lg{\mathfrak{g}}
\def\Lh{\mathfrak{h}}
\def\Lb{\mathfrak{b}}
\def\Ln{\mathfrak{n}}
\def\Lgl{\mathfrak{g}\mathfrak{l}}
\def\Lzg{\mathfrak{z}(\Lg)}
\def\Lgkappa{\widehat{\Lg}^{\kappa}}
\def\Lgkappastar{\widehat{\Lg}^{\kappa^{*}}}

\def\Lhkappa{\widehat{\Lh}^{\kappa}}
\def\Lhkappastar{\widehat{\Lh}^{\kappa^{*}}}
\def\Lhkappakappastar{\Lhkappa \oplus \Lhkappastar}
\def\Lgkappacategoryintegrable{	V^{\kappa}(\Lg)\mathrm{-Mod}^{\jetinf G}}

\def\Satakecategory{\cdo\mathrm{-Mod}^{\jetinf G\times \jetinf G}}

\def\ra{\rightarrow}
\def\lra{\longrightarrow}

\def\jetinf{\mathcal{J}_{\infty}}

\def\cdo{\cD_{G}^{\kappa}}

\def\ewalg{\mathcal{W}_{G}^{\kappa}}

\usepackage{mdframed}
\newenvironment{note}
{ \begin{mdframed}[]  }
	{  \end{mdframed}}

\newif\ifhidenotes
\hidenotestrue  

\ifhidenotes
\usepackage{environ}
\NewEnviron{hide}{}

\fi

\title{Representation theory of the principal equivariant affine $\mathcal{W}$-algebra and Langlands~duality}
\author{Damien Simon\footnote{damien.simon@universite-paris-saclay.fr}}
\date{}

\begin{document}
	
	\maketitle 
	
	\begin{abstract}
		We study the structure and the representation theory of a certain class of vertex algebras. Our study was partly motivated by the quantum geometric Langlands program and we explain what some of our results mean in this framework.
		
		We begin our investigation with the vertex algebra of chiral differential operators on a reductive group $\cdo$ for generic levels. In particular, we prove that its vertex algebraic structure is essentially unique. We also study its representation theory and show that the geometric Satake equivalence degenerates. The latter leads us to formulate a vertex-algebraic version of the fundamental local equivalence of Gaitsgory and Lurie.
		In turn, this brings us to study the representation theory of the principal equivariant affine $\mathcal{W}$-algebra $\ewalg$, defined by Arakawa as the principal quantum Hamiltonian reduction of $\cdo$. We construct a family of simple modules for $\ewalg$ whose combinatorics matches that of the representation theory of the Langlands dual group. Finally, we establish the fundamental local equivalence when the group is an algebraic torus or is simple adjoint of classical simply laced type.
	\end{abstract}
	\newpage
	\tableofcontents

	\section*{Introduction}
	
	We study the structure and representation theory of two vertex algebras attached to a reductive algebraic group: its algebra of chiral differential operators, and the principal equivariant $\mathcal{W}$-algebra obtained from the latter by quantum Hamiltonian reduction. Throughout, the Langlands dual combinatorics emerges from their representations.
	
	Let $G$ be a reductive algebraic group over the field $\C$ of complex numbers with Lie algebra $\Lg$. Fix a level $\kappa$, which is a bilinear symmetric invariant form on $\Lg$. If $G$ is the multiplicative group or is a simple group, a level can be thought of as a scalar. Let $\cdo$ be the vertex algebra of chiral differential operators on $G$ at level $\kappa$. It comes equipped with an embedding of the universal affine vertex algebra $V^{\kappa}(\Lg)$. At generic levels, $\cdo$ admits a chiral Peter--Weyl decomposition (Theorem \ref{theorem chiral Peter-Weyl}). Our first result establishes that this decomposition uniquely determines the vertex algebraic structure of $\cdo$.
	\begin{introtheorem}[Theorem \ref{theorem rigidity of chiral Peter-Weyl}]	
		For a generic level $\kappa$, the vertex algebra structure on $\cdo$ is uniquely determined by the fact it is a simple conformal extension of $V^{\kappa}(\Lg) \otimes V^{\kappa^{*}}(\Lg)$, where $\kappa^{*}$ is defined by the relation 
		$$
		\kappa + \kappa^{*} = -\kappa_{\Lg},
		$$
		with $\kappa_{\Lg}$ the Killing form on $\Lg$.
	\end{introtheorem}
	Then, we prove a statement about the representation theory of $\cdo$.
	\begin{introtheorem}[Theorem \ref{theorem degeneration of the geometric Satake correspondence}, Remark \ref{remark degeneration geometric Satake torus}]
		\label{introtheorem degeneration satake}
		Assume that $G$ is a torus or a simple group. Then for almost every level $\kappa$, the category $\cdo\mathrm{-Mod}^{\jetinf G \times \jetinf G}$ is equivalent to the category of~$\C$-vector spaces.
	\end{introtheorem}
	As we will explain shortly, this result means that the geometric Satake equivalence does not deform well with respect to the level. To fix this defect, Gaitsgory and Lurie came up with what is now known as the fundamental local equivalence. To formulate it, introduce the principal equivariant $\mathcal{W}$-algebra on $G$: 
	$$
	\ewalg=H^{*}_{DS}(\cdo),
	$$
	where $H^{*}_{DS}$ is the (principal) Drinfeld--Sokolov reduction functor. We conjecture
	\begin{introconjecture}[Conjecture \ref{conjecture FLE}]
		\label{introconjecture FLE}
		For generic levels $\kappa$, there is an equivalence of abelian categories 
		$$
		\ewalg\mathrm{-Mod}^{\jetinf G} \simeq \check{G}\mathrm{-Mod}
		$$
		where $\check{G}$ is the Langlands dual group.
	\end{introconjecture} 
	Here is how we expect the equivalence to work. Fix a dominant cocharacter~$\gamma$, apply a spectral flow twist of parameter $\gamma$ on $\cdo$ and take its Drinfeld--Sokolov reduction: this is the sought-after $\ewalg$-module. We prove, without restrictions on the group, that this procedure always yields simple objects with the correct equivariance:
	\begin{introtheorem}[Theorem \ref{theorem simplicity of modules for equivariant w algebra} and Remark \ref{remark restriction to the simple case and generalization to reductive case}]
		\label{introtheorem construction of simple ewalg modules}
	For generic levels $\kappa$, the family of $\ewalg$-modules $H_{DS}^{*}(\gamma \cdot \mathcal{D}_{G}^{\kappa})$ consists of simple, pairwise non-isomorphic objects of $\ewalg\mathrm{-Mod}^{\jetinf G}$.
	\end{introtheorem}
	
	We propose a general approach to this conjecture and prove:
	\begin{introtheorem}[Theorem \ref{theorem FLE pour le tore} and Theorem \ref{theorem FLE for adjoint type ADE}]
		\label{introtheorem fle ad}
		Conjecture \ref{introconjecture FLE} holds when $G$ is a torus or when $G$ is a simple adjoint group of classical simply laced type.
	\end{introtheorem}
	
	Our proofs are purely vertex-algebraic; the quantum geometric Langlands program serves here as a motivation. We begin by situating our results within that program, before turning to our techniques in more detail.
		
	Let $T$ be a maximal torus of $G$ and $B$ be a Borel containing it. The study of the representation theory of $G$ naturally leads one to consider various geometric spaces built out of the group. An example is the flag variety $G/B$, a smooth algebraic variety of finite type. For instance, the study of (twisted) algebraic $\cD_{G/B}$-modules \textit{i.e.}, the theory of modules over the sheaf of (twisted) differential operators on $G/B$ leads to the theory of Beilinson--Bernstein localization~(\cite{BeilinsonBernstein1993}).
	
	The theory of $\cD$-modules on a space $X$ can be thought of as the study of systems of linear partial differential equations on $X$ (\cite{kashiwarabookdmod}, \cite{Borel1987}, \cite{HTT2008}). When solving these, one produces topological objects, typically local systems of finite dimensional vector spaces on $X$. It sometimes happens that a $\cD$-module and the associated topological object are equivalent, and this kind of phenomenon goes under the name of Riemann--Hilbert correspondence. A large class of $\cD$-modules for which that happens has been identified. These are the so-called regular holonomic $\cD$-modules and their topological counterparts are called perverse sheaves (\cite{Deligne1970}, \cite{Mebkhout1984}, \cite{Kashiwara1984}, \cite{GoreskyMacPherson1980}, \cite{BBDG1982}). They are also important because of their interactions with the sheaves-functions dictionary of Grothendieck (\cite{Deligne1977a}). 
	
	An example is the geometric Satake equivalence (\cite{Satake1963}, \cite{GinzburgSatake}, \cite{MirkovicSatake}, \cite{BaumannRiche}) which is a first step in the geometric version of the Langlands program, as envisioned by Drinfeld and Laumon in the 1980s (\cite{DrinfeldLanglands}, \cite{GerardLanglands}). In order to state it, let $\jetinf G$ be the jet space and $\mathcal{L}G$ be the loop space of the group $G$. These are algebraic groups such that $\jetinf G (\C) = G(\C[[t]])$ and $\mathcal{L}G(\C) = G(\C((t)))$. Define the affine Grassmannian~$\mathrm{Gr}_{G} = \mathcal{L}G/\jetinf G$ to be the quotient of the loop group of $G$ by its jet space, it is an ind-scheme of infinite type. The geometric Satake equivalence states that the category of $\jetinf G$-equivariant perverse sheaves on $\mathrm{Gr}_{G}$ is equivalent to the category of representations of $\check{G}$. 
	
	This raises the following question: is there a way to state the geometric Satake equivalence in terms of $\cD$-modules on the affine Grassmannian $\mathrm{Gr}_{G}$? To answer that, one may hope for a formalism of $\cD$-modules which is robust enough to make sense on~$\mathrm{Gr}_{G}$ as well as a suitable Riemann--Hilbert correspondence. One step in that direction is suggested by the work of Arkhipov and Gaitsgory (\cite{arkhipov_gaitsgory_differential_operators}) in which they build a chiral algebra whose category of modules is a substitute for a category of $\cD$-modules on the loop group $\mathcal{L}G$. These chiral algebras are known as chiral algebras of differential operators and were introduced by Beilinson and Drinfeld (\cite{beilinson2004chiral}). They were also independently discovered in the language of vertex algebras by Gorbounov, Malikov, and Schechtman and later called vertex algebras of chiral differential operators (\cite{MSV1999}, \cite{GMV}). In terms of vertex algebras, the construction of Arkhipov and Gaitsgory reads as follows: the category $\cdo\mathrm{-Mod}$ is an appropriate definition of the category $\cD_{\kappa}(\mathcal{L}G)\mathrm{-Mod}$ of $\kappa$-twisted $\cD$-modules on the loop group. In particular for a certain level $\kappa_{c}$, called the critical level, the category $\cD_{G}^{\kappa_{c}}\mathrm{-Mod}$ behaves as an appropriate category of untwisted $\cD$-modules on the loop group. 
	
	Before explaining the construction of $\cdo$, let us say a word about the classical setting of the theory of algebraic $\cD$-modules. Let $X$ be a smooth affine scheme of finite type with structure sheaf~$\cO_{X}$. It carries a sheaf $\cD_{X}$ of differential operators built out of $\cO_{X}$ and the tangent sheaf~$\Theta_{X}$. It is a sheaf of noncommutative rings that contains $\cO_{X}$ as a subring and $\Theta_{X}$ as a Lie subalgebra interacting via the Leibniz rule. The category of algebraic~$\cD_{X}$-module on $X$ is then the category of sheaves of, say left, modules over~$\cD_{X}$ that are quasicoherent as $\cO_{X}$-modules. The quasicoherence assumption ensures that a~$\cD_{X}$-module is controlled by its global sections, which is a left module over the noncommutative ring $\cD_{X}(X)$ of global differential operators on $X$. From now on, we use the same notation for a sheaf and its global sections. 
	
	A key example for us is the case where $X=G$ is the group itself. The theory of~$\mathcal{D}$-modules is then nothing but the representation theory of the ring~$\cD_{G}$. This ring has a nice description in terms of $G$ and its Lie algebra $\Lg$. In fact, we can think of $\Lg$ as global left-invariant vector fields of $G$ and they generate over $\cO_{G}$ the Lie algebra of global vector fields $\Theta_{G}$. As a vector space we have 
	$$
	\cD_{G} = \mathcal{U}(\Lg) \otimes \cO_{G}
	$$
	and the ring structure is determined by the Leibniz rule. It naturally contains~$\cU(\Lg)$ and~$\cO_{G}$ as subalgebras. Note that we made the choice to consider $\Lg$ as left-invariant global vector fields on $G$. Thinking of $\Lg$ as right-invariant global vector fields shows that $\cD_{G}$ contains another subalgebra which is a copy of~$\cU(\Lg)$. By construction, these two copies commute with one another.
	
	The construction of $\cdo$ is a variation on that of $\cD_{G}$ where $G$ is replaced by~$\mathcal{L}G$ and rings are replaced by vertex algebras. Introduce the affine Kac--Moody algebra $\Lgkappa = \Lg((t)) \oplus \C\mathbb{1}$ (\cite{Kac1990}). It is a central extension of~~$\Lie(\mathcal{L}G)=\Lg((t))$ given by the level~$\kappa$. The universal affine vertex algebra $V^{\kappa}(\Lg)$ (\cite{FrenkelZhuAffineVOA}) plays an analogous role to the universal enveloping algebra for $\Lgkappa$. The Lie subalgebra~$\Lg[[t]]=\Lie(\jetinf G)$ of $\Lgkappa$ acts on $\cO_{\jetinf G}$ as left-invariant vector fields and so by analogy we define, as vector spaces 
	$$
	\cdo = V^{\kappa}(\Lg) \otimes \cO_{\jetinf G}.
	$$
	One shows that there exists a vertex algebra structure on $\cdo$ such that $V^{\kappa}(\Lg)$ and $\cO_{\jetinf G}$ are vertex subalgebras of $\cdo$. Again, we chose to consider left-invariant vector fields on $\jetinf G$ over right-invariant ones. This time, corresponding to the latter, there is a copy of $V^{\kappa^{*}}(\Lg)$ in $\cdo$ but the level gets shifted (see Theorem \ref{theorem dual embedding of universal affine vertex algebra to cdo}). Again, these two copies commute in an appropriate sense. 
	
	This vertex algebra will play a role analogous to (the global sections) of the sheaf of ($\kappa$-twisted) differential operators on the loop group $\mathcal{L}G$. It is thus natural to define 
	$$
	\cD\mathrm{-Mod}_{\kappa}(\mathcal{L}G) := \cdo\mathrm{-Mod}.
	$$
	In particular, when $\kappa = \kappa_{c}$ is the critical level, then we get a definition of untwisted $\cD$-modules on the loop group 
	$$
	\cD\mathrm{-Mod}(\mathcal{L}G) := \cD^{\kappa_{c}}_{G}\mathrm{-Mod}.
	$$
	
	To take into account objects on the affine Grassmannian, and the equivariance of these objects themselves, one needs to further develop the formalism. In the classical setting where $X$ is a smooth scheme of finite type and is acted on by a connected subgroup $K$ of $G$, it is remarkable that the category of $K$-equivariant $\cD$-modules on $X$ is a full subcategory of the category of $\cD$-modules on $X$. That is, equivariance of $\cD$-modules is a property and not an additional structure (see for instance Remark~30.10 of~\cite{EtingofRepresentationsLieGroups2024}). Precisely, given a $\cD_{X}$-module $\mathcal{M}$, then $\mathcal{M}$ is $K$-equivariant if the induced action of $\Lie(K)$ integrates to an action of the group $K$ (see Proposition~30.10 of \textit{loc.cit.}). 
	
	
	Back to the infinite dimensional setting, let $K = \jetinf G \subset \mathcal{L}G$ act on $\mathcal{L}G$ by left multiplication. First, because $\mathrm{Gr}_{G}$ is a quotient, it is natural to define the category of $\cD$-modules on $\mathrm{Gr}_{G}$ to be the category of $K$-equivariant $\cD$-modules on the loop group. In terms of the vertex algebra $\cdo$, given a $\cdo$-module $M$ this translates to asking that when seen as a $V^{\kappa^{*}}(\Lg)$-module via restriction, the action of $\Lg[[t]]$ integrates to an action of $\jetinf G$. This remarkable subcategory of $V^{\kappa^{*}}(\Lg)\mathrm{-Mod}$ is known as the Kazhdan--Lusztig category~(\cite{kazhdan1991tensor}, \cite{kazhdan1993tensor}, \cite{kazhdan1994tensor}, \cite{kazhdan1994tensor4}), we denote it by $V^{\kappa^{*}}(\Lg)\mathrm{-Mod}^{\jetinf G}$. Let 
	$$
	\cdo\mathrm{-Mod}^{1 \times \jetinf G}
	$$
	be the full subcategory of $\cdo\mathrm{-Mod}$ consisting of objects satisfying this condition. That defines (see Theorem A.3 of \cite{arkhipov_gaitsgory_differential_operators}) the category of $\kappa$-twisted $\cD$-modules on the affine Grassmannian
	$$
	\cD\mathrm{-Mod}_{\kappa}(\mathrm{Gr}_{G}) := \cdo\mathrm{-Mod}^{1 \times \jetinf G}.
	$$
	In the same fashion, define 
	$$
	\cdo\mathrm{-Mod}^{\jetinf G \times \jetinf G}
	$$
	to be the full subcategory of modules consisting of objects that when seen by restriction as $V^{\kappa\oplus \kappa^{*}}(\Lg \oplus \Lg)$-modules belong to $V^{\kappa\oplus \kappa^{*}}(\Lg \oplus \Lg)\mathrm{-Mod}^{\jetinf G \times \jetinf G}$. Let
	$$
	\cD\mathrm{-Mod}_{\kappa}(\mathrm{Gr}_{G})^{\jetinf G} := \cdo\mathrm{-Mod}^{\jetinf G \times \jetinf G}
	$$
	be the category of $\kappa$-twisted $\jetinf G$-equivariant $\cD$-modules on the affine Grassmannian. In this language, it is reasonable to expect that the geometric Satake equivalence is something that happens at the critical level and takes the form of an equivalence of categories 
	$$
	\cD_{G}^{\kappa_{c}}\mathrm{-Mod}^{\jetinf G \times \jetinf G} \simeq \check{G}\mathrm{-Mod}
	$$ 
	where $\check{G}\mathrm{-Mod}$ is the category of representations of $\check{G}$. It is natural to ask if it is possible to describe the category 
	$$
	\cdo\mathrm{-Mod}^{\jetinf G \times \jetinf G}
	$$
	for levels $\kappa$ that are not critical. When moving away from the critical level, we enter the paradigm of the quantum geometric Langlands program (\cite{Stoyanovsky2006}, \cite{Gaitsgory2016}). This explains why Theorem \ref{introtheorem degeneration satake} can be interpreted as a quantum degeneration of the geometric Satake equivalence. 
	
	This deficiency of the geometric Satake equivalence was well-known to experts and a category of substitution was introduced by Gaitsgory following an idea of Lurie in~\cite{Gaitsgory2008} (see also \cite{CampbellDhillonRaskin2019}). It is the appropriately defined category of $\kappa$-twisted Whittaker $\cD$-modules on $\mathrm{Gr}_{G}$. We now introduce the category that should be a substitute in our language. Our formulation was motivated by the affine Skryabin equivalence proven by Raskin~(\cite{Raskin2021}). 
	Consider the (principal) $\mathcal{W}$-algebra of $\Lg$ at level $\kappa$ defined from the universal affine algebras, $W^{\kappa}(\Lg) := H^{*}_{DS}(V^{\kappa}(\Lg))$  (\cite{FeiginFrenkel1990}). It is a vertex algebra and we get a well-defined functor 
	$$
	H^{*}_{DS} : V^{\kappa}(\Lg)\mathrm{-Mod} \lra W^{\kappa}(\Lg)\mathrm{-Mod}.
	$$
	It also works in a relative setting, namely if $V$ is a vertex algebra equipped with a morphism 
	$$
	V^{\kappa}(\Lg) \lra V
	$$
	then $H^{*}_{DS}(V)$ is a vertex algebra that comes equipped with a morphism 
	$$
	W^{\kappa}(\Lg) \lra H^{*}_{DS}(V).
	$$
	Moreover we have a well-defined functor 
	$$
	H^{*}_{DS} : V\mathrm{-Mod} \lra H^{*}_{DS}(V)\mathrm{-Mod}.
	$$
	This construction applied to $\cdo$ with respect to the morphisms of vertex algebras 
	$$
	V^{\kappa}(\Lg) \lra V^{\kappa}(\Lg) \otimes V^{\kappa^{*}}(\Lg) \lra \cdo 
	$$
	defines the (principal) equivariant affine $\mathcal{W}$-algebra $\mathcal{W}_{G}^{\kappa}$ on the group $G$ at level $\kappa$~(\cite{Arakawa2018}). It comes equipped with morphisms of vertex algebras
	$$
	W^{\kappa}(\Lg) \lra W^{\kappa}(\Lg) \otimes V^{\kappa^{*}}(\Lg) \lra \mathcal{W}_{G}^{\kappa}.
	$$
	We define 
	$$
	\mathcal{W}_{G}^{\kappa}\mathrm{-Mod}^{ \jetinf G}
	$$
	to be the full subcategory of $\mathcal{W}_{G}^{\kappa}\mathrm{-Mod}$ consisting of modules that are integrable with respect to $\jetinf G$, when seen as $V^{\kappa^{*}}(\Lg)$-modules by restriction. This category should deform nicely and be related for generic $\kappa$ to the category of representations of the Langlands dual group $\check{G}$: this is our Conjecture \ref{introconjecture FLE}. To put it plainly, our conjecture is an abelian category formulation of the fundamental local equivalence of the quantum geometric Langlands program in the language of vertex algebras.


	The main idea of this work is that the chiral Peter--Weyl theorem (Theorem  \ref{theorem chiral Peter-Weyl}) allows us, for generic levels, to think of $\cdo$ as a lattice vertex algebra. In this analogy, the underlying lattice is the semigroup of dominant characters and Fock spaces are replaced by Weyl modules. 
	
	A guiding principle in the study of lattice vertex algebras is that their representation theory is controlled by the dual lattice. This principle applied here is what gives birth to the Langlands dual combinatorics. 
	
	To make this observation precise, we use in Section \ref{section spectral flow in the theory of vertex algebras} Li's delta operator~(\cite{Li1997}) to define what we call the spectral flow group and compute it in the case of $\cdo$ (Proposition \ref{proposition computation of the spectral flow group for the cdo}). Because this group fully controls the representation theory in the case of lattice vertex algebras (see \textit{e.g.}, Proposition~3.4 of \cite{Li1997}), our analogy dictates that it might be relevant in our situation. 
	
	In fact, in the case of an algebraic torus, $\cdo$ is isomorphic to what is called in the literature a half-lattice vertex algebra (see \textit{e.g.}, \cite{BermanDongTan2002}). In Section~\ref{section the case of an algebraic torus} we show how the standard techniques apply and prove Conjecture~\ref{introconjecture FLE} in this case. In fact, we even get a more precise statement (Corollary \ref{corollaire of theorem fle tore}) about how the geometric Satake equivalence behaves with respect to the level. 
	
	In general, the spectral flow group provides a large supply of simple modules on~$\cdo$. Then we need to go through the quantum Hamiltonian reduction procedure to produce meaningful objects. In Section \ref{section quantum Hamiltonian reduction and spectral flow} we explain the relation between this construction and the twisted quantum Hamiltonian reduction functor. By using the results of Arakawa and Frenkel (\cite{ArakawaFrenkel2019}) we are able to describe the objects we expect to appear in Conjecture~\ref{introconjecture FLE} and prove their simplicity: this is our Theorem \ref{introtheorem construction of simple ewalg modules}.
	
	Another guiding principle is that vertex algebras tend to be more rigid than usual algebras. This led us to prove a uniqueness statement on simple vertex algebras having the Peter--Weyl decomposition (see Theorems \ref{theorem rigidity of chiral Peter-Weyl} and \ref{proposition unicity of shifted cdos}). The proof uses the theory of spherical varieties (\cite{LunaVust1983}, \cite{Knop1991}) in an essential way. 
	
	This uniqueness statement is the key to prove properties of some shifted versions of $\cdo$ (called the quantum geometric Langlands kernel vertex operator algebras and introduced in \cite{CreutzigGaiotto2020}) and to study their behaviour under the quantized Moore--Tachikawa convolution operation (\cite{Arakawa2018}).   
	
	In particular when $G$ is a simple adjoint group of type $A$ or $D$, \textit{e.g.} $G = \mathrm{PSL}_{n}$, we explain in Section \ref{section the case of a simple simply laced group of adjoint type} how to use the convolution operation to prove Conjecture \ref{introconjecture FLE} (see~§\ref{section the strategy} for more details on the strategy). This is related to the Goddard--Kent--Olive coset realization of $\mathcal{W}$-algebras (see \cite{GoddardKentOlive1986} for $\mathfrak{sl}_{2}$, \cite{ArakawaCreuztigLinshaw2019} and \cite{CreutzigNakatsuka2023} for type ADE, \cite{creutziglinshawtrialities} for type A, \cite{creutziglinshawtrialitiesortho} for type D). 
	
	Finally, in Section \ref{section further perspectives} we discuss possible extensions of the ideas of this work and further research directions.
	
	\section*{Acknowledgments}
	This work is part of the author's PhD thesis conducted under the supervision of Anne Moreau. I am grateful for her help throughout the preparation of this work, her careful rereading of this article and for suggesting such a beautiful topic. 
	
	I would like to thank Tomoyuki Arakawa, Michel Brion, Gurbir Dhillon, Andrew Linshaw, Shigenori Nakatsuka and Olivier Schiffmann for useful discussions and comments. 
	
	I thank Sam Raskin for explaining certain aspects of his work \cite{Raskin2021}, especially why it is reasonable to expect statements such as our version of the fundamental local equivalence at the level of abelian categories.
	
	I also thank Dennis Gaitsgory for sharing many of his ideas and especially for pointing out the relevance of \cite{Gaitsgory2008} and bringing \cite{CampbellDhillonRaskin2019} to my attention. 
	
	I benefited immensely from discussions with Thomas Creutzig. He explained to me why Theorem \ref{theorem spectral flow drinfeld sokolov is the same as twisted up to cohomological shift} should be true and the proof strategy explained in §\ref{section the strategy} grew out of conversations we had in Orsay and Erlangen.  
	
	Finally, I cannot thank Gérard Laumon enough for kindly and patiently sharing his love of mathematics with me during these past six years in Orsay.

	\section{Conventions and notations}
	
	All constructions are over the field $\C$ of complex numbers and unlabeled tensor products are understood to be taken over $\C$.
	
	For us, algebraic groups are assumed to be affine and connected, we denote the multiplicative group by $\G_{m} = \Spec(\C[t,t^{-1}])$. By an algebraic torus we mean a direct product of finitely many copies of $\G_{m}$.
	
	If $\Lg$ is a Lie algebra and $M$ is a $\Lg$-module, we denote by $M^{\Lg}$ the invariant subspace consisting of elements of $M$ annihilated by every element of $\Lg$.
	
	If $V$ is a vector space then we denote by $\cF(V)$ the vector space of fields on $V$. It consists of series $a(z) \in \End(V)[[z,z^{-1}]]$ such that for all $v \in V$, $a(z)v \in V((z))$. If $V$ is a $\frac{1}{2}\Z$-graded vertex superalgebra, we denote by $V_{n}$ the $n$-th graded piece, for $n \in \frac{1}{2}\Z$.
	
	If $I$ is a set and $i,j \in I$ we denote by $\delta_{i,j}$ the Kronecker delta function equal to $1$ if $i=j$ and $0$ else. 
	
	\subsection{Reductive groups}
	\label{section reductive groups}
	Some good general references on algebraic groups are \cite{Borel1991}, \cite{SGA3-2011}, \cite{tauvel_yu_lie_algebras}, and \cite{Milne2017}. We denote by 
	$$
	G\mathrm{-Mod}
	$$
	the category of representations of $G$. Let $Z(G)$ be the center of $G$ and choose a maximal torus $T$, a Borel subgroup $B$ containing $T$, and denote by $N$ the unipotent radical of $B$. Let $N_{G}(T)$ be the normalizer of the maximal torus $T$ and $W = N_{G}(T)/T$
	be the Weyl group and $\omega_{0}$ be its longest element. 
	
	Let $\Lg = \Lie(G)$ be the Lie algebra of $G$, $\mathfrak{z}(\Lg) = \Lie(Z(G))$, $\Lh = \Lie(T)$ (it is a Cartan subalgebra of $\Lg$), $\Lb = \Lie(B)$, $\Ln = \Lie(N)$ the corresponding Lie algebras. Since $\Lg$ is reductive, it decomposes as the direct sum of its center $\Lzg$ and its derived Lie algebra~$[\Lg,\Lg]$. Write
	\begin{equation}
		\label{decomposition de Lg}
		\Lg = \Lzg \oplus [\Lg,\Lg] = \Lzg \oplus \bigoplus_{s=1}^{r} \Lg_{s},
	\end{equation}	
	with $\Lg_{s}$ simple.
	For each $s \in \{1,\dots,r\}$, let $\Lh_{s}:=\Lh \cap \Lg_{s}$, it is a Cartan subalgebra of $\Lg_{s}$ and we have the equality
	$
	\Lh = \Lzg \oplus \bigoplus_{s=1}^{r} \Lh_{s}.
	$
	Let 
	$
	X^{*}(T)
	$
	be the group of characters and~$X_{*}(T)$ 
	the group of cocharacters of the maximal torus $T$. They both are free abelian groups of rank equal to the dimension of $T$. Via the functor Lie we identify $X^{*}(T)$ (resp.~$X_{*}(T)$) with subgroups of $\Lh^{*}$ (resp. $\Lh$). These are full rank lattices in the sense that the following canonical maps are isomorphisms of vector spaces
	$$
	X_{*}(T) \otimes_{\mathbb{Z}} \C \lra \Lh, \hspace{0.5mm} X^{*}(T) \otimes_{\mathbb{Z}} \C \lra \Lh^{*}.
	$$
	Recall that each element of $X^{*}(\G_{m})$ is of the form $t\in \G_{m} \mapsto t^{n}\in \G_{m}$ for a unique $n\in \mathbb{Z}$. This yields a group isomorphism between $X^{*}(\G_{m})$ and $\mathbb{Z}$.
	If $\gamma \in X_{*}(T)$ is a cocharacter and $\alpha \in X^{*}(T)$ is a character, their composition $\alpha \circ \gamma$ is, via the previous identification, an element of $\mathbb{Z}$. That defines a $\mathbb{Z}$-bilinear map
	\begin{equation}
		\label{equation perfect pairing between characters and cocharacters}
		X^{*}(T) \otimes_{\mathbb{Z}} X_{*}(T) \lra \mathbb{Z}
	\end{equation}
	that happens to be a perfect pairing of free $\mathbb{Z}$-modules in a compatible way with the perfect pairing between $\Lh$ and $\Lh^{*}$.
	
	The group $G$ acts on its Lie algebra $\Lg$ via the so-called adjoint representation $\Ad : G \ra \GL(\Lg)$. By restricting this action to the maximal torus $T$ we obtain the decomposition $\Lg = \bigoplus_{\alpha \in X^{*}(T)} \Lg^{\alpha}$ where $\Lg^{\alpha} = \{x \in \Lg \hspace{0.5mm}|\hspace{0.5mm} \Ad(t)(x) = \alpha(t)x \textrm{ for all $t \in T$}\} $. Recall that a nonzero $\alpha \in X^{*}(T)$ such that $\Lg^{\alpha}$ is nonzero is called a root, they form a finite set denoted $\Phi(G,T)$. We have $\Lg^{0} = \Lh$ and we get the following root space decomposition of $\Lg$:
	\begin{equation}
		\label{equation root space decomposition of Lg}
		\Lg = \Lh \oplus \bigoplus_{\alpha \in \Phi(G,T)} \Lg^{\alpha}. 
	\end{equation}
	If we differentiate the adjoint representation of the group $G$ we obtain the adjoint representation of its Lie algebra $\ad : \Lg \ra \Lgl(\Lg)$ and, after restricting to the Cartan subalgebra, a similar story holds. We denote by $\Phi(\Lg,\Lh)$ the set of roots with respect to $\Lh$.
	
	Now, because $\Lh = \Lie(T)$, it is easy to check that $\Phi(G,T) = \Phi(\Lg,\Lh)$ under the inclusion $X^{*}(T) \subset \Lh$, we denote this set by~$\Phi$. The story for coroots is parallel, let $\Phi(G,T)^{\vee} \subset X_{*}(T)$ be the finite subset of coroots for the group and $\Phi(\Lg,\Lh)^{\vee} \subset \Lh$ be the finite subset of coroots for the Lie algebra. We have that $\Phi(G,T)^{\vee} = \Phi(\Lg,\Lh)^{\vee}$ under the inclusion $X_{*}(T) \subset \Lh$. We denote this set by $\Phi^{\vee}$. To each root is naturally associated a coroot, and we denote by $\alpha \in \Phi \mapsto \alpha^{\vee} \in \Phi^{\vee}$ this bijection.
	
	Let $Q$ (resp. $Q^{\vee}$) be the subgroup of $\Lh^{*}$ (resp. $\Lh$) generated by roots (resp. coroots). By the theory of semisimple Lie algebras, the $\C$-span of the roots (resp. coroots) is~$\bigoplus_{s=1}^{r}\Lh_{s}^{*}$ (resp. $\bigoplus_{s=1}^{r}\Lh_{s}$).
	
	The choice we made of a Borel $B$ containing $T$ defines a subset $\Phi_{+} \subset \Phi$ of positive roots by saying that $\alpha \in \Phi$ is positive if $\Lg^{\alpha} \subset \Lb$. The following definitions are important to classify the representations of the algebraic group~$G$. Let
	$$
	X^{*}(T)_{+} = \{ \lambda \in X^{*}(T) \hspace{0.5mm}|\hspace{0.5mm} \lambda(\alpha^{\vee})\geq0 \textrm{ for all $\alpha \in \Phi_{+}$}\}
	$$
	be the semigroup of dominant characters,
	$$
	X_{*}(T)_{+} = \{ \gamma \in X_{*}(T) \hspace{0.5mm}|\hspace{0.5mm} \gamma(\alpha)\geq0 \textrm{ for all $\alpha \in \Phi_{+}$}\}
	$$
	be the semigroup of dominant cocharacters, 
	$$
	P = \{ \lambda \in \Lh^{*} \hspace{0.5mm}|\hspace{0.5mm} \lambda(\alpha^{\vee}) \in \mathbb{Z} \textrm{ for all $\alpha \in \Phi$} \}, \\ 
	$$
	be the group of weights,
	$$
	P_{+} = \{ \lambda \in \Lh^{*} \hspace{0.5mm}|\hspace{0.5mm} \lambda(\alpha^{\vee}) \in \mathbb{Z}_{+} \textrm{ for all $\alpha \in \Phi_{+}$} \}. \\ 
	$$
	be the semigroup of dominant weights,
	$$
	\check{P} = \{ \gamma \in \Lh \hspace{0.5mm}|\hspace{0.5mm} \alpha(\gamma) \in \mathbb{Z} \textrm{ for all $\alpha \in \Phi$} \}, \\
	$$
	be the group of coweights,
	$$
	\check{P}_{+} = \{ \gamma \in \Lh \hspace{0.5mm}|\hspace{0.5mm} \alpha(\gamma) \in \mathbb{Z}_{+} \textrm{ for all $\alpha \in \Phi_{+}$} \}. 
	$$
	be the semigroup of dominant coweights.
	The two terminologies are compatible in the sense that
	\begin{align*}
		X^{*}(T)_{+} &= P_{+}\cap X^{*}(T), \\
		X_{*}(T)_{+} &= \check{P}_{+} \cap X_{*}(T).
	\end{align*}
	The following inclusions of abelian groups hold
	\begin{align*}
		Q &\subset X^{*}(T) \subset P \subset \Lh^{*}, \\
		Q^{\vee} &\subset X_{*}(T) \subset \check{P} \subset \Lh.  
	\end{align*}
	
	\begin{note}
		We denote a basis of indecomposable positive roots by $\alpha_{1},\dots,\alpha_{p} \in \Phi_{+}$. They are so that any positive (resp. negative) root can be written uniquely as a sum of the $\alpha_{i}$'s with nonnegative (resp. nonpositive) integer coefficients. In particular $\alpha_{1},\dots,\alpha_{p}$ is a $\C$-basis of $(\Lh^{ss})^{*}$. Passing to coroots we get $\alpha_{1}^{\vee},\dots,\alpha_{p}^{\vee} \subset \Phi^{\vee}$, and they form a $\C$-basis of $\Lh^{ss}$ inside of $\Lh$. We let $(\varpi_{i})_{1\leqslanti \leqslantp}$ (resp. $(\varpi_{i}^{\vee})_{1\leqslanti \leqslantp}$) denote the unique $\C$-basis of $(\Lh^{ss})^{*}$ (resp. $\Lh^{ss}$) such that
		$$
		\varpi_{i}(\alpha_{j}^{\vee}) = \delta_{i,j}  \textrm{ (resp. } \varpi^{\vee}_{i}(\alpha_{j}) = \delta_{i,j}) \textrm{ for all $1 \leqslanti,j \leqslantp$},
		$$
		the $(\varpi_{i})_{1 \leqslanti \leqslantp}$ (resp. $(\varpi_{i}^{\vee})_{1\leqslanti \leqslantp}$) are called the fundamental weights (resp. fundamental coweights) of the reductive Lie algebra $\Lg$.
		\begin{note}
			\textcolor{red}{Is it fair to call that fundamental weights of $\Lg$ and not of $\Lg^{ss}$?}
		\end{note}
		The following equalities hold
		\begin{align*}
			P &= \Lzg^{*} \oplus \bigoplus_{i=1}^{p} \mathbb{Z}\varpi_{i} \subset \Lh^{*}, 
			P_{+} =   \Lzg^{*} \oplus \bigoplus_{i=1}^{p} \mathbb{Z}_{+}\varpi_{i} \subset \Lh^{*}, \\
			\check{P} &=  \Lzg \oplus \bigoplus_{i=1}^{p} \mathbb{Z}\varpi^{\vee}_{i} \subset \Lh, 
			\check{P}_{+} = \Lzg \oplus \bigoplus_{i=1}^{p} \mathbb{Z}_{+}\varpi^{\vee}_{i} \subset \Lh.
		\end{align*}
	\end{note}
	
	For $\lambda \in X^{*}(T)_{+}$, denote by $V_{\lambda}$ a simple representation of $G$ or $\Lg$ of highest weight~$\lambda$. 
	\begin{theoreme}[Theorems 22.2 and 22.42 of \cite{Milne2017}]
		\label{theorem classification of simple G-modules and simple LieG modules}
		
		Any simple, hence finite dimensional, $G$-module is isomorphic to $V_{\lambda}$ for a unique $\lambda \in X^{*}(T)_{+}$. Moreover the category~$G\mathrm{-Mod}$ is semisimple.
		
		Any finite dimensional simple $\Lg$-module is isomorphic to $V_{\lambda}$ for a unique $\lambda \in P_{+}$. Moreover the category of locally finite $\Lg$-modules on which $\Lzg$ acts semisimply is semisimple.
	\end{theoreme}
	
	The following numerical criterion for integrability will be useful.
	
	\begin{corollaire}
		\label{corollary sufficient condition for integrability}
		Let $M$ be a locally finite $\Lg$-module on which the center $\Lzg$ of $\Lg$ acts semisimply. Then the action of $\Lg$ on $M$ integrates to a representation of the algebraic group $G$ if and only if the set of weights of $M$ with respect to $\Lh$ is contained in $X^{*}(T)$.
	\end{corollaire}
	
	\begin{note}
		In view of Theorem \ref{theorem classification of simple G-modules and simple LieG modules}, if $G$ is a reductive group then we have a bijection
		$$
		\lambda \in X^{*}(T)_{+} \longmapsto [V_{\lambda}] \in \mathrm{Irr}(G)
		$$
		where $[V_{\lambda}]$ is the isomorphism class of the irreducible module $V_{\lambda}$. For each $\lambda \in X^{*}(T)_{+}$ we can differentiate the module $V_{\lambda}$ to obtain a $\Lg$-module that we denote again $V_{\lambda}$. Because $G$ is connected, this module remains irreducible and is again of highest weight $\lambda$, where $\lambda$ is now seen in 
		$$
		X^{*}(T)_{+} \subset X^{*}(T) \subset P \subset \Lh^{*},
		$$
		as previously explained. 
	\end{note}

	\subsection{Affine vertex algebras}
	
	Some good general references on vertex algebras are \cite{Kac1998}, \cite{FrenkelBenZvi2004}, \cite{LepowskyLi1994}, and~\cite{arakawa_moreau_arc_spaces}. 
	
	Let $\kappa$ be a level, \textit{i.e.}, a bilinear symmetric form on $\Lg$ such that for all $x,y,z \in \Lg$ the following equality of complex numbers holds $\kappa([x,y],z) = \kappa(x,[y,z]).$ An important example of level is the Killing form, denoted by $\kappa_{\Lg}$.
	
	We now introduce the affine Kac--Moody algebra associated to our data. Let $\Lgkappa$ be the central extension of $\Lg((t))$ by $\C\mathbb{1}$ given by $\kappa$. By definition, $\Lgkappa$ fits in the exact sequence of Lie algebras 
	\begin{equation}
		\label{equation exact sequence defining the kac moody}
		0 \lra \mathbb{C}\mathbb{1} \lra \Lgkappa \lra \Lg((t)) \lra 0
	\end{equation}
	where $\mathbb{1}$ is central in $\Lgkappa$, and for all $x,y \in \Lg$, $m,n \in \mathbb{Z}$, we have
	$$
	[xt^{m},yt^{n}] = [x,y]t^{m+n} + m\delta_{m+n,0}\kappa(x,y)\mathbb{1}.
	$$
	Note that the exact sequence \eqref{equation exact sequence defining the kac moody} splits over the Lie subalgebra $\Lg[[t]] \subset \Lg((t))$, hence $\Lg[[t]] \oplus \C \mathbb{1}$ is a Lie subalgebra of $\Lgkappa$.
	
	A $\Lgkappa$-module $M$ is said to be smooth if for all $m \in M$ there exists $n\geqslant0$ such that $t^{n}\Lg[[t]]m = 0$ and $\mathbb{1}$ acts as the identity. We denote by 
	$
	\Lgkappa\mathrm{-Mod}
	$
	the category of smooth~$\Lgkappa$-modules. 
	
	Let $V^{\kappa}(\Lg)$ be the universal affine vertex algebra associated with $\Lg$ at level $\kappa$. It is generated as a vertex algebra by elements of $\Lg$.
	To describe the relations between the generators, we will make use of the $\lambda$-bracket notation (see \textit{e.g.} §2.3 of \cite{Kac1998}). For all $x,y \in \Lg$, we have 
	\begin{equation}
		\label{equation lambda bracket for the universal affine vertex algebra at level kappa}
		[x_{\lambda}y] = [x,y] + \kappa(x,y)\lambda. 
	\end{equation}
	The categories $\Lgkappa-\textrm{Mod}$ and $V^{\kappa}(\Lg)-\textrm{Mod}$ are isomorphic (see \textit{e.g.} Theorem~6.2.13 of~\cite{LepowskyLi1994}). 
	
	There is a natural induction procedure that makes a $\Lg$-module $V$ into a $\Lgkappa$-module. We first see $V$ as a $\Lg[[t]] \oplus \C\mathbb{1}$-module by having $\mathbb{1}$ act as the identity and $\Lg[[t]]$ act through the projection $\Lg[[t]] \ra \Lg$ that maps $t$ to zero. We then induce the resulting module to $\Lgkappa$. This defines a functor
	$$
	\IndKacMoody : \Lg\mathrm{-Mod} \lra \Lgkappa\mathrm{-Mod}. 
	$$
	If $\lambda \in P_{+}$ and $V_{\lambda}$ denotes the highest weight $\Lg$-module of weight $\lambda$, let $|\lambda \rangle \in V_{\lambda}$ be a highest weight vector, we call the induced module
	\begin{equation}
		\label{equation definition weyl modules}
		V_{\lambda}^{\kappa} = \IndKacMoody V_{\lambda} 
	\end{equation}
	the Weyl module of highest weight $\lambda$. They satisfy the following universal property:
	\begin{proposition}
		\label{proposition Weyl modules are highest weight modules}
		Let $\lambda \in P_{+}$, $M$ be a $V^{\kappa}(\Lg)$-module and $m \in M$ be a highest weight vector for $\Lg$ of weight $\lambda$ such that ~$t\Lg[[t]]m = 0$ and $
		\dim(\cU(\Lg)m)< \infty.
		$
		Then there exists a unique morphism of $V^{\kappa}(\Lg)$-modules 
		$
		V^{\kappa}_{\lambda} \lra M
		$
		mapping $|\lambda \rangle $ to $m$.
	\end{proposition}
	\begin{proof}
		This follows from the definition of Weyl modules as induced modules and the fact that $\Hom_{\Lg}(V_{\lambda},M)$ is naturally in bijection with the highest weight vectors for $\Lg$ of weight $\lambda$ in $M$ that span a finite dimensional $\Lg$-module.
	\end{proof}
	
	\begin{definition}
		\label{definition Kazhdan Lusztig category}
		We denote by
		$$
		\Lg\mathrm{-Mod}^{G}
		$$
		the full subcategory of $\Lg$-Mod consisting of locally finite $\Lg$-modules for which the action integrates to a representation of the algebraic group $G$. 
		Define the Kazhdan--Lusztig category 
		$$
		V^{\kappa}(\Lg)\mathrm{-Mod}^{\jetinf G}
		$$ 
		to be the full subcategory of the category of $V^{\kappa}(\Lg)$-modules consisting of objects~$M$ such that
		\begin{itemize}
			\item as a $\Lg$-module, $M$ is locally finite and the action of $\Lg$ integrates to a representation of the algebraic group $G$ \textit{i.e.}, $M \in \Lg\mathrm{-Mod}^{G}$,
			\item the Lie subalgebra $t\Lg[[t]]$ of $\Lgkappa$ acts locally nilpotently.
		\end{itemize}
	\end{definition}
	
	\begin{remarque}
		\label{remark the meaning of the definition of the Kazhdan Lusztig category}
		Because $G$ is connected the functor $\Lie$ defines an isomorphism of categories
		$$
		\mathrm{Lie} : G\mathrm{-Mod} \simeq \Lg\mathrm{-Mod}^{G}
		$$
		which we use to identify them from now on.
		When the group $G$ is a connected, simple and simply connected group, say for instance $G = \mathrm{SL}_{2}$, then for any locally finite $\Lg$-module the action automatically integrates to a representation of the algebraic group~$G$. 
		As alluded to in the introduction, the Kazhdan--Lusztig category $\Lgkappacategoryintegrable$ consists precisely of $\Lgkappa$-modules such that the action of $\Lg[[t]]$ integrates to an algebraic representation of $\jetinf G$. To make sense of this statement and prove it, the formalism of pro-algebraic groups and their pro-Lie algebras is useful (see \textit{e.g.} §4.4 of~\cite{Kumar2002}).
	\end{remarque}	
	\begin{note}
		To see why it is true on the arrows, I think this is nothing but $\Hom_{\Lg}(M,M') = \Hom(M,M')^{\Lg} = \Hom(M,M')^{G} = \Hom_{G}(M,M')$
	\end{note}

	\section{Chiral differential operators on an algebraic group}
	
	\subsection{General results}
	\label{section general results}
	There are essentially two ways of thinking of $\cO_{G}$, the coordinate ring of $G$, as a $G$-module. For $\alpha \in G$, $f \in \cO_{G}$ and $g \in G$ set 
	\begin{align*}
		\lambda_{\alpha}(f)(g) &= f(\alpha^{-1}g), \\
		\rho_{\alpha}(f)(g) &= f(g\alpha).
	\end{align*}
	That defines two group morphisms
	\begin{align*}
		\lambda : \alpha \in G \longmapsto \lambda_{\alpha} \in \GL(\cO_{G}), \\
		\rho : \alpha \in G \longmapsto \rho_{\alpha} \in \GL(\cO_{G}).
	\end{align*} 
	They each make $\cO_{G}$ into a $G$-module and commute with one another so they make $\cO_{G}$ into a $G \times G$-module. We will refer to this $G\times G$-module structure on $\cO_{G}$ as the regular representation of the group $G$. Explicitly, if $\alpha,\beta \in G$ and $f \in \cO_{G}$ then define $(\alpha,\beta) \cdot f \in \cO_{G}$ to take on each $g \in G$ the value 
	
	$$
	((\alpha,\beta) \cdot f)(g) = f(\alpha^{-1}g\beta).
	$$
	
	As a vector space $\Lg$ is the tangent space of $G$ at $e$, the identity element. Recall that it consists of the linear maps $x : \cO_{G} \lra \C$ satisfying for all $f,g \in \cO_{G}$ the equality of complex numbers
	$$
	x(fg) = f(e)x(g) + x(f)g(e).
	$$
	Let $\Der(\cO_{G})$ denote the Lie algebra of global vector fields on the group $G$ \textit{i.e.}, the derivations of the algebra $\cO_{G}$. We say that $X \in \Der(\cO_{G})$ is a left-invariant derivation if for all $\alpha \in G$, we have the equality 
	$
	\lambda_{\alpha} \circ X = X \circ \lambda_{\alpha}.
	$
	They form a Lie subalgebra of $\Der(\cO_{G})$ that we denote by $\Der(\cO_{G})^{\lambda}.$
	We define in a similar fashion right-invariant derivations and denote their set by $\Der(\cO_{G})^{\rho}$. Any derivation $X \in \Der(\cO_{G})$ can be made into an element of the tangent space at the identity so we have two linear maps
	$$
	\mathrm{ev}_{e} \circ \cdot: \Der(\cO_{G})^{\lambda} \overset{}{\lra} T_{e}(G) \longleftarrow \Der(\cO_{G})^{\rho} : -\mathrm{ev}_{e} \circ \cdot
	$$
	where $\mathrm{ev}_{e} : f \in \cO_{G} \mapsto f(e) \in \C$ is the evaluation at the identity. Both these maps admit inverses that are defined for $x \in \Lg$ by  	
	\begin{align}
		\label{equation definition x_{L}}
		x_{L} &= (\id \otimes x) \circ \Delta \in \Der(\cO_{G})^{\lambda},\\
		\label{equation definition x_{R}}
		x_{R} &=(-x \otimes \id) \circ \Delta \in \Der(\cO_{G})^{\rho},
	\end{align}
	where $\Delta : \cO_{G} \lra \cO_{G} \otimes \cO_{G}$ is the comultiplication map. More explicitly, we have for all $x \in \Lg$, $f \in \cO_{G}$ and $g \in G$ the equalities of complex numbers
	\begin{align}
		x_{L}(f)(g) &= x(\lambda_{g^{-1}}f), \\
		x_{R}(f)(g) &= -x(\rho_{g}f).
	\end{align}
	Recall that by definition the Lie algebra structure of $\Lg$ is such that for all $x,y \in \Lg$ the following equality holds in $\Der(\cO_{G})$
	\begin{equation}
		\label{equation x_{L} is a morphism of Lie algebras}
		[x_{L},y_{L}] = [x,y]_{L}.
	\end{equation}
	Because we have put a minus sign in formula $\eqref{equation definition x_{R}}$ we have the equality in $\Der(\cO_{G})$
	\begin{equation}
		\label{equation x_{R} is a morphism of Lie algebras}
		[x_{R},y_{R}] = [x,y]_{R}.
	\end{equation}
	This shows that we have defined two Lie algebra maps 
	\begin{align*}
		&x \in \Lg \longmapsto x_{L} \in \Der(\cO_{G}) \subset \Lgl(\cO_{G}), \\
		&x \in \Lg \longmapsto x_{R} \in \Der(\cO_{G}) \subset \Lgl(\cO_{G}),
	\end{align*}
	whose images commute, that is for all $x,y \in \Lg$ we have in $\Der(\cO_{G})$ the equality
	\begin{equation}
		\label{equation left and right embeddings commute}
		[x_{L},y_{R}] = 0.
	\end{equation}
	This endows $\cO_{G}$ with the structure of a $\Lg \oplus \Lg$-module, explicitly, for all $x,y \in \Lg$ and~$f \in \cO_{G}$ define 
	$$
	(x,y)\cdot f = x_{L}(y_{R}(f))=y_{R}(x_{L}(f)).
	$$ 
	This is, up to exchanging the factors, the differential of the regular representation, since for all $x,y \in \Lg$ and $f \in \cO_{G}$ we have 
	\begin{align}
		\mathrm{d}\lambda(y)(f) = y_{R}(f), \\
		\mathrm{d}\rho(x)(f) = x_{L}(f),
	\end{align}
	where $\mathrm{d}\lambda : \Lg \lra \Lgl(\cO_{G})$ and $\mathrm{d}\rho : \Lg \lra \Lgl(\cO_{G})$ denote the differential at the identity of the morphism of groups $\lambda$ and $\rho$ (see \textit{e.g.} §23.4 of \cite{tauvel_yu_lie_algebras}).

	\begin{note}
		
		\begin{example}
			Let $G= \G_{m}$ be a one-dimensional torus. Consider the element 
			$$
			x : P \in \C[z,z^{-1}] \mapsto \partial_{z}(P)(1) \in \C.
			$$
			It easily seen to define an element of the Lie algebra $\Lg$ of $G$. We can compute 
			$$
			x_{L}(z) = (\id \circ x)(z_{1}z_{2}) = z_{1}x(z_{2}) = z_{1},
			$$
			and
			$$
			-x_{R}(z) = (x \circ \id)(z_{1}z_{2}) = x(z_{1})z_{2} = z_{2},
			$$
			so 
			$$
			x_{L} = -x_{R} = z\partial_{z},
			$$
			this is easily seen to hold for any $x \in \Lg$.
			
		\end{example }

		\begin{lemme}
			Let $V \subset H$ a finite dimensional subspace of $H$ such that for all $\alpha \in G$, $\rho_{\alpha}(V) \subset V$ and $\rho : G \lra \GL(V)$ be the induced morphism of algebraic groups. Then for all $x \in T_{e}(G)$ and $f \in V$ we have 
			\begin{equation}
				\textrm{d}\rho(x)(f) = x_{L}(f).
			\end{equation}
			If $V$ is a finite dimensional subspace of $H$ such that for all $\alpha \in G$, $\lambda_{\alpha}(V) \subset V$ and $\lambda : G \lra \GL(V)$ is the induced morphism of algebraic groups. Then for all $x \in T_{e}(G)$ and $f \in V$ we have 
			\begin{equation}
				\textrm{d}\lambda(x)(f) = x_{R}(f).
			\end{equation}
		\end{lemme}
		\begin{proof}
			
			Quote 23.4.11 and 23.4.3. from Tauvel-Yu's algebraic groups book.
			
			Let's try to give a convincing proof.
			\begin{enumerate}
				\item Choose a basis to identify $V$ with $\C^{d}$,
				\item Under this choice of basis, identify $T_{e}(\GL_{d})$ with $\Lgl_{d}$,
				\item Compute the differential.
			\end{enumerate}
			I checked on my own 
			
		\end{proof}
		
	\end{note}
	
	The coordinate ring $\cO_{\jetinf G}$ of the algebraic group $\jetinf G$ is a differential commutative algebra \textit{i.e.}, a commutative vertex algebra. Its Lie algebra is $\Lg[[t]]$ and the bracket is given for all $x,y \in \Lg$ and $m,n \geqslant0$ by 
	$$
	[xt^{m},yt^{n}] = [x,y]t^{m+n}.
	$$
	By the same construction as before, there are two Lie algebra maps from $\Lg[[t]]$ to $\Der(\cO_{\jetinf G})$ corresponding to left-invariant and right-invariant vector fields. In particular $\cO_{\jetinf G}$ is a $\Lg[[t]] \oplus \Lg[[t]]$-module. Let us view elements of $\Lg[[t]]$ as left-invariant vector fields of $\jetinf G$. This makes $\cO_{\jetinf G}$ into a $\Lg[[t]]$-module. By having $\mathbb{1}$ act as the identity we make it into a $\Lg[[t]] \oplus \C \mathbb{1}$-module. 
	
	Define the vertex algebra of chiral differential operators on the algebraic group $G$ to be the induced module 
	\begin{equation}
		\label{equation definition of cdo as an induced module}
		\cD_{G}^{\kappa} = \Ind_{\Lg[[t]] \oplus \C \mathbb{1}}^{\Lgkappa} \cO_{\jetinf G} = \cU(\Lgkappa) \otimes_{\cU(\Lg[[t]]) \otimes \C[\mathbb{1}]} \cO_{\jetinf G}.
	\end{equation}
	Note that, as vector spaces, 
	$$
	\cdo = V^{\kappa}(\Lg) \otimes \cO_{\jetinf G}.
	$$
	Since both $V^{\kappa}(\Lg)$ and $\cO_{\jetinf G}$ are $\Z_{\geqslant0}$-graded vector spaces, then so is $\cD_{G}^{\kappa}$ by setting 
	$$
	(\cdo)_{n} = \bigoplus_{k + k' = n} (V^{\kappa}(\Lg))_{k} \otimes (\cO_{\jetinf G})_{k'}
	$$
	for each $n \in \Z_{\geqslant0}$.

	According to \cite{arkhipov_gaitsgory_differential_operators} and \cite{Gorbounov2001} (see also §3.4 of \cite{arakawa_moreau_arc_spaces}) there exists a unique vertex algebra structure on $\cdo$ such that the injective maps 
	\begin{align*}
		&\pi_{L} : x \in V^{\kappa}(\Lg) \longmapsto x \otimes 1 \in \cdo \\
		&\iota : f \in \cO_{\jetinf G} \longmapsto 1 \otimes f \in \cdo 
	\end{align*}
	are vertex algebra morphisms satisfying for all $x \in \Lg$ and $f \in \cO_{G}$ the following Leibniz rule 
	$$
	[\pi_{L}(x)_\lambda \iota(f)] = \iota(x_{L}(f)).
	$$
	As it is harmless we will forget the $\iota$ in the notations so that the Leibniz rule reads 
	$$
	[\pi_{L}(x)_{\lambda}f] = x_{L}(f).
	$$
	To summarize, $\cdo$ is generated as a vertex algebra by elements of the form $\pi_{L}(x)$ for $x \in \Lg$ and regular functions on the group $G$ in such a way that for all $x,y \in \Lg$ and~$f,g \in \cO_{G}$ we have
	\begin{align}
		&[\pi_{L}(x)_{\lambda}\pi_{L}(y)] = \pi_{L}([x,y]) + \kappa(x,y)\lambda, \label{equation commutation relation for left-invariant vector field in the cdo lambda bracket}\\
		&[\pi_{L}(x)_{\lambda}f] = x_{L}(f), \label{equation Leibniz rule in the cdo lambda bracket} \\
		&[f_{\lambda}g] = 0. \label{equation cO(G) is commutative inside of the cdo}
	\end{align}
	
	\begin{note}
		This readily implies, by skewcommutativity, 
		\begin{equation}
			[f_{\lambda}\pi_{L}(x)] = - [\pi_{L}(x)_{\lambda}f] = -x_{L}(f).
		\end{equation}
	\end{note}
	
	Any element of $\cdo$ can be written as a linear combination of elements of the form
	\begin{equation}
		\label{equation generators of the cdo as vector space}
		\pi_{L}(x^{i_{1}})_{(-m_{1})}...\pi_{L}(x^{i_{r}})_{(-m_{r})}f^{j_{1}}_{(-n_{1})}... f^{j_{s}}_{(-n_{s})}\vac    
	\end{equation}
	where $r,s \in \Z_{\geqslant0}, m_{1},...,m_{r},n_{1},...,n_{s}>0, x^{i_{1}},...,x^{i_{r}} \in \Lg, f^{j_{1}},...,f^{j_{s}} \in \cO_{G}$. 
	
	\begin{remarque}
		Note that since $\cO_{G}$ is of finite type, $\cdo$ is strongly finitely generated. Moreover, if $G$ is simple, then $\kappa$ can be identified with a complex number. Thinking of it as a formal parameter leads to an interesting point of view on $\cdo$ as coming from a deformable family of vertex algebras (see §3.1 of \cite{creutziglinshawcoset} for the terminology). Notice that all the relations \eqref{equation commutation relation for left-invariant vector field in the cdo lambda bracket}-\eqref{equation cO(G) is commutative inside of the cdo} between the generators of $\cdo$ are polynomial functions of the variable $\kappa$. In fact, this is unsurprising as one can define by generators and relations a certain $1$-parameter vertex algebra $\cdo$ over the commutative ring $\C[\kappa]$ and recover $\mathcal{D}_{G}^{\kappa_{0}}$ at any given level $\kappa_{0}$ by specialization.
	\end{remarque}
	
	We can describe the degree $1$ part of $\cdo$, as vector spaces we have
	$$
	(\cdo)_{1} =\cO_{G} \otimes \Lg \oplus (\cO_{\jetinf G})_{1}.
	$$
	\begin{note}
		\begin{align*}
			(\cdo)_{1} &= (V^{\kappa}(\Lg))_{1} \otimes (\cO_{\jetinf G})_{0} \oplus (V^{\kappa}(\Lg))_{0} \otimes (\cO_{\jetinf G})_{1} \\
			&= (\cO_{\jetinf G})_{0} \otimes (V^{\kappa}(\Lg))_{1} \oplus (V^{\kappa}(\Lg))_{0} \otimes (\cO_{\jetinf G})_{1} \\
			&= \cO_{G} \otimes \Lg \oplus (\cO_{\jetinf G})_{1}.
		\end{align*}	
	\end{note}
	Let $\Omega^{1}_{G}$ be the $\cO_{G}$-module of global differential $1$-forms on the algebraic group $G$. The maps
	$$
	f \otimes x \in \cO_{G} \otimes \Lg \longmapsto fx_{L} \in \Der(\cO_{G}),
	$$
	and 
	$$
	f \partial(g) \in (\cO_{\jetinf G})_{1} \longmapsto f \mathrm{d}g \in \Omega^{1}_{G}
	$$
	are isomorphisms of $\cO_{G}$-modules. So, in fact, as vector spaces
	$$
	(\cdo)_{1} = \Der(\cO_{G}) \oplus \Omega^{1}_{G}.
	$$

	\begin{note}
		
		By straightforward computations we have
		\begin{lemme}
			\label{lemma computation in degree 1 for the cdo}
			Let $f,g,h \in \cO_{G}$ and $x,y \in \Lg$ then 
			\begin{multline}
				[f\pi_{L}(x)_{\lambda}g\pi_{L}(y)] = fx_{L}(g)\pi_{L}(y)+fg\pi_{L}([x,y]) -gy_{L}(f)\pi_{L}(x) \\
				-g\partial(x_{L}(y_{L}(f))) \\
				+\lambda(fg\kappa(x,y)-fy_{L}(x_{L}(g))-x_{L}(g)y_{L}(f) - g(x_{L}(y_{L}(f))),
			\end{multline}
			and 
			\begin{equation}
				[f\pi_{L}(x)_{\lambda}g\partial(h)] = fx_{L}(g)\partial(h) +\partial(f)gx_{L}(h) + fg\partial(x_{L}(h)) + \lambda fgx_{L}(h).
			\end{equation}
			If $D$ is a derivation of $\cO_{G}$ and $\omega,\omega'$ are $1$-forms then we have 
			\begin{align*}
				[D_{\lambda}f] &= D(f) \\
				[D_{\lambda}\omega] &= \mathrm{Lie}(D)(\omega) + \omega(D)\lambda  \\
				[\omega_{\lambda}\omega'] &= [\omega_{\lambda}f] = [f_{\lambda}f'] = 0,		
			\end{align*}
			where $\mathrm{Lie}(D) : \Omega^{1}_{G} \lra \Omega^{1}_{G}$ is the Lie derivative. 
		\end{lemme}
		
	\end{note}
	
	\begin{note}	
		We would like to compute for all $x,y \in \Lg$ and $f,g,h \in \cO_{G}$
		$$
		[f\pi_{L}(x)_{\lambda}g\pi_{L}(y)]\textrm{ and }[f\pi_{L}(x)_{\lambda}g \partial h ].
		$$
		Of course because of the commutativity of $\cO_{\jetinf G}$ we have that $1$-forms lambda-commute with each other. We have easily 
		\begin{equation}
			[f\pi_{L}(x)_{\lambda}g] = fx_{L}(g).
		\end{equation}
		These ones take a bit longer to compute
		\begin{multline}
			[f\pi_{L}(x)_{\lambda}g\pi_{L}(y)] = fx_{L}(g)\pi_{L}(y)+fg\pi_{L}([x,y]) -gy_{L}(f)\pi_{L}(x) \\
			-g\partial(x_{L}(y_{L}(f))) \\
			+\lambda(fg\kappa(x,y)-fy_{L}(x_{L}(g))-x_{L}(g)y_{L}(f) - g(x_{L}(y_{L}(f))),
		\end{multline}
		and 
		\begin{equation}
			[f\pi_{L}(x)_{\lambda}g\partial(h)] = fx_{L}(g)\partial(h) +\partial(f)gx_{L}(h) + fg\partial(x_{L}(h)) + \lambda fgx_{L}(h).
		\end{equation}
	\end{note}
	
	In \cite{ArakawaMoreau2021}, Arakawa and Moreau introduced a very convenient criterion to show that certain vertex algebras are simple in terms of the geometry of their associated scheme. One can show that the associated scheme of $\cdo$ is the cotangent bundle $T^{*}G$ of $G$. Because this is a smooth, reduced and symplectic scheme, we have:
	\begin{proposition}[Discussion following Corollary 9.3 of \textit{loc.cit.}]
		\label{proposition cdo is a simple vertex algebra}
		For any level $\kappa$, the vertex algebra~$\cdo$ is simple. 
	\end{proposition}
	
	As mentioned in the introduction, the right action chiralizes, but the level for which it does develops an anomaly. 
	\begin{theoreme}
		\label{theorem dual embedding of universal affine vertex algebra to cdo}
		There exists a unique level $\kappa^{*}$ such that there exists a $\Z$-graded injective vertex algebra map
		$$
		\pi_{R} : V^{\kappa^{*}}(\Lg) \lra \cdo 
		$$ 
		satisfying for all $x,y\in \Lg$ and $f\in \cO_{G}$
		\begin{align}
			&[\pi_{R}(x)_{\lambda}\pi_{L}(y)] = 0, \label{equation commutation pi l and pi r} \\
			&[\pi_{R}(x)_{\lambda}f] = x_{R}(f).
		\end{align}
		The level $\kappa^{*}$ is related to $\kappa$ via the formula 
		$$
		\kappa^{*} = -\kappa_{\Lg} - \kappa,
		$$
		in particular 
		$$
		(\kappa^{*})^{*} = \kappa.
		$$
		Explicitly, let $(x^{i})_{1 \leqslant i \leqslant \dim(\Lg)}$ be a basis of $\Lg$ and $(\omega^{i})_{1 \leqslant i \leqslant \dim(\Lg)}$ the $\cO_{G}$-basis of $\Omega^{1}_{G}$ dual to the $\cO_{G}$-basis of $\Der(\cO_{G})$ given by the $(x^{i}_{L})_{1 \leqslant i \leqslant\dim(\Lg)}$. Then, for all $1 \leqslant i \leqslant\dim(\Lg)$, letting  
		$$
		x^{i}_{R} = \sum_{p} f^{i}_{p} x^{p}_{L},
		$$
		where $(f^{i}_{p}) \in \GL(\cO_{G})$, we have the formula  
		\begin{equation}
			\label{equation formula for the dual right embedding}
			\pi_{R}(x^{i}) = x^{i}_{R} + \sum_{q}\left(\sum_{p} \kappa^{*}(x^{p},x^{q})f^{i}_{p}\right)\omega^{q}.
		\end{equation}
		
	\end{theoreme}

	\begin{proof}
		To see why formula \eqref{equation formula for the dual right embedding} satisfies all the requirements, see §4 of \cite{Gorbounov2001}, Theorem 4.4 of \cite{arkhipov_gaitsgory_differential_operators} or Theorem 3.3 of \cite{arakawa_moreau_arc_spaces}. The converse is a straightforward but tedious computation (see Theorem 4.1.4 of \cite{my_phd}).
	\end{proof}
	\begin{remarque}
		This beautiful fact is attributed to Feigin, Frenkel and Gaitsgory in~\cite{Gorbounov2001}. To the best of our knowledge, its earliest appearance in the case of special linear groups can be found in \cite{FeiginParkhomenko1996}.
		
		We also note that we do not need the group $G$ to be reductive for all the results of this paragraph to hold. This assumption will be essential from now on.
	\end{remarque}
	
	As we have just seen, $\cdo$ comes equipped with two vertex algebra morphisms
	$$
	\pi_{L} : V^{\kappa}(\Lg) \lra \cdo \longleftarrow V^{\kappa^{*}}(\Lg) : \pi_{R}
	$$
	whose images commute (see equation \eqref{equation commutation pi l and pi r}). Equivalently this defines a morphism of vertex algebras 
	\begin{equation}
		\label{equation definition pil otimes pir}
		\pi_{L} \otimes \pi_{R} : V^{\kappa \oplus \kappa^{*}}(\Lg \oplus \Lg) \lra \cdo.
	\end{equation}
	
	\begin{definition}
		\label{definition critical level}
		We call critical level the unique level $\kappa_{c}$ on $\Lg$ such that 
		$$
		\kappa_{c}^{*} = \kappa_{c}.
		$$
		We have the equality 
		$$
		\kappa_{c} = -\frac{1}{2} \kappa_{\Lg} 
		$$
		where $\kappa_{\Lg}$ is the Killing form of $\Lg$. When $\Lg$ is simple we have the equality 
		$$
		\kappa_{c} = - \check{h} \widetilde{\kappa_{\Lg}}
		$$
		where $\check{h}$ is the dual Coxeter number and $\widetilde{\kappa_{\Lg}} = \frac{1}{2\check{h}}\kappa_{\Lg}$ is the normalized Killing form.
	\end{definition}

	\subsection{The conformal structure}
	\label{section The conformal structure}
	
	We make a brief digression to discuss the conformal structure of affine vertex algebras and $\cdo$. For the convenience of the reader we split the discussion between the abelian and the simple case. The reader familiar with the Segal--Sugawara construction in these contexts may go directly to the main statement (Theorem \ref{theorem conformal structure of cdo reductive}).
	
	\subsubsection{The case of an abelian Lie algebra}
	\label{section Abelian case}
	Assume that $G = T = \G_{m}^{r}$ is an algebraic torus of rank $r$ with Lie algebra $\Lg = \Lh = \C^{r}$. Let $\kappa$ be a symmetric bilinear form on $\Lh$. Note that such a form is automatically $\Lh$-invariant because $\Lh$ is abelian. Recall that, because the Killing form is zero, we have
	$$
	\kappa^{*} = -\kappa.	
	$$
	
	Assume that $\kappa$ is nondegenerate. Let $(h^{i})_{1 \leqslant i \leqslant r}$ be an orthonormal basis of $\Lh$ with respect to $\kappa$ \textit{i.e.}, a basis such that for all $1 \leqslant i,j \leqslant r$ we have the equality of complex numbers 
	$$
	\kappa(h^{i},h^{j}) = \delta_{i,j}.
	$$
	In that case, the vertex algebra $V^{\kappa}(\Lh)$ (resp. $V^{\kappa^{*}}(\Lh)$) is conformal (see Theorem~6.2.18 of \cite{LepowskyLi1994}) and a choice of conformal vector is given by  
	$$
	T_{L} = \frac{1}{2} \sum_{i=1}^{r} h^{i}_{(-1)}h^{i}_{(-1)}\vac \in V^{\kappa}(\Lh), \textrm{ (resp. $T_{R} = -\frac{1}{2} \sum_{i=1}^{r}h^{i}_{(-1)}h^{i}_{(-1)}\vac \in V^{\kappa^{*}}(\Lh)$)}.
	$$
	Let 
	\begin{equation}
		\label{equation definition conformal vector cdo torus}
		T = \pi_{L}(T_{L}) + \pi_{R}(T_{R})
	\end{equation}
	It is clear that $T$ has conformal weight $2$ and the goal is to show that it is a conformal vector of $\cdo$.
	
	Define $(\omega^{i})_{1 \leqslant i \leqslant r}$ to be the $\cO_{G}$-basis of $\Omega^{1}_{G}$ dual to the $\cO_{G}$-basis $(h^{i}_{L})_{1 \leqslant i \leqslant r}$ of~$\Der(\cO_{G})$. Then from equation \eqref{equation formula for the dual right embedding} we have for all $i \in \{1,...,r\}$ the equality 
	$$
	\pi_{R}(h^{i}) = - \pi_{L}(h^{i}) + \omega^{i} .
	$$
	We have 
	\begin{align*}
		\pi_{R}(T_{R}) &= -\frac{1}{2} \sum_{i=1}^{r} (-\pi_{L}(h^{i}) + \omega^{i})_{(-1)}(-\pi_{L}(h^{i})+\omega^{i})_{(-1)}\vac.
	\end{align*}
	So it is straightforward that 
	$$
	\pi_{R}(T_{R}) = -\frac{1}{2}\sum_{i=1}^{r} \pi_{L}(h^{i})\pi_{L}(h^{i}) + \frac{1}{2} \sum_{i=1}^{r} (\pi_{L}(h^{i})\omega^{i} + \omega^{i}\pi_{L}(h^{i})) - \frac{1}{2} \sum_{i=1}^{r} \omega^{i}\omega^{i}.
	$$
	For all $i \in \{1,...,r\}$ we have the equality $
	[\pi_{L}(h^{i})_{(-1)},\omega^{i}_{(-1)}] = 0
	$ in $\End(\cdo)$. So finally, we have the equality 
	\begin{equation}
		\label{equation of the conformal vector of cdo in the abelian case}
		T = \sum_{i=1}^{r} \pi_{L}(h^{i})\omega^{i} - \frac{1}{2}\sum_{i=1}^{r}\omega^{i}\omega^{i}. 
	\end{equation}
	
	\begin{lemme}
		\label{lemma conformal structure for the cdo of the torus}
		The vector $T$ is a conformal vector of central charge $2r$ of the vertex algebra $\cdo$. 
	\end{lemme}
	\begin{proof} 
		We have to check that $T_{(0)}$ is the translation operator. Let $f \in \cO_{G}$, then in view of equation \eqref{equation of the conformal vector of cdo in the abelian case} we have 
		\begin{align*}
			T_{(0)}f = \sum_{i=1}^{r} (\pi_{L}(h^{i})_{(-1)}\omega^{i})_{(0)}f 
			= \sum_{i=1}^{r} \pi_{L}(h^{i})_{(0)}(f)\omega^{i} 
			= \sum_{i=1}^{r} h^{i}_{L}(f)\omega^{i}
			= \partial(f).
		\end{align*}
		The other properties are easily checked.
	\end{proof}

	\begin{note}
		Assume that $\kappa$ is nondegenerate. Let $(h^{i})_{i=1}^{r}$ and $(h_{i})_{i=1}^{r}$ be two dual basis of $\Lh$ in duality with respect to $\kappa$, that is for all $1 \leqslanti,j \leqslantr$ we have the equality of complex numbers
		$$
		\kappa(h^{i},h_{j}) = \delta_{i,j}.
		$$  
		In that case the vertex algebras $V^{\kappa}(\Lh)$ and $V^{\kappa^{*}}$ are conformal and a choice of conformal vector is given by the Segal--Sugawara construction by setting  
		$$
		T_{L} = \frac{1}{2} \sum_{i=1}^{r} h^{i}_{(-1)} {h_{i}}_{(-1)} \vac \in (V^{\kappa}(\Lh))_{2}
		$$
		and 
		$$
		T_{R} = -\frac{1}{2} \sum_{i=1}^{r}h^{i}_{(-1)}{h_{i}}_{(-1)}\vac \in (V^{\kappa^{*}}(\Lh))_{2}.
		$$
		Define $(\omega^{i})_{i=1}^{r}$ (resp. $(\omega_{i})_{i=1}^{r}$)
	\end{note}

	\subsubsection{The case of a simple Lie algebra}
	\label{section Simple case}
	
	Assume that $G$ is a simple algebraic group, \textit{i.e.}, a reductive algebraic group whose Lie algebra $\Lg$ is simple. Because $\Lg$ is simple there exists a unique complex number $k \in \C$ such that 
	$$
	\kappa = k \widetilde{\kappa}_{\Lg}
	$$
	and we refer interchangeably to $\kappa$ and $k$ as the level in that case. Here $\kappa^{*}$ corresponds to the complex number 
	\begin{equation}
		\label{equation dual level as a complex number for a simple factor}
		k^{*} = -k - 2\check{h}.
	\end{equation}
	Let us assume that the level is non-critical \textit{i.e.},
	$
	k \neq -\check{h}.$ 
	In that case we clearly also have $k^{*} \neq -\check{h}$. Let $(x^{i})_{i=1}^{\dim(\Lg)}$ be an orthonormal basis of $\Lg$ with respect to the normalized Killing form $\widetilde{\kappa}_{\Lg}$ \textit{i.e.}, a basis such that for all $i,j \in\{1,...,\dim(\Lg)\}$ we have the equality of complex numbers 
	$$
	\widetilde{\kappa_{\Lg}}(x^{i},x^{j}) = \delta_{i,j}.
	$$
	
	In that case the vertex algebra $V^{\kappa}(\Lg)$ (resp. $V^{\kappa^{*}}(\Lg)$) is conformal (see \textit{e.g.} Theorem~6.2.18 of \cite{LepowskyLi1994}) and a choice of conformal vector is given by the Segal--Sugawara construction
	\begin{align*}
		T_{L} &= \frac{1}{2(k + \check{h})} \sum_{i=1}^{\dim(\Lg)} x^{i}_{(-1)}x^{i}_{(-1)}\vac \in V^{\kappa}(\Lg) \\
		T_{R} &= -\frac{1}{2(k+\check{h})} \sum_{i=1}^{\dim(\Lg)}x^{i}_{(-1)}x^{i}_{(-1)}\vac \in V^{\kappa^{*}}(\Lg)
	\end{align*}
	because we have the equality of complex numbers 
	$$
	\frac{1}{k + \check{h}}= - \frac{1}{k^{*} + \check{h}}.
	$$
	Let 
	$$
	T = \pi_{L}(T_{L}) + \pi_{R}(T_{R})
	$$
	It is clear that $T$ is of conformal degree $2$.
	
	Define $(\omega^{i})_{1 \leqslant i \leqslant \dim(\Lg)}$ to be the $\cO_{G}$-basis of $\Omega^{1}_{G}$ dual to the $\cO_{G}$-basis $(x^{i}_{L})_{1 \leqslant i \leqslant \dim(\Lg)}$ of $\Der(\cO_{G})$. It follows from Proposition 3.18 of \cite{zhu_vertex_operator_algebras} that we have the equality in $\cdo$ 
	\begin{equation}
		\label{equation for T in the simple case}
		T = \sum_{i=1}^{\dim(\Lg)} \pi_{L}(x^{i})\omega^{i} + \frac{k^{*}}{2}\sum_{i=1}^{\dim(\Lg)} \omega^{i}\omega^{i}.
	\end{equation}
	and from there it is proven that $T$ is a conformal vector of central charge~$2\dim(G)$ of the vertex algebra $\cdo$ (see Proposition 3.22 of \textit{loc.cit.}).

	\subsubsection{The case of a reductive Lie algebra}
	\label{section Reductive case}
	Assume that $G$ is a reductive algebraic group. Recall the decomposition \eqref{decomposition de Lg} of $\Lg$. Because $\kappa$ is an invariant form it respects the previous decomposition, that is 
	$$
	\kappa = \kappa_{|\Lzg} \oplus \bigoplus_{s=1}^{r} \kappa_{|\Lg_{s}}.
	$$
	For each $s \in \{1,...,r\}$, because $\Lg_{s}$ is simple and $\kappa_{|\Lg_{s}}$ is invariant, there exists $k_{s} \in \C$ such that 
	$$
	\kappa_{|\Lg_{s}} = k_{s} \widetilde{\kappa_{\Lg_{s}}}.
	$$
	
	Let us assume momentarily that $\kappa$ is such that $\kappa-\kappa_{c}$ is nondegenerate where $\kappa_{c}$ is the critical level (see Definition \ref{definition critical level}). This means that $\kappa$ is such that $\kappa_{|\Lzg}$ is nondegenerate and $k_{s} \neq -\check{h}_{s}$ for all $1 \leqslant s \leqslant r$ where $\check{h}_{s}$ is the dual Coxeter number of the simple factor $\Lg_{s}$.
	
	We will apply the constructions of §\ref{section Abelian case} to $\Lzg$ and of §\ref{section Simple case} to each simple factor of $\Lg$ to construct a conformal vector in the general reductive case.
	
	Fix an orthonormal basis $(h^{i})_{i=1}^{\dim(\Lzg)}$ of the center $\Lzg$ with respect to $\kappa_{|\Lzg}$ and set 
	$$
	T_{L}^{ab} = \frac{1}{2} \sum_{i=1}^{\dim(\Lzg)} h^{i}_{(-1)}h^{i}_{(-1)}\vac \in V^{\kappa}(\Lzg)_{2},T_{R}^{ab} = -\frac{1}{2} \sum_{i=1}^{\dim(\Lzg)}h^{i}_{(-1)}h^{i}_{(-1)}\vac \in V^{\kappa^{*}}(\Lzg)_{2}.
	$$
	For each $s \in \{1,...,r\}$, fix an orthonormal basis $(x^{i_{s}}_{s})_{i_{s} =1}^{\dim(\Lg_{s})}$ of $\Lg_{s}$ with respect to the normalized Killing form $\widetilde{\kappa}_{\Lg_{s}}$ and set 
	\begin{align*}
		T_{L}^{s} &= \frac{1}{2(k_{s} + \check{h}_{s})} \sum_{i_{s}=1}^{\dim(\Lg_{s})} {x^{i_{s}}_{s}}_{(-1)}{x^{i_{s}}_{s}}_{(-1)}\vac \in V^{\kappa}(\Lg)_{2},\\
		T_{R}^{s} &= -\frac{1}{2(k_{s}+\check{h}_{s})} \sum_{i_{s}=1}^{\dim(\Lg_{s})}{x^{i_{s}}_{s}}_{(-1)}{x^{i_{s}}_{s}}_{(-1)}\vac \in V^{\kappa^{*}}(\Lg)_{2}.
	\end{align*}
	We have an isomorphism of vertex algebras 
	$$
	V^{\kappa}(\Lg) \simeq V^{\kappa_{\Lzg}}(\Lzg) \otimes \bigotimes_{s=1}^{r}V^{\kappa_{|\Lg_{s}}}(\Lg_{s}).
	$$
	So it makes sense to define 
	\begin{align*}
		T_{L} &= T^{ab}_{L} + T^{1}_{L} + ... + T^{r}_{L} \in V^{\kappa}(\Lg),\\
		T_{R} &= T^{ab}_{R} + T^{1}_{R} + ... + T^{r}_{R} \in V^{\kappa^{*}}(\Lg).
	\end{align*}
	Set 
	$$
	T = \pi_{L}(T_{L}) + \pi_{R}(T_{R}) \in \cdo.
	$$
	Clearly $T$ is of conformal degree $2$ and 
	$
	(h^{i})_{i=1}^{\dim \Lzg} \cup \bigcup_{s=1}^{r} (x^{i_{s}}_{s})_{i_{s} =1}^{\dim \Lg_{s}}
	$
	is a basis of $\Lg$ and hence 
	$
	(h^{i}_{L})_{i=1}^{\dim\Lzg} \cup \bigcup_{s=1}^{r} ((x^{i_{s}}_{s})_{L})_{i_{s} =1}^{\dim \Lg_{s}}
	$
	is a $\cO_{G}$-basis of $\Der(\cO_{G})$. Denote by~$
	(\omega^{i})_{i=1}^{\dim \Lzg} \cup \bigcup_{s=1}^{r} (\omega^{i_{s}}_{s})_{i_{s} =1}^{\dim \Lg_{s}}
	$
	its dual basis in $\Omega^{1}_{G}$.
	
	As a straightforward consequence of equalities \eqref{equation of the conformal vector of cdo in the abelian case} and \eqref{equation for T in the simple case} we have the equality:
	\begin{equation}
		\label{equation for T in the reductive case}
		T = \sum_{i=1}^{r} \pi_{L}(h^{i})\omega^{i} + \sum_{s=1}^{r} \sum_{i_{s}=1}^{\dim \Lg_{s}} \pi_{L}(x^{i_{s}}_{s})\omega^{i_{s}}_{s} - \frac{1}{2} \sum_{i=1}^{\dim \Lzg} \omega^{i}\omega^{i} + \sum_{s=1}^{r} \frac{k_{s}^{*}}{2} \sum_{i_{s}=1}^{\dim \Lg_{s}} \omega^{i_{s}}_{s} \omega^{i_{s}}_{s}.
	\end{equation}
	
	\begin{theoreme}
		\label{theorem conformal structure of cdo reductive}
		The vector $T$ is a conformal vector of central charge $2\dim G$ of the vertex algebra $\cdo$. Equivalently, the vertex algebra $\cdo$ is a conformal extension of~$V^{\kappa}(\Lg) \otimes V^{\kappa^{*}}(\Lg)$ with respect to the embedding 
		$$
		\pi_{L} \otimes \pi_{R} : V^{\kappa}(\Lg) \otimes V^{\kappa^{*}}(\Lg) \lra \cdo.
		$$
	\end{theoreme}
	\begin{proof}
		Equality \eqref{equation for T in the reductive case} shows that $T_{(0)}$ is the translation operator in a similar fashion as in the proof of Lemma \ref{lemma conformal structure for the cdo of the torus}. The rest is easily checked.
	\end{proof}
	
	\begin{remarque}
		Note that here we assume a condition on the level $\kappa$ such that the Segal--Sugawara constructions are available. It is remarkable that equation~\eqref{equation for T in the reductive case} does not involve the inverse of $(k+\check{h})$. The formula therefore makes sense at any level and one can check that $\cdo$ is indeed always conformal of central charge $2\dim(G)$ (see Proposition~3.22 of \cite{zhu_vertex_operator_algebras}).  	
	\end{remarque}

	\subsection{Chiral Peter--Weyl decomposition}
	
	The map $\pi_{L} \otimes \pi_{R}$ makes $\cdo$ into a $V^{\kappa}(\Lg) \otimes V^{\kappa^{*}}(\Lg)$-module. When the level is generic we can give a very explicit description of this structure. Let us first explain what we mean by a generic level.
	
	\begin{definition}
		\label{definition generic level}
		We say that $\kappa$ is generic if $\kappa_{|\Lzg}$ is nondegenerate and $k_{s} \notin \mathbb{Q}$ for all $s \in \{1,...,r\}$, where $k_{s} \in \C$ is the unique complex number such that 
		$
		\kappa_{|\Lg_{s}} = k_{s} \widetilde{\kappa_{\Lg_{s}}}.
		$
	\end{definition}
	\begin{remarque}
		\label{remark irrational level are stable by dual and direct sums}
		Note that if $\kappa$ is generic then so is $\kappa^{*}$ and that the direct sum of two generic levels is again generic.  
	\end{remarque}
		\begin{center}
		\textrm{From now on we assume that $\kappa$ is a generic level.}
	\end{center}
	The upshot is that for generic levels the Kazhdan--Lusztig category is easy to describe. The following statement can be found in the literature for the case of a simple group (see Lemma 30.1 of \cite{kazhdan1994tensor4}) or of an abelian Lie algebra (see §3 of \cite{LepowskyWilson1982}). We are not aware of a reference stating it at the level of generality we need, so we refer to §4.3 of \cite{my_phd} for a proof:
	
	\begin{theoreme}
		\label{theorem Kazhdan--Lusztig categories for irrational levels}
		For all $\lambda \in P_{+}$, the Weyl module $V_{\lambda}^{\kappa}$ is an irreducible $V^{\kappa}(\Lg)$-module. Moreover the functors 
		$$
		M \in \Lgkappacategoryintegrable \longmapsto M^{t\Lg[[t]]} \in \Lg\mathrm{-Mod}^{G}
		$$
		and 
		$$
		V \in \Lg\mathrm{-Mod}^{G} \longmapsto \IndKacMoody V \in \Lgkappacategoryintegrable
		$$
		are quasi-inverses of one another. 
	\end{theoreme}

	The following statement can be found in §4.4 of \cite{arkhipov_gaitsgory_differential_operators}, we give a more detailed proof for the reader's convenience: 
	
	\begin{theoreme} [Chiral Peter--Weyl]
		\label{theorem chiral Peter-Weyl}
		For any level $\kappa$ we have
		$$
		\cdo \in V^{\kappa \oplus \kappa^{*}}(\Lg \oplus \Lg)\mathrm{-Mod}^{\jetinf G \times \jetinf G}.
		$$
		Moreover, if $\kappa$ is generic, then $\cdo$ decomposes as a $V^{\kappa}(\Lg)\otimes V^{\kappa^{*}}(\Lg)$-module as the direct sum of Weyl modules
		\begin{equation}
			\label{equation chiral Peter-Weyl decomposition}
			\cdo = \bigoplus_{\lambda \in X^{*}(T)_{+}} V_{\lambda}^{\kappa} \otimes V_{-\omega_{0}\lambda}^{\kappa^{*}},
		\end{equation}
		where $\omega_{0}$ is the longest element of the Weyl group.
	\end{theoreme}
	\begin{proof}
		We start by showing that 
		$$
		\cdo \in V^{\kappa \oplus \kappa^{*}}(\Lg \oplus \Lg)\mathrm{-Mod}^{\jetinf G \times \jetinf G}.
		$$
		Because the conformal degrees of $\cdo$ are lower bounded, it is clear that the Lie algebra~$t\Lg[[t]] \oplus t\Lg[[t]]$ acts locally nilpotently. It remains to show that the action of the Lie algebra $\Lg \oplus \Lg$ is locally finite and integrates to an action of the group~$G$. Recall the generating subset \eqref{equation generators of the cdo as vector space} and choose an element $x \in \cdo$ of this form
		$$
		x = \pi_{L}(x^{i_{1}})_{(-m_{1})}...\pi_{L}(x^{i_{r}})_{(-m_{r})}f^{j_{1}}_{(-n_{1})}... f^{j_{s}}_{(-n_{s})}\vac.
		$$
		Because $\cO_{G}$ is locally finite as a $\Lg \oplus \Lg$-module there exists $V \subset \cO_{G}$ that is of finite dimension, contains $f^{j_{1}},...,f^{j_{s}}$ and is  $\Lg \oplus \Lg$-stable. It is easily checked that the vector space 
		$$
		\pi_{L}(\Lg)_{(-m_{1})}...\pi_{L}(\Lg)_{(-m_{r})}V_{(-n_{1})}...V_{(-n_{s})}\vac 
		$$
		is finite dimensional, contains $x$ and is $\Lg \oplus \Lg$-stable. Now to see that the action of $\Lg \oplus \Lg$ integrates to an action of the algebraic group $G \times G$, one checks that Corollary \ref{corollary sufficient condition for integrability} applies. This proves the first part of the statement. 
		
		Let us now assume that $\kappa$ is generic. Because the two actions of $\Lgkappa$ and $\Lgkappastar$ commute, it is clear that 
		$$
		(\cdo)^{t\Lg[[t]] \oplus t\Lg[[t]]} = ((\cdo)^{0 \oplus t\Lg[[t]]})^{t\Lg[[t]] \oplus 0}.
		$$
		Using the description of $\cdo$ as an induced module (recall equation \eqref{equation definition of cdo as an induced module}) and the equality of $\Lg$-modules $\cO_{\jetinf G}^{t\Lg[[t]]} = \cO_{G}$ we see that 
		$$
		(\cdo)^{0 \oplus t\Lg[[t]]} = V^{\kappa}(\Lg)\otimes \cO_{G}.
		$$ 
		But since  
		$$
		(V^{\kappa}(\Lg)\otimes \cO_{G})^{t\Lg[[t]] \oplus 0} = \cO_{G}
		$$
		we have shown that 
		$$
		(\cdo)^{t\Lg[[t]] \oplus t\Lg[[t]]} = \cO_{G}
		$$
		as a $\Lg \oplus \Lg$-module. 
		
		Finally the description of the Kazhdan--Lusztig category of Theorem \ref{theorem Kazhdan--Lusztig categories for irrational levels} together with the algebraic Peter--Weyl theorem (see \textit{e.g.} §27.3.9 of \cite{tauvel_yu_lie_algebras}) proves equality~\eqref{equation chiral Peter-Weyl decomposition}. 	
	\end{proof}
	\begin{remarque}
		In the case of a simple group, this was also proven by Minxian Zhu in \cite{zhu_vertex_operator_algebras} (see also \cite{FrenkelStyrkas2006} where $\mathcal{D}_{SL_{2}}^{\kappa}$ was introduced and studied under the name of modified regular representation and §5  of \cite{FeiginSurveyExtensions}). A more general statement studying the situation for any level can be found in \cite{Zhu2011}. 
		
		For generic levels it is easy to deduce the simplicity of $\cdo$ from the chiral Peter--Weyl theorem (see Remark 4.4.2 of \cite{my_phd}). But note that Proposition \ref{proposition cdo is a simple vertex algebra} holds for any level. 
	\end{remarque}
	
	\begin{theoreme}
		\label{theorem rigidity of chiral Peter-Weyl}
		Let $G$ be a reductive group and $\kappa$ be a generic level. Then there exists on the $V^{\kappa}(\Lg)\otimes V^{\kappa^{*}}(\Lg)$-module 
		$$
		\bigoplus_{\lambda \in X^{*}(T)_{+}} V_{\lambda}^{\kappa} \otimes V_{-\omega_{0}\lambda}^{\kappa^{*}}
		$$
		a unique structure of simple vertex algebra that is a conformal extension of $V^{\kappa}(\Lg)\otimes V^{\kappa^{*}}(\Lg)$ given by that of $\cdo$. 
	\end{theoreme}
	
	\begin{proof}
		It follows from Proposition~\ref{proposition cdo is a simple vertex algebra} and Theorem \ref{theorem chiral Peter-Weyl} that $\cdo$ satisfies the conditions of the theorem. Conversely, let 
		$$
		\mathcal{A} = \bigoplus_{\lambda \in X^{*}(T)_{+}} V_{\lambda}^{\kappa} \otimes V_{-\omega_{0}\lambda}^{\kappa^{*}}
		$$
		be a vertex algebra satisfying the hypothesis of the theorem. Because the conformal grading is enforced, one sees that $\mathcal{A}$ is $\mathbb{Z}_{\geq 0}$-graded and that 
		\begin{equation}
			\label{equation peter weyl for A}
			\mathcal{A}_{0} = \bigoplus_{\lambda \in X^{*}(T)_{+}} V_{\lambda} \otimes V_{-\omega_{0}\lambda}.
		\end{equation}
		Moreover it follows from the Borcherds identities that the pair $(\mathcal{A}_{0},-_{(-1)}-)$ is an associative and commutative algebra with unit. For all $x \in \Lg \subset V^{\kappa}(\Lg)$, $y \in \Lg \subset V^{\kappa^{*}}(\Lg)$ and~$f,g \in \mathcal{A}_{0}$ it follows from the Borcherds identities that
		\begin{align}
			x_{(0)}(f_{(-1)}g) &= (x_{(0)}f)_{(-1)}g + f_{(-1)}(x_{(0)}g), \label{equation x is a derivation}\\
			y_{(0)}(f_{(-1)}g) &= (y_{(0)}f)_{(-1)}g + f_{(-1)}(y_{(0)}g).\label{equation y is a derivation}
		\end{align}
		Assume that we can show that $\mathcal{A}_{0} = \cO_{G}$ as commutative $G \times G$-algebras. Then because of the generators of $\cdo$ and the relations between them (recall equations \eqref{equation commutation relation for left-invariant vector field in the cdo lambda bracket}, \eqref{equation Leibniz rule in the cdo lambda bracket}, and \eqref{equation cO(G) is commutative inside of the cdo}) we have a nonzero morphism of vertex algebras $\cdo \lra \mathcal{A}$. This morphism has to be injective by the simplicity of $\cdo$ (Proposition \ref{proposition cdo is a simple vertex algebra}) and hence is surjective because of the chiral Peter--Weyl decomposition of $\cdo$ (equation \eqref{equation chiral Peter-Weyl decomposition}).
		
		We are left to show that $(\mathcal{A}_{0},-_{(-1)}-)$ is isomorphic to $\cO_{G}$ as $G\times G$-commutative algebras. Introduce $\Zhu(\mathcal{A})$, the Zhu algebra of $\mathcal{A}$ (see §2 of \cite{Zhu1996}). Because $\mathcal{A}_{0}$ is the top space of the simple $\mathcal{A}$-module $\mathcal{A}$ then $\mathcal{A}_{0}$ is a simple $\Zhu(\mathcal{A})$-module (see Theorem 2.2.2 of \textit{loc.cit.}). Given a homogeneous element $a \in \mathcal{A}$, its image in $\Zhu(\mathcal{A})$ acts by $a_{(\deg(a)-1)}$ on $\mathcal{A}_{0}$ (see Theorem~2.1.2 of \textit{loc.cit.}). By induction on the conformal degree of elements in $\mathcal{A}$ we see that the image of the map
		$$
		\Zhu(\mathcal{A}) \lra \End(\mathcal{A}_{0})
		$$
		is generated as an associative algebra by elements of the form
		$$
		x_{(0)}, y_{(0)}, f_{(-1)},
		$$
		with $x \in \Lg \subset V^{\kappa}(\Lg)$, $y \in \Lg \subset V^{\kappa^{*}}(\Lg)$, $f \in \mathcal{A}_{0}$. Denote this subalgebra by $D$, so that $\mathcal{A}_{0}$ is a simple $D$-module. Equalities \eqref{equation x is a derivation} and \eqref{equation y is a derivation} mean that the map
		$$
		-_{(-1)}- : f \otimes g \in \mathcal{A}_{0} \otimes \mathcal{A}_{0} \longmapsto f_{(-1)}g\in \mathcal{A}_{0}
		$$
		is $\Lg \oplus \Lg$-equivariant. But the action of $\Lg \oplus \Lg$ on $\mathcal{A}_{0}$ integrates to an action of~$G\times G$. As such, the map
		$$
		-_{(-1)}- : f \otimes g \in \mathcal{A}_{0} \otimes \mathcal{A}_{0} \longmapsto f_{(-1)}g\in \mathcal{A}_{0}
		$$
		is $G \times G$-equivariant. It is then clear from the definition of $D$ that a $D$-stable subset of $\mathcal{A}_{0}$ is the same as an ideal that is moreover a $\Lg \oplus \Lg$-submodule which is the same as a $G \times G$-submodule (recall that $G$ is connected). This implies that $\mathcal{A}_{0}$ is reduced because the nilradical of $\mathcal{A}_{0}$ is obviously an ideal and is clearly stable under $G \times G$ by the $G \times G$-equivariance of the product. Hence the nilradical of $\mathcal{A}_{0}$ is $D$-stable and simplicity enforces it to be $0$ because $1 \neq 0 \in \mathcal{A}_{0}$ is not nilpotent. We now show that~$\mathcal{A}_{0}$ is an integral domain. Because we already know it is reduced, all we have to show is that $X = \Spec(\mathcal{A}_{0})$ is irreducible. By what we explained earlier, $G \times G$ acts on the affine scheme $X$. Consider an irreducible component of $X$. It is closed by definition. Because $G \times G$ is connected, any irreducible component is $G \times G$-stable. This implies that the ideal defining the component is stable under $G \times G$. Hence, it is the zero ideal by simplicity. This proves that $X$ is reduced and irreducible and, hence, $\mathcal{A}_{0}$ is an integral domain. 
		
		We now show that $\mathcal{A}_{0}$ is of finite type. Recall the decomposition of $\mathcal{A}_{0}$ as $\Lg \oplus \Lg$-module: 
		$$
		\mathcal{A}_{0} = \bigoplus_{\lambda \in X^{*}(T)_{+}} V_{\lambda} \otimes V_{-\omega_{0}\lambda}.
		$$
		Because $\mathcal{A}_{0}$ is integral, the product of two highest weight vectors is nonzero. In particular, the product of two highest weight vectors of weight $\lambda$ and $\mu$ is a highest weight vector of weight $\lambda + \mu$. Now pick elements of $X^{*}(T)_{+}$, say $\lambda_{1},...,\lambda_{r} \in X^{*}(T)_{+}$ such that any element of $X^{*}(T)_{+}$ can be written as a linear combination of these elements with non-negative integral coefficients. Then the finite dimensional space 
		$$
		\bigoplus_{i=1}^{r} V_{\lambda_{i}}\otimes V_{-\omega_{0}\lambda_{i}}
		$$ 
		spans $\mathcal{A}_{0}$ as an algebra, hence $\mathcal{A}_{0}$ is of finite type. 
		
		Because $X$ is a $G \times G$-scheme of finite type it contains a closed orbit. The ideal defining this orbit is $G \times G$-stable. Hence by simplicity this ideal is zero. So $X$ is a~$G\times G$-homogeneous variety and therefore $X$ is smooth, because the characteristic is zero.
		
		We now use the theory of spherical varieties to conclude. The fact that $X$ is spherical follows for instance from the Lemma 2.12 of \cite{Brion2010}. While there are many $G \times G$-commutative algebras having the decomposition 
		$$
		\bigoplus_{\lambda \in X^{*}(T)_{+}} V_{\lambda} \otimes V_{-\omega_{0}\lambda}
		$$
		only one may be smooth as a consequence of Knop's conjectures that were proven by Losev (Theorem 1.3 of \cite{Losev2009}). This proves that $\mathcal{A}_{0} = \cO_{G}$ as a $G \times G$-commutative algebra and finishes the proof.
	\end{proof}
	
	\subsection{Degeneration of the geometric Satake equivalence}

	We start to investigate the representation theory of $\cdo$. Let us collect some useful facts on the representation theory of conformal vertex algebras. Let $V$ be a conformal vertex algebra with conformal vector $\omega$, $M$ be a $V$-module and~$m \in M$. 
	
	\begin{lemme}[Discussion following Definition 4.1.6 of \cite{LepowskyLi1994}]
		\label{lemma conformal weights are in the same Z coset when you are monogeneous}
		Assume that the operator $\omega_{(1)}$ acts semisimply on $M$ and that $M$ is a monogeneous $V$-module. Then the eigenvalues of $\omega_{(1)}$ all belong to the same congruence class modulo $\mathbb{Z}$.
	\end{lemme}
	
	\begin{proposition}[Proposition 4.5.11 and Corollary 4.5.15 of \cite{LepowskyLi1994}]
		\label{proposition annihilator is an ideal}
		Let
		$$
		\mathrm{Ann}(m) = \{v \in V \hspace{0.5mm}|\hspace{0.5mm} v(z)m = 0\}=\{v \in V \hspace{0.5mm}|\hspace{0.5mm} v_{(n)}m = 0 \textrm{ for all $n\in \Z$}\}.
		$$
		It is an ideal of $V$. In particular if $V$ is simple and $m \in M$ is nonzero then for all nonzero $v \in V$, $v(z)m \neq 0$ \textit{i.e.}, there exists $n \in \Z$ such that
		$$
		v_{(n)} m \neq 0.
		$$
	\end{proposition}

	\begin{lemme}[Definition 4.7.1 and Corollary 4.7.8 of \cite{LepowskyLi1994}]
		\label{lemma vacuum-like vectors implies there exists a morphism from the vertex algebra}
		We say that $m$ is a vacuum-like vector if for all $v\in V$ and~$n\geqslant0$
		$$
		v_{(n)} m = 0.
		$$
		In that case there exists a unique morphism of $V$-modules
		$$
		V \lra M
		$$
		mapping the vacuum $\vac$ to $m$. In particular, if $V$ is simple and $M$ is a simple $V$-module, it is isomorphic to the adjoint module $V$ if and only if it contains a nonzero vacuum-like vector.
	\end{lemme}
	
	\begin{lemme}[Corollary 4.7.6 of \cite{LepowskyLi1994}]
		\label{lemma criterion for being vacuum-like}
		The vector $m$ is vacuum-like if and only if~$\omega_{(0)} m = 0$.
	\end{lemme}
	
	The conformal structure of $\cdo$ is compatible under $\pi_{L} \otimes \pi_{R}$ with the conformal structure given by the Segal--Sugawara construction on $V^{\kappa}(\Lg)\otimes V^{\kappa^{*}}(\Lg)$ as we explained in §\ref{section The conformal structure}.	In particular any $\cdo$-module is naturally a $V^{\kappa}(\Lg)\otimes V^{\kappa^{*}}(\Lg)$-module by restriction. The previous lemma implies:
	
	\begin{lemme}
		\label{lemma criterion for beinn vacuum-like in the case of the cdo}
		Let $M$ be a $\cdo$-module. A vector $m \in M$ is vacuum-like with respect to the $\cdo$-module structure if and only if it is vacuum-like with respect to the $V^{\kappa \oplus \kappa^{*}}(\Lg \oplus \Lg)$-module structure. In particular any morphism of $V^{\kappa \oplus \kappa^{*}}(\Lg \oplus \Lg)$-modules from $V^{\kappa \oplus \kappa^{*}}(\Lg \oplus \Lg)$ to $M$ extends uniquely to a morphism of $\cdo$-modules from $\cdo$ to $M$. Moreover if $M$ is a simple $\cdo$-module, it is isomorphic as a $\cdo$-module to the adjoint module $\cdo$ if and only if it contains $V^{\kappa}_{0}\otimes V^{\kappa^{*}}_{0}$ as a $V^{\kappa \oplus \kappa^{*}}(\Lg \oplus \Lg)$-submodule.
	\end{lemme}	
	
	We now define some of the categories that will be interesting for us
	
	\begin{definition}
		\label{definition module categories for the cdo}
		Denote by 
		$$
		\cdo\mathrm{-Mod}^{1 \times \jetinf G}
		$$ 
		the full subcategory of $\cdo\mathrm{-Mod}$ consisting of modules belonging to $V^{\kappa}(\Lg)\mathrm{-Mod}^{\jetinf G}$ when seen as $V^{\kappa^{*}}(\Lg)$-modules.  Let 
		$$
		\cdo\mathrm{-Mod}^{\jetinf G \times \jetinf G}
		$$ 
		be the full subcategory of $\cdo\mathrm{-Mod}$ consisting of modules belonging to the Kazhdan--Lusztig category~$V^{\kappa \oplus \kappa^{*}}(\Lg \oplus \Lg)\mathrm{-Mod}^{\jetinf G \times \jetinf G}$ when seen as $V^{\kappa \oplus \kappa^{*}}(\Lg \oplus \Lg)$-modules.
	\end{definition}
	
	\begin{note}
		\begin{remarque}
			As explained in the introduction, the category $\cdo\mathrm{-Mod}^{1 \times \jetinf G}$ is meant to be a substitute for the category of $\kappa$-twisted $\cD$-modules on $\mathrm{Gr}_{G}$ and $\cdo\mathrm{-Mod}^{\jetinf G \times \jetinf G}$ is a substitute for the category of $\kappa$-twisted $\cD$-modules on $\mathrm{Gr}_{G}$ that are $\jetinf G$-equivariant. 
		\end{remarque}
	\end{note}

	\begin{theoreme}[Degeneration of the geometric Satake equivalence]
		\label{theorem degeneration of the geometric Satake correspondence}
		Let $G$ be a simple algebraic group and $\kappa$ be a generic level. Then $\cdo$ is the only simple object of the category $\cdo\mathrm{-Mod}^{\jetinf G \times \jetinf G}$. Moreover the category $\cdo\mathrm{-Mod}^{\jetinf G \times \jetinf G}$ is equivalent to the category of $\C$-vector spaces. 
	\end{theoreme}
	\begin{proof}
		Recall from Definition \ref{definition generic level} that in this setting, the data of a generic level is that of an irrational complex number $k \in \C \setminus \Q$ such that $\kappa = k \widetilde{\kappa_{\Lg}}$. 
		
		The fact that $\cdo$ is a simple object of the category $\cdo\mathrm{-Mod}^{\jetinf G \times \jetinf G}$ follows from Proposition \ref{proposition cdo is a simple vertex algebra} and Theorem \ref{theorem chiral Peter-Weyl}. 
		
		Let $M$ be an object of the category $\cdo\mathrm{-Mod}^{\jetinf G \times \jetinf G}$. Thanks to Theorem~\ref{theorem Kazhdan--Lusztig categories for irrational levels} there exists a nonempty set $I$ and two families~$(\alpha_{i})_{i \in I}, (\beta_{i})_{i \in I} \in X^{*}(T)_{+}$ such that we have an isomorphism of $V^{\kappa}(\Lg) \otimes V^{\kappa^{*}}(\Lg)$-modules
		\begin{equation}
			\label{equation decomposition of M in simple Weyl modules}
			M = \bigoplus_{i \in I} V_{\alpha_{i}}^{\kappa} \otimes V^{\kappa^{*}}_{\beta_{i}}.
		\end{equation}
		Let $x,y \in \Lg$ and $f,g\in \cO_{G}$, then the following equalities hold in $\End(M)$
		\begin{align*}
			&[\pi_{L}(x)_{(0)},\pi_{L}(y)_{(0)}] = \pi_{L}([x,y])_{(0)}, \\
			&[\pi_{L}(x)_{(0)},f_{(-1)}] = (x_{L}(f))_{(-1)}, \\
			&[f_{(-1)},g_{(-1)}] = 0,
		\end{align*}
		in view of equations \eqref{equation commutation relation for left-invariant vector field in the cdo lambda bracket}, \eqref{equation Leibniz rule in the cdo lambda bracket} and \eqref{equation cO(G) is commutative inside of the cdo}. This readily implies that $M$ is a representation of the ring $\cD_{G}$ of differential operators on the group $G$. Let $i\in I$ and consider the~$\Lg$-submodule $V_{\alpha_{i}}\otimes V_{\beta_{i}} \subset M$. Let $M'$ be the $\cD_{G}$-submodule generated by $V_{\alpha_{i}}\otimes V_{\beta_{i}}$. Let~$j \in I$, the conformal degrees appearing in the factor $V^{\kappa}_{\alpha_{j}} \otimes V^{\kappa^{*}}_{\beta_{j}}$ are of the form 
		$$
		\frac{c_{j}}{2(k+\check{h})}+n
		$$ 
		where $c_{j} \in \Q$ and $n \in \Z_{+}$ in view of the conformal degrees appearing in Weyl modules (see \textit{e.g.}, Proposition 2.7 of \cite{kazhdan1991tensor} or Theorem 6.2.21 of \cite{LepowskyLi1994}). Now since~$k \notin \mathbb{Q}$ and hence $(k+\check{h}) \notin \Q$, for all $i,j \in I$ the equality 
		$$
		\frac{c_{j}}{2(k+\check{h})}+n = \frac{c_{i}}{2(k+\check{h})}
		$$
		forces $n$ to be zero. But it is clear that the action of $\cD_{G}$ preserves conformal degree. These two facts show that there exists a subset $S \subset I$ such that 
		$$
		M' \subset \bigoplus_{s \in S} V_{\alpha_{s}} \otimes V_{\beta_{s}}.
		$$
		Using Corollary \ref{corollary sufficient condition for integrability} we see that $M'$ is a $\cD_{G}$-module such that the action of $\Lg$ integrates to an action of the algebraic group $G$. Then using for instance Proposition 30.18 of~\cite{EtingofRepresentationsLieGroups2024} we conclude that $M'$ is a $G$-equivariant $\cD_{G}$-module. 
		
		This is easily seen to imply that it is a direct sum of copies of $\cO_{G}$ as a $\cD_{G}$-module and so in particular as a $\Lg \oplus \Lg$-module. In particular, $0 \in S\subset I$. In other words, $V^{\kappa}_{0} \otimes V^{\kappa^{*}}_{0} = V^{\kappa}(\Lg) \otimes V^{\kappa^{*}}(\Lg)$ appears as a~$V^{\kappa}(\Lg) \otimes V^{\kappa^{*}}(\Lg)$-submodule of $M$. Lemma \ref{lemma criterion for beinn vacuum-like in the case of the cdo} then concludes that $M$ contains a copy of $\cdo$.
		
		So it is now clear that $\cdo$ is the only simple object of the category. Because any object contains a copy of $\cdo$, we simply have to show that $\cdo$ is projective to conclude. Let 
		$$
		0 \lra A \lra B \overset{\pi}{\lra} \cdo \lra 0
		$$
		be an exact sequence in the category $\mathcal{D}_{G}^{\kappa}\mathrm{-Mod}^{\jetinf G \times \jetinf G}$. This is also an exact sequence in the category $V^{\kappa\oplus\kappa^{*}}(\Lg\oplus \Lg)\mathrm{-Mod}^{\jetinf G \times \jetinf G}$. Since $V^{\kappa}(\Lg) \otimes V^{\kappa^{*}}(\Lg) \subset \cdo$ is projective in $V^{\kappa\oplus\kappa^{*}}(\Lg \oplus \Lg)\mathrm{-Mod}^{\jetinf G \times \jetinf G}$ (Theorem \ref{theorem Kazhdan--Lusztig categories for irrational levels}) we can lift it to $B$ as a $V^{\kappa\oplus\kappa^{*}}(\Lg\oplus \Lg)$-module via a map $s : V^{\kappa\oplus\kappa^{*}}(\Lg \oplus \Lg)\lra B$. Using Lemma~\ref{lemma criterion for beinn vacuum-like in the case of the cdo}, we can extend the map $s$ to a map of $\cdo$-modules 
		$$
		\tilde{s} : \cdo \lra B.
		$$
		Now it is clear that the map $\tilde{s}$ is a section because $\pi \circ \tilde{s} : \cdo \lra \cdo$ is a map of $\cdo$-modules that coincides with the identity on $V^{\kappa}(\Lg) \otimes V^{\kappa^{*}}(\Lg)$. By the simplicity of $\cdo$ (Proposition \ref{proposition cdo is a simple vertex algebra}) and Schur's Lemma it must be the identity map.		
	\end{proof}
	
	\begin{remarque}
		We note that the argument appearing at the end of this proof can be generalized as follows. Let $V$ be a conformal vertex algebra and $W$ be a simple conformal vertex algebra that is a conformal extension of $V$. Let $\mathcal{C}$ be a full subcategory of $V\mathrm{-Mod}$ such that $V,W \in \mathcal{C}$. Let $\mathcal{C'}$ be the full subcategory of $W\mathrm{-Mod}$ that consists of $W$-modules that belong to $\mathcal{C}$ when seen as $V$-modules. Then if $V$ is projective in $\mathcal{C}$ then $W$ is projective in $\mathcal{C}'$.
	\end{remarque}
	
	\section{Spectral flow in the theory of vertex algebras}
	\label{section spectral flow in the theory of vertex algebras}
	
	In view of Theorem \ref{theorem degeneration of the geometric Satake correspondence}, to build new interesting representations of $\cdo$ we have to move away from the subcategory 
	$$
	\Satakecategory.
	$$ 
	This is a motivation to introduce spectral flows (see Remark \ref{remarque Kazhdan--Lusztig categories are unstable with respect to spectral flow}). To put the notion in perspective, let us introduce a general framework before applying it in our setting. 
	
	\subsection{Generalities}
	Fix a finite dimensional abelian Lie algebra $\Lh$ endowed with a nondegenerate symmetric bilinear form $\kappa$. Let $V^{\kappa}(\Lh)$ be the Heisenberg vertex algebra associated to this data. Let $V$ be a conformal vertex superalgebra with conformal vector $\omega$ endowed with a vertex superalgebra morphism 
	$$
	\varphi : V^{\kappa}(\Lh) \lra V.
	$$
	This morphism is necessarily injective since $V^{\kappa}(\Lh)$ is simple in view of Theorem~\ref{theorem Kazhdan--Lusztig categories for irrational levels}, so we view $V^{\kappa}(\Lh)$ as a vertex subalgebra of $V$. We require the following two conditions  
	
	\begin{enumerate}[label=\roman*), ref=\roman*]	
		\item \label{condition direct sum of fock spaces} When viewing $V$ as a $V^{\kappa}(\Lh)$-module via $\varphi$ then $V$ is a direct sum of Fock spaces \textit{i.e.}, there exist a set $i \in I$ and a map $\lambda : i\in I \mapsto \lambda_{i} \in \Lh^{*}$ such that 
		$$
		V = \bigoplus_{i \in I} \cF^{\kappa}_{\lambda_{i}}.
		$$
		\item \label{condition primary fields} All the elements $h \in \Lh = V^{\kappa}(\Lh)_{1}$ are primary of weight one in $V$, that is, we have 
		$$
		[\omega_{\lambda}\varphi(h)] = (\partial + \lambda)\varphi(h).
		$$
		
	\end{enumerate}
	
	We associate to this data the following abelian group that depends on $V, \Lh, \kappa$ and $\varphi$. In the notation we only include the dependency on $V$ as the others should always be clear from the context.
	\begin{definition}
		\label{definition spectral flow group}
		In the previous notations, we define
		$$
		\cSF(V) = \{ h \in V^{\kappa}(\Lh)_{1} = \Lh \hspace{0.5mm} | \hspace{0.5mm} \varphi(h)_{(0)} \textrm{ acts semisimply on $V$ with eigenvalues in $\mathbb{Z}$} \}.
		$$
		We call $\cSF(V)$ the spectral flow group of $V$.
	\end{definition}

	Following \cite{Li1997} we define for any $h \in \cSF(V)$ the associated Li's delta operator
	\begin{equation}
		\label{equation Li's delta operator}
		\Delta(h,z) = z^{h_{(0)}}\exp \left(\sum_{n>0}\frac{h_{(n)}}{-n}(-z)^{-n} \right) \in \cF(V).
	\end{equation}
	We now digress a bit to discuss the meaning of this formula. First the term $z^{h_{(0)}}$ is to be understood in the following sense. Let $V$ be a $\C$-vector space and $A \in \End(V)$ such that $A$ acts semisimply with eigenvalues in $\mathbb{Z}$. Denote by $(\pi_{n})_{n \in \mathbb{Z}}$ the family of projectors associated to the direct sum decomposition 
	$$
	V = \bigoplus_{n\in \mathbb{Z}} V_{n}
	$$
	of $V$ in eigenspaces for the semi-simple endomorphism $A$. Now define 
	$$
	z^{A} = \sum_{n \in \mathbb{Z}} \pi_{n}z^{n} \in \cF(V).
	$$
	That is for each $n \in \mathbb{Z}$ and $v \in V_{n}$ 
	$$
	z^{A} v = vz^{n}.
	$$

	To understand the meaning of the exponential term appearing in formula~\eqref{equation Li's delta operator}, recall that the operators $(h_{(n)})_{n\in \mathbb{Z}}$ satisfy the Heisenberg commutation relations. In particular the annihilation operators \textit{i.e.}, the nonnegative modes $(h_{(n)})_{n \geqslant0}$, pairwise commute. The exponential is then defined with the usual formula and it indeed defines a field on $V$. In fact, it is easily checked that more is true. Namely for all $h \in \cSF(V)$ and $v \in V$, 
	$$
	\Delta(h,z) v \in V[z,z^{-1}] \subset V((z)).
	$$
	Note that for all $h,h' \in \Lh$, we have the equalities
	\begin{align*}
		\Delta(h,z)\Delta(h',z) &= \Delta(h+h',z), \\
		\Delta(0,z) &= \id_{V},
	\end{align*}
	so that we have 
	$$
	\Delta(h,z)^{-1} = \Delta(-h,z).
	$$

	Now, given $h \in \cSF(V)$ and a module $M$ of the vertex algebra $V$, we can define a new module $h \cdot M$ as follows. The underlying vector space of $h \cdot M$ is the same, but the action gets twisted. If $Y_{M} : V \lra \cF(M)$ is the map defining the module structure of M we define for $v \in V$ 
	$$
	Y_{h \cdot M}(v,z) = Y_{M}(\Delta(h,z)^{-1} v,z) = Y_{M}(\Delta(-h,z)v,z),
	$$
	where $Y_{M} : V \lra \cF(M)$ is naturally extended to a $\C[z,z^{-1}]$-linear map denoted again $Y_{M} : V[z,z^{-1}] \lra \cF(M)$. The following theorem summarizes the properties of spectral flow:
	\begin{theoreme}[\cite{Li1997} and Proposition 5.4 of \cite{li1995localsystemstwistedvertex}]
		\label{theorem properties of spectral flow}
		Let $V$ be a conformal vertex superalgebra, $\Lh$ be a finite dimensional abelian Lie algebra with a nondegenerate symmetric bilinear form $\kappa$. Let 
		$$
		\varphi : V^{\kappa}(\Lh) \lra V
		$$
		be a vertex superalgebra morphism satisfying Conditions (\ref{condition direct sum of fock spaces}) and (\ref{condition primary fields}).
		
		Let $\cSF(V)$ be the spectral flow group associated to this data as in Definition~\ref{definition spectral flow group}. Let $M$ be a $V$-module and $h,h' \in \cSF(V)$, then $h\cdot M$ is again a $V$-module and we have $h \cdot (h' \cdot M) \simeq (h + h') \cdot M$ and $ 0 \cdot M \simeq M$. In particular, $\cSF(V)$ acts by invertible exact endofunctors on the category $V\mathrm{-Mod}$ and sends simple objects to simple objects.  
	\end{theoreme}

	\subsection{The spectral flow group of $V^{\kappa}(\Lg)$}
	We work out in detail the application of this theory to the Heisenberg vertex algebra and to the universal affine vertex algebra associated with a reductive Lie algebra. Of course, Heisenberg vertex algebras are particular cases of affine vertex algebras. So some results will be generalized later. Yet we believe that to understand what is going on, it is essential to keep this example in mind.
	
	Let $\Lh$ be a finite dimensional abelian Lie algebra and $\kappa$ a nondegenerate symmetric bilinear form on $\Lh$. Let $V^{\kappa}(\Lh)$ be the associated Heisenberg vertex algebra. In the previous notations, we choose $\varphi$ to be the identity which is easily seen to satisfy conditions~(\ref{condition direct sum of fock spaces}) and (\ref{condition primary fields}). Recall that by definition 
	$$
	\cSF(V^{\kappa}(\Lh)) = \{ h \in V^{\kappa}(\Lh)_{1} = \Lh \hspace{0.5mm} | \hspace{0.5mm} h_{(0)} \textrm{ acts semisimply on $V^{\kappa}(\Lh)$ with eigenvalues in $\mathbb{Z}$} \}.
	$$
	\begin{proposition}
		The spectral flow group $\cSF(V^{\kappa}(\Lh))$ is equal to $\Lh$. 
	\end{proposition}
	\begin{proof}
		We have that for all $h\in \Lh$, $h_{(0)} = 0$ in $\End(V^{\kappa}(\Lh))$ because the Lie algebra $\Lh$ is abelian and of the Heisenberg commutation relations. The result follows.
	\end{proof}
	We can now compute what happens on any module. If $x \in \cSF(V^{\kappa}(\Lh)) = \Lh$ and $h \in \Lh = V^{\kappa}(\Lh)_{1} \subset V^{\kappa}(\Lh)$ then 
	\begin{align*}
		\Delta(x,z) h = z^{x_{(0)}}\exp\left(\sum_{n>0}\frac{x_{(n)}}{-n}(-z)^{-n} \right)  h =  h + \kappa(x,h)\vac z^{-1}.
	\end{align*} 
	Because $V^{\kappa}(\Lh)$ is generated as a vertex algebra by elements in $\Lh = V^{\kappa}(\Lh)_{1}$, this computation is all we need. Let $M$ be a $V^{\kappa}(\Lh)$-module via the map $Y_{M} : V^{\kappa}(\Lh) \lra \cF(M)$. The action is twisted as follows. Set  
	\begin{align*}
		Y_{x\cdot M}(h,z) := Y_{M}(\Delta(-x,z)h,z) = Y_{M}(h,z) - \kappa(x,h)\id z^{-1}.
	\end{align*}
	That is to say $x \cdot M$ is the $\Lhkappa$-module such that for all $m \in M$ and $n \in \mathbb{Z}$ 
	\begin{equation}
		\label{equation spectral flow on heisenberg modes}
		h_{(n)} \cdot_{new} m = h_{(n)} \cdot_{old} m - \delta_{n,0} \kappa(x,h)m 
	\end{equation}
	where the subscript \emph{old} refers to the given module structure on $M$ and \emph{new} to the one obtained after spectral flow on $x \cdot M$ (recall that the underlying vector space is unchanged).
	The following is a fundamental example:
	\begin{proposition}
		\label{proposition spectral flow of Fock spaces}
		Let $\Lh$ be an abelian Lie algebra and $\kappa$ be a nondegenerate symmetric bilinear form on $\Lh$. Let $V^{\kappa}(\Lh)$ be the associated Heisenberg vertex algebra. Let~$x \in \cSF(V^{\kappa}(\Lh)) = \Lh$ and $\lambda \in \Lh^{*}$ and $\cF^{\kappa}_{\lambda}$ be the Fock space(=Weyl module) of weight $\lambda$ then there exists an isomorphism of $\Lhkappa$-modules
		$$
		x \cdot \cF_{\lambda} \simeq \cF_{\lambda - \kappa(x,\cdot) }.
		$$
	\end{proposition}
	\begin{proof}
		It is obvious from formula \eqref{equation spectral flow on heisenberg modes} that after spectral flow the highest weight vector~$\vaclambda$ becomes a highest weight vector of weight $\lambda - \kappa(x,\cdot)$. So Proposition~\ref{proposition Weyl modules are highest weight modules} defines for us a nonzero map of $V^{\kappa}(\Lh)$-modules 
		$$
		\cF_{\lambda - \kappa(x,\cdot)} \lra x \cdot \cF_{\lambda},
		$$ 
		and we know from Theorems \ref{theorem Kazhdan--Lusztig categories for irrational levels} and \ref{theorem properties of spectral flow} that both modules are simple, hence the map is an isomorphism. 
	\end{proof}
	
	This result allows us to think of all Fock spaces as being an orbit under spectral flow of the vacuum (zero) Fock space. In fact, since $\kappa$ is nondegenerate, when $h$ ranges through the spectral flow group $\cSF(V^{\kappa}(\Lh)) = \Lh$, $\kappa(h,\cdot)$ ranges through the whole dual space $\Lh^{*}$, so starting from the vertex algebra $\cF_{0}=V^{\kappa}(\Lh)$ itself we obtain every Fock space.

	\begin{note}
		We now have a look at what happens to the conformal structure. Let $(h^{i})_{1 \leqslant i \leqslant\dim(\Lh)}$ be an orthonormal basis with respect to $\kappa$, that is, a $\C$-basis of $\Lh$ such that for all $1 \leqslant i,j \leqslant\dim(\Lh)$
		$$
		\kappa(h^{i},h^{j}) = \delta_{i,j}.
		$$
		A conformal vector of $V^{\kappa}(\Lh)$ is given by the following formula 
		$$
		\omega = \frac{1}{2} \sum_{i=1}^{\dim(\Lh)} h^{i}_{(-1)}h^{i}_{(-1)} \vac.
		$$
		Fix $x \in \cSF(V^{\kappa}(\Lh)) = \Lh$, we compute
		\begin{align*}
			\Delta(x,z) \omega &= z^{x_{(0)}}\exp\left(\sum_{n>0}\frac{x_{(n)}}{-n}(-z)^{-n} \right)  \omega \\
			&= z^{x_{(0)}} \left(\id  \omega + \frac{1}{1!}\sum_{n>0}\frac{x_{(n)} \omega}{-n}(-z)^{-n} + \frac{1}{2!}\sum_{n,m>0} \frac{x_{(n)}x_{(m)}  \omega}{(-n)(-m)}(-z)^{-(n+m)} + \dots \right)\\	
			&= z^{x_{(0)}}\left( \omega + \frac{x_{(1)} \omega}{-1} (-z)^{-1} + \frac{x_{(2)}  \omega}{-2}(-z)^{-2} + \frac{1}{2!} \frac{x_{(1)}x_{(1)} \omega}{(-1)(-1)}(-z)^{-2} \right) \\
			&= z^{x_{(0)}}\left( \omega + x_{(1)} \omega z^{-1} - \frac{(x_{(2)} - x_{(1)}x_{(1)})  \omega}{2}z^{-2} \right) \\
			&= \omega + x_{(1)} \omega z^{-1} - \frac{(x_{(2)} - x_{(1)}x_{(1)})  \omega}{2}z^{-2}.
		\end{align*}
		Where all the vanishings are obvious due to conformal degree reasons, we are left with three terms to compute, it is easy to see that 
		\begin{align*}
			x_{(2)}  \omega = 0
		\end{align*}
		because $x_{(2)}$ commutes with all the $(h^{i}_{(-1)})_{1\leqslanti \leqslant\dim(\Lh)}$ and $x_{(2)}\vac = 0$. We compute 
		\begin{align*}
			x_{(1)} \omega &= \frac{1}{2}\sum_{i=1}^{\dim(\Lh)} x_{(1)}h^{i}_{(-1)}h^{i}_{(-1)}\vac \\
			&= \frac{1}{2}\sum_{i=1}^{\dim(\Lh)} h^{i}_{(-1)}x_{(1)}h^{i}_{(-1)}\vac + \frac{1}{2}\sum_{i=1}^{\dim(\Lh)} \kappa(h^{i},x)h^{i}_{(-1)}\vac \\
			&= \frac{1}{2}\sum_{i=1}^{\dim(\Lh)} h^{i}_{(-1)}h^{i}_{(-1)}x_{(1)}\vac + \frac{1}{2}\sum_{i=1}^{\dim(\Lh)} h^{i}_{(-1)}\kappa(h^{i},x)\vac + \frac{1}{2}\sum_{i=1}^{\dim(\Lh)} \kappa(h^{i},x)h^{i}_{(-1)}\vac \\
			&= 0 + \frac{1}{2}\sum_{i=1}^{\dim(\Lh)} h^{i}_{(-1)}\kappa(h^{i},x)\vac + \frac{1}{2}\sum_{i=1}^{\dim(\Lh)} \kappa(h^{i},x)h^{i}_{(-1)}\vac \\
			&=  \frac{1}{2}x_{(-1)}\vac + \frac{1}{2}x_{(-1)}\vac = x.
		\end{align*}
		Finally 
		\begin{align*}
			x_{(1)}x_{(1)} \omega &= x_{(1)}  x_{(1)}	 \omega \\
			&= x_{(1)}  x \textrm{ thanks to the previous computation,}\\
			&=  x_{(1)}x_{(-1)}\vac \\
			&=  \kappa(x,x)\vac.
		\end{align*}
		Putting everything together we obtain 
		\begin{equation}
			\label{equation action of the spectral flow on abelian segal sugawara}
			\Delta(x,z)\omega = \omega + x z^{-1} + \frac{\kappa(x,x)}{2}\vac z^{-2}.
		\end{equation}
		
		In particular we can see how the action of the gradation operator $\omega_{(1)}$ transforms. If $M$ is a $V^{\kappa}(\Lh)$-module, after applying a spectral flow of parameter $x$ we have for all $m \in M$ from equation \eqref{equation action of the spectral flow on abelian segal sugawara} the equality 
		$$
		\omega_{(1)} \cdot_{new} m = \omega_{(1)} \cdot_{old} m - x_{(0)} \cdot_{old} m + \frac{\kappa(x,x)}{2}m.
		$$
		
	\end{note}
	
	\begin{note}
		We are mostly interested by how the Hamiltonian transforms under spectral flow, it is usual to define $L_{0}^{ab}$ to be $\omega_{(-1)}$ \aref. If $M$ is a $V^{\kappa}(\Lh)$-module, $m \in M$ is a vector and $x \in \cSF(V^{\kappa}(\Lh))$ we see from formula \eqref{equation action of the spectral flow on abelian segal sugawara} that 
		\begin{equation}
			\label{equation action of the spectral flow on abelian Hamiltonian}
			L_{0}^{ab} \cdot_{new} m = (L_{0}^{ab} + x_{(0)} + \frac{\kappa(x,x)}{2}\id) \cdot_{old} m.
		\end{equation}
		
		\begin{remarque}
			\label{remark there is a more conceptual way to make the computation}
			\textcolor{red}{There is a more effective way to do these computations, but is this worth noticing?.}
		\end{remarque}
	\end{note}

	Let $\Lg$ be a reductive Lie algebra and $\kappa$ be a nondegenerate symmetric bilinear $\Lg$-invariant form on $\Lg$. If we fix a Cartan subalgebra $\Lh$ of $\Lg$ and denote again by $\kappa$ the restriction from $\Lg$ to $\Lh$, it is still nondegenerate. We have a morphism of vertex algebras
	$$
	\varphi : V^{\kappa}(\Lh) \lra V^{\kappa}(\Lg)
	$$
	satisfying conditions (\ref{condition direct sum of fock spaces}) and (\ref{condition primary fields}).
	
	\begin{proposition}[See also §2 of \cite{CreutzigRidout2013}]
		\label{proposition spectral flow group associated to affine vertex algebras}
		The spectral flow group $\cSF(V^{\kappa}(\Lg))$ is equal to the additive group of coweights $\check{P}$.
	\end{proposition}
	\begin{proof}
		Let $x \in \cSF(V^{\kappa}(\Lg)) \subset \Lh \subset \Lg.$ For every root $\alpha \in \Phi$ pick a nonzero root vector~$e^{\alpha} \in \Lg^{\alpha}$. The following equality holds in $V^{\kappa}(\Lg)$ 
		\begin{align*}
			x_{(0)} e^{\alpha} = \alpha(x)e^{\alpha} \vac.
		\end{align*}
		Because we assumed that $x \in \cSF(V^{\kappa}(\Lg))$, it forces that $\alpha(x) \in \mathbb{Z}$ for all $\alpha \in \Phi$, this is precisely asking that $x$ belongs to $\check{P}$. Conversely, if $x \in \check{P}$ then one easily checks that $x_{(0)}$ acts semisimply with integral eigenvalues.
	\end{proof}
	
	Recall the root space decomposition \eqref{equation root space decomposition of Lg} of $\Lg$ and let $x \in \check{P}$ be a coweight, $\alpha \in \Phi$ be a root, $e^{\alpha}$ be a root vector of weight $\alpha$ and $h \in \Lh$. The same computation as in the Heisenberg case shows that 
	$$
	\Delta(x,z) h = h + \kappa(x,h)\vac z^{-1}.
	$$
	A simple computation shows that 
	\begin{align*}
		\Delta(x,z) e^{\alpha} = z^{\alpha(x)}e^{\alpha}.
	\end{align*}

	Because $V^{\kappa}(\Lg)$ is generated as a vertex algebra by elements in $\Lg = V^{\kappa}(\Lg)_{1}$ this computation is all we need. Let $M$ be a $V^{\kappa}(\Lg)$-module via $Y_{M} : V^{\kappa}(\Lg) \lra \cF(M)$. The twist is as follows, for all $h \in \Lh \subset V^{\kappa}(\Lg)_1$
	$$
	Y_{x\cdot M}(h,z) := Y_{M}(\Delta(-x,z)h,z) = Y_{M}(h,z) - \kappa(x,h)\id z^{-1}.
	$$ 
	For all $\alpha \in \Phi$ and root vector $e^{\alpha} \in \Lg^{\alpha} \subset V^{\kappa}(\Lg)$, 
	\begin{align*}
		Y_{x\cdot M}(e^{\alpha},z) &:= Y_{M}(\Delta(-x,z)e^{\alpha},z)
		= z^{-\alpha(x)}Y_{M}(e^{\alpha},z).
	\end{align*} 
	
	That is to say, $x \cdot M$ is the smooth $\Lg^{\kappa}$-module such that for all $m \in M$ and $n \in \mathbb{Z}$ we have 
	\begin{align}
		\label{equation spectral flow on affine algebra modes for root vector}
		{e^{\alpha}}_{(n)} \cdot_{new} m &= {e^{\alpha}}_{(n-\alpha(x))} \cdot_{old} m \\
		\label{equation spectral flow on affine algebra modes for vector in the Cartan}
		h_{(n)} \cdot_{new} m &= h_{(n)} \cdot_{old} m - \delta_{n,0} \kappa(x,h)m
	\end{align}
	where the subscripts \emph{old} refers to the given module structure on $M$ and \emph{new} to the new one obtained after spectral flow on $x \cdot M$ (recall that the underlying vector space is unchanged). 
	
	\begin{remarque}
		\label{remarque Kazhdan--Lusztig categories are unstable with respect to spectral flow}
		It is clear that, for a simple Lie algebra, the Kazhdan--Lusztig category is not stable under spectral flow. Namely for any nonzero $x \in \check{P}$ and $M \in V^{\kappa}(\Lg)\mathrm{-Mod}^{\jetinf G}$, equation \eqref{equation spectral flow on affine algebra modes for root vector} shows that 
		$$
		x \cdot M \notin V^{\kappa}(\Lg)\mathrm{-Mod}^{\jetinf G}.
		$$
		This highly contrasts with the case of an abelian Lie algebra (compare with Proposition~\ref{proposition spectral flow of Fock spaces}).
	\end{remarque}
	
	\begin{note}
		Let us assume that $\kappa$ is such that $\kappa-\kappa_{c}$ is nondegenerate where $\kappa_{c}$ is the critical level so that the Segal--Sugawara construction is available, making $V^{\kappa}(\Lg)$ into a conformal vertex algebra.
		
		We now compute the action of spectral flow on the Segal--Sugawara construction. Using the notations of Appendix \ref{section The Kazhdan--Lusztig category} the conformal vector $\omega$ of $V^{\kappa}(\Lg)$ is given by the sum of the conformal vector on the Heisenberg vertex algebra corresponding to the center $\Lzg$ of $\Lg$
		$$
		\omega^{ab} = \frac{1}{2}\sum_{j=1}^{\dim(\Lzg)} h^{j}_{(-1)}h^{j}_{(-1)}\vac \in V^{\kappa_{|\Lzg}}(\Lzg) \subset V^{\kappa}(\Lg),
		$$
		and of the affine vertex algebras corresponding to each simple factor of $\Lg$ 
		$$
		\omega^{s} = \frac{1}{2(k_{s}+\check{h}_{s})} \left( \sum_{j=1}^{\dim(\Lh_{s})}h^{j}_{(-1)}{h_{j}}_{(-1)} + \sum_{\alpha \in \Phi(\Lg_{s},\Lh_{s})_{+}} e^{\alpha}_{(-1)}f^{\alpha}_{(-1)} + f^{\alpha}_{(-1)}e^{\alpha}_{(-1)} \right) \vac \in V^{\kappa_{|\Lg_{s}}},
		$$
		that is to say 
		$$
		\omega = \omega^{ab} + \sum_{s=1}^{r} \omega^{s}.
		$$
		
		\begin{note}
			\textcolor{red}{We can give a way more effective proof by simply doing the reductive case directly without ever refering to any decomposition, by simply using the conceptual computation involving omega(-1), but that's cheating somehow}
		\end{note}
		\begin{proposition}
			\label{proposition modification of the segal sugawara conformal vector under spectral flow for reductive case}
			Let $x \in \cSF(V^{\kappa}(\Lg)) = \check{P}$. The following equality holds in $V^{\kappa}(\Lg)[z,z^{-1}]$
			$$
			\Delta(x,z)\omega = \omega + xz^{-1} + \frac{\kappa(x,x)}{2}\vac z^{-2}.
			$$
			In particular, for any $V^{\kappa}(\Lg)$-module $M$ and $m \in M$, the new action of $\omega_{(1)}$ on~$x \cdot M$ is given in terms of the old one by 
			\begin{equation}
				\label{equation new action of the Hamiltonian for a affine lie algebra spectral flow}
				\omega_{(1)} \cdot_{new} m = (\omega_{(1)} - x_{(0)} + \frac{\kappa(x,x)}{2}\id) \cdot_{old} m.
			\end{equation}
		\end{proposition}
		Before starting the proof of the proposition, we single out the following easy and useful consequence of equation \eqref{equation new action of the Hamiltonian for a affine lie algebra spectral flow}
		\begin{corollaire}
			\label{corollary spectral flow and modification of conformal weight}
			Let $x\in \cSF(V^{\kappa}(\Lg)) = \check{P}$, $M$ be a $V^{\kappa}(\Lg)$-module fix $m \in M$ a vector of weight $\lambda \in \Lh^{*}$ and of conformal weight $\mu \in \C$ with respect to $L_{0}$. Then, in $x\cdot M$, the vector $m$ is of weight $\lambda - \kappa(x,\cdot) \in \Lh^{*}$ and of conformal weight $\mu - \lambda(x) + \frac{\kappa(x,x)}{2}$. 
		\end{corollaire}
		We now turn to the proof of the Proposition 
		\begin{proof}
			We start by reducing ourselves to the case where $\Lg$ is simple or abelian. Recall that $x \in \check{P} \subset \Lh$ and that we have the decomposition 
			$$
			\Lh = \Lzg \oplus \bigoplus_{s=1}^{r} \Lh_{s}.
			$$
			Write $x = x^{ab} + \sum_{s=1}^{r}x_{s}$ with respect to this direct sum decomposition. The following equality holds 
			$$
			\Delta(x,z)\omega = \Delta(x^{ab},z)\omega^{ab} + \sum_{s=1}^{r}\Delta(x_{s},z)\omega^{s}/
			$$
			It is easy to see that proving the result for the abelian and each simple factor yields the full result because, the level $\kappa$ being $\Lg$-invariant, the decomposition $\Lg = \Lzg \oplus \bigoplus_{s=1}^{r}\Lg^{s}$ is orthogonal with respect to $\kappa$. 
			
			The abelian case has already been proven before. We now assume $\Lg$ is simple. Then, as before, for conformal degree reasons we have 
			$$
			\Delta(x,z)\omega = \omega + x_{(1)}  \omega z^{-1} - \frac{ x_{(2)} \omega - x_{(1)}x_{(1)} \omega}{2}z^{-2}.
			$$
			The following equality holds in $\End(V^{\kappa}(\Lg))$ for all $n \in \mathbb{Z}$ and $a\in \Lg$
			\begin{equation}
				\label{equation commutation relation with omega(-1)}
				[\omega_{(-1)},a_{(n)}] = -na_{(n-2)},
			\end{equation}
			see for instance Theorem 6.2.16 in \cite{LepowskyLi1994}. The result is easy to deduce from equation \eqref{equation commutation relation with omega(-1)}. We have 
			\begin{align*}
				x_{(1)} \omega &= x_{(1)}\omega_{(-1)}\vac \\
				&= \omega_{(-1)}x_{(1)}\vac -[x_{(1)},\omega_{(-1)}]\vac \\
				&= 0-[\omega_{(-1)},x_{(1)}]\vac\\
				&= x_{(-1)}\vac \\
				&= x.
			\end{align*}
			Similarly 
			\begin{align*}
				x_{(2)} \omega &= -[\omega_{(-1)},x_{(2)}]\vac = -2x_{(0)}\vac = 0 ,\\
				x_{(1)}x_{(1)} \omega &= x_{(1)} (x_{(1)}\omega) = x_{(1)} x = \kappa(x,x).
			\end{align*}
			The announced formula for $\Delta(x,z)\omega$ follows. The second formula follows by picking the coefficient in front of $z^{-2}$ in $Y_{M}(\Delta(-x,z)\omega,z)$.
		\end{proof}
		
		\begin{note}
			This equation is the technical step to show that $\omega$ is a conformal vector of $V^{\kappa}(\Lg)$, see \aref Li Lepowsky page 211 theorem 6.2.16.. 
		\end{note}
		
		From now on, we think of the reductive Lie algebra $\Lg$ as the Lie algebra of some connected reductive algebraic group $G$ and we assume that $\kappa$ is generic. We are interested in the behaviour of the Kazhdan--Lusztig category $\Lgkappacategoryintegrable$ under spectral flow. 
		
		\begin{note}
			The following two examples are essential 
			
			\begin{example}
				\label{example KL category for the torus and spectral flow}
				Assume that $G = H = \G_{m}^{r}$ is an algebraic torus so that $\Lg = \Lh$ is an abelian Lie algebra and that $\kappa$ is nondegenerate. In that case, the modules in the Kazhdan--Lusztig category are (possibly infinite) direct sum of Fock spaces $\cF_{\lambda}$ whose highest weight $\lambda$ lies in the character lattice $X^{*}(T) \subset \Lh^{*}$. Let $x \in \cSF(V^{\kappa}(\Lh)) = \Lh$, \textcolor{red}{Finir l'example...}
			\end{example}
			
			\begin{example}
				\label{example KL category for simple Lie algebras}
				\textcolor{red}{Give a $\Lh$-weight argument that instantly shows there is noway we remain in the KL category for a simple Lie algebra}
			\end{example}
		\end{note}

		\begin{proposition}
			\label{proposition stability of Kazhdan--Lusztig categories}
			Let $x \in \cSF(V^{\kappa}(\Lg))=\check{P}$ and $M$ a nonzero module in the category $\Lgkappacategoryintegrable$. We write 
			$$
			x = x^{ab} + x^{ss} \in \Lzg \oplus \Lh^{ss}.
			$$ 
			Then $x \cdot M$ remains in the category $\Lgkappacategoryintegrable$ if and only if 
			$$
			\kappa(x^{ab},\cdot) \in X^{*}(T) \textrm{ and } x^{ss} = 0.  
			$$
			Moreover, in that case, for all $\lambda\in X^{*}(T)_{+}$ written in the form
			$$
			\lambda = \lambda^{ab} + \lambda^{ss} \in (\Lh^{ab})^{*} \oplus (\Lh^{ss})^{*},
			$$  
			then $x \cdot V_{\lambda}^{\kappa}$ is isomorphic to 
			$$
			V_{\lambda - \kappa(x,\cdot)}^{\kappa} = V^{\kappa}_{(\lambda^{ab}-\kappa(x^{ab},\cdot),\lambda^{ss})}
			$$ as $V^{\kappa}(\Lg)$-module.
		\end{proposition}
		\begin{proof}
			Thanks to Corollary \ref{corollary description of the kazhdan lusztig category}, and because $M$ is nonzero, it is isomorphic to a nonempty, possibly infinite, direct sum of Weyl modules so we can restrict our attention to the case where $M$ itself is a Weyl module. Let us assume that $x^{ss} \neq 0$. Then there exists a root $\alpha \in \Phi$ and a nonzero root vector $e^{\alpha} \in \Lg^{\alpha}$ such that 
			$$
			\alpha(x) > 0.
			$$
			In view of equation \eqref{equation spectral flow on affine algebra modes for root vector}, the mode $e^{\alpha}_{(\alpha(x)-1)}$ acts after twist as $e^{\alpha}_{(-1)}$ did, but since $\alpha(x)-1 \geqslant0$ this shows that $x \cdot M$ cannot remain in the Kazhdan--Lusztig category since $e^{\alpha}_{(-1)}$ does not act locally nilpotently on Weyl modules. For the converse, let $x=x^{ab} \in \Lzg$. Then the spectral flow of a Weyl module of weight $\lambda$ is isomorphic to 
			$$
			V^{\kappa}_{\lambda - \kappa(x,\cdot)} = V^{\kappa}_{(\lambda^{ab}-\kappa(x^{ab},\cdot),\lambda^{ss})}
			$$
			by the same arguments that prove Proposition \ref{proposition spectral flow of Fock spaces}. Now, such a Weyl module belongs to the Kazhdan--Lusztig category if and only if 
			$$
			\lambda - \kappa(x,\cdot) \in X^{*}(T)_{+}
			$$
			in view of Proposition \ref{proposition integrable Weyl modules}. This is the case if and only if 
			$$
			\kappa(x^{ab},\cdot) \in X^{*}(T).
			$$
		\end{proof}

		\begin{proposition}
			\label{proposition conformal weights of spectrally flown Weyl modules}
			Let $\Lg$ be a reductive Lie algebra and $\kappa$ a generic level. Let $\gamma \in \check{P} = \cSF(V^{\kappa}(\Lg))$ which we write in the form 
			$$
			\gamma = \gamma^{ab} + \gamma^{1} + \cdots + \gamma^{r} \in \Lg = \Lzg \oplus \bigoplus_{s=1}^{r} \Lg^{s},
			$$
			and let $\lambda \in P$. Then $\gamma \cdot V_{\lambda}^{\kappa}$ is again graded by conformal degrees and the possible eigenvalues of the commuting operators $(\omega^{ab}_{(1)},\omega^{1}_{(1)},...,\omega^{r}_{(1)})$ are of the form 
			\begin{multline*}
				\bigg(\frac{\kappa_{|\Lzg}(\lambda^{ab}-\gamma^{ab},\lambda^{ab}-\gamma^{ab})}{2}+n_{ab}, \frac{C_{\Lg_{1}}(\lambda_{1})}{2(k_{1}+\check{h}_{1})}+n_{1} + \frac{\kappa_{|\Lg_{1}}(\gamma_{1},\gamma_{1})}{2} , \dots, \\
				\frac{C_{\Lg_{r}}(\lambda_{r})}{2(k_{r}+\check{h}_{r})}+n_{r}+ \frac{\kappa_{|\Lg_{r}}(\gamma_{r},\gamma_{r})}{2} \bigg)
			\end{multline*}
			in the notations of Lemma \ref{lemma degres conformes des modules de Weyl} with the exception that now $n_{ab} \in \Z_{\geqslant0}$ and $n_{1},...,n_{r} \in \Z$.
		\end{proposition}
		\begin{proof}
			It is a straightforward consequence of Lemma \ref{lemma degres conformes des modules de Weyl} and Corollary \ref{corollary spectral flow and modification of conformal weight}. 
		\end{proof}
		
	\end{note}

	\subsection{The spectral flow group of $\cdo$}
	\label{section spectral flow group of cdo}	
	Denote again by
	$$
	\pi_{L} \otimes \pi_{R} : V^{\kappa \oplus \kappa^{*}}(\Lh \oplus \Lh) \lra \cdo
	$$
	the restriction to the Cartan subalgebra $\Lh \oplus \Lh$ of $\Lg \oplus \Lg$ of $\pi_{L} \otimes \pi_{R}$ (recall \eqref{equation definition pil otimes pir}).
	
	This morphism satisfies condition (\ref{condition direct sum of fock spaces}) in view of Theorem \ref{theorem chiral Peter-Weyl}, and condition~(\ref{condition primary fields}) is clearly satisfied. That allows us to define the spectral flow group:
	\begin{align*}
		\cSF(\cdo) = &\{ (h,h') \in V^{\kappa \oplus \kappa^{*}}(\Lh \oplus \Lh)_{1} \simeq \Lh \oplus \Lh \hspace{0.5mm} | \\
		\hspace{0.5mm} &\pi_{L}(h)_{(0)}+\pi_{R}(h')_{(0)} \textrm{ acts semisimply with integral eigenvalues on $\cdo$}\}.
	\end{align*}
	
	\begin{proposition}
		\label{proposition computation of the spectral flow group for the cdo}
		The group $\cSF(\cdo)$ is isomorphic to $X_{*}(T) \times \check{P}$, an isomorphism is given by 
		$$
		(\gamma,x) \in X_{*}(T) \times \check{P} \mapsto (\gamma + x, \omega_{0}^{T}(x)) \in \cSF(\cdo).
		$$
	\end{proposition}
	\begin{proof}
		As explained before, we have the inclusions of vertex algebras 
		$$
		V^{\kappa \oplus \kappa^{*}}(\Lh \oplus \Lh) \subset V^{\kappa \oplus \kappa^{*}}(\Lg \oplus \Lg) \subset \cdo.
		$$
		This readily implies the inclusion
		$$
		\cSF(\cdo) \subset \cSF(V^{\kappa\oplus \kappa}(\Lg \oplus \Lg)).
		$$
		Using Proposition \ref{proposition spectral flow group associated to affine vertex algebras} to compute the right-hand side, we see that 
		$$
		\cSF(\cdo) \subset \check{P} \times \check{P}.
		$$
		Let $(h,h') \in \cSF(\cdo)$. In particular $\pi_{L}(h)_{(0)} + \pi_{R}(h')_{(0)}$ should act on $\cO_{G} \subset \cdo$ with integral eigenvalues. For all $\lambda \in X^{*}(T)$ there exists a nonzero function in $\cO_{G}$ of weight~$(\lambda, -\omega_{0}(\lambda))$ with respect to $\Lh \oplus \Lh$. So, by assumption, we must have
		$$
		\lambda(h)-\omega_{0}\lambda(h') = \lambda(h-\omega_{0}^{T}(h')) \in \mathbb{Z}. 
		$$
		Now recall that $\eqref{equation perfect pairing between characters and cocharacters}$ defines a perfect pairing between the free $\Z$-modules $X^{*}(T)$ and~$X_{*}(T)$ inside of $\Lh^{*}$ and $\Lh$. So there exists $\gamma \in X_{*}(T)$ such that 
		$$
		h -\omega_{0}^{T}(h') = \gamma.
		$$ 
		Now, if we write $h' = \omega_{0}^{T}(x) \in \check{P}$ we get using $ (\omega_{0}^{T})^{2} = \id$ the equality
		$$
		(h,h') = (x+\gamma,\omega_{0}^{T}x).
		$$
		Conversely if we pick $(\gamma,x) \in X_{*}(T) \times \check{P}$, one can check that $\pi_{L}(x+\gamma)_{(0)} + \pi_{R}(\omega_{0}^{T}x)_{(0)}$ acts semisimply (this is clear) with integral eigenvalues on $\cdo$ using for instance the generators \eqref{equation generators of the cdo as vector space}. 
		\begin{note}
			More precisely, we can apply Lemma \ref{lemma computation of spectral flow} to $A = \pi_{L}(x+\gamma)_{(0)} + \pi_{R}(\omega_{0}^{T}x)$, $W = \{\vac\}$ and to the set of endomorphisms $\cG$ consisting of the elements
			$$ 
			\{\pi_{L}(h)_{(m)}, \pi_{L}(e^{\alpha})_{(n)}, f_{(k)} \}
			$$
			for $h \in \Lh$, $e^{\alpha} \in \Lg^{\alpha}$ for $\alpha \in \Phi$ and $f \in \cO_{G}$, with $m,n,k<0$. The fact that $\cG$ spans $\cdo$ from the vacuum follows from the explicit generators given in \eqref{equation generators of the cdo as vector space}. 
		\end{note}
	\end{proof}
	
	\begin{remarque}
		The isomorphism given in the previous proposition is non-canonical. We could have for instance chosen the following one
		$$
		(\gamma,x) \in X_{*}(T) \times \check{P} \mapsto (x, \omega_{0}^{T}(x)+\gamma) \in \cSF(\cdo).
		$$
		The reason why we make this specific choice should become clear later. From now on, we identify $\cSF(\cdo)$ with $X_{*}(T) \times \check{P}$ using the isomorphism of Proposition~\ref{proposition computation of the spectral flow group for the cdo}. Note that this result holds for any level.
	\end{remarque}	
	
	\section{Quantum Hamiltonian reduction and spectral flow}
	\label{section quantum Hamiltonian reduction and spectral flow}
	
	\subsection{Reminders on quantum Hamiltonian reduction}
	\label{section reminders on quantum Hamiltonian reduction}
	A good reference is Chapter 15 of \cite{FrenkelBenZvi2004}. Let $\Lg$ be a simple Lie algebra with a given Cartan subalgebra $\Lh$ and a choice of a Borel $\Lb$ containing it. Let
	$
	r = \dim(\Lh)
	$ 
	be the rank of $\Lg$. Let 
	$
	\Lg = \Ln_{-} \oplus \Lh \oplus \Ln_{+}
	$
	be the associated triangular decomposition. That is  
	~$\Ln_{\pm} = \sum_{\alpha \in \Phi^{\pm}} \Lg^{\alpha}$
	in terms of the notations introduced in §\ref{section reductive groups}. For each $\alpha \in \Phi_{+}$, denote by $e^{\alpha}$ a nonzero root vector in $\Lg^{\alpha}$ and by~$f^{\alpha}$ a nonzero root vector in $\Lg^{-\alpha}$ such that 
	$$
	\widetilde{\kappa_{\Lg}}(e^{\alpha},f^{\alpha}) = 1.
	$$
	Introduce the structure constants $(c^{\alpha,\beta}_{\gamma})_{\alpha,\beta,\gamma \in \Phi_{+}} \in \C$ to be such that 
	$$
	[e^{\alpha},e^{\beta}] = \sum_{\gamma \in \Phi_{+}} c^{\alpha,\beta}_{\gamma} e^{\gamma}
	$$ 
	for all $\alpha,\beta \in \Phi_{+}$.
	Let $\alpha_{1},...,\alpha_{r} \in \Phi_{+}$ be a basis of simple roots, and let 
	$$
	f = \sum_{i=1}^{r} f^{\alpha_{i}}.
	$$
	The following formula 
	$$
	\chi : x \in \Ln_{+} \longmapsto \widetilde{\kappa_{\Lg}}(f,x) \in \C
	$$
	defines a character of the Lie algebra $\Ln_{+}$ that takes on $e^{\alpha}$ the value $1$ if $\alpha$ is simple and~$0$ otherwise. Let us denote by $\bigwedge^{\frac{\infty}{2}+\bullet}\Ln_{+}$ the fermionic ghost vertex algebra associated to $\Ln_{+}$. As a vector space, it is equal to 
	$$
	\left(\bigwedge (\varphi_{\alpha,n})_{\alpha \in \Phi_{+}, n<0} \otimes \bigwedge (\varphi^{*}_{\alpha,n})_{\alpha \in \Phi_{+},n \leqslant0}\right)\vac,
	$$
	where $\bigwedge(a_{i})_{i \in I}$ denotes the exterior algebra on the elements $(a_{i})_{i \in I}$. The following (super)commutation relations, 
	\begin{align*}
		[\varphi_{\alpha,m},\varphi_{\beta,n}]_{+} &=
		[\varphi^{*}_{\alpha,m},\varphi^{*}_{\beta,n}]_{+} = 0, \\
		[\varphi_{\alpha,m},\varphi^{*}_{\beta,n}]_{+}&= \delta_{\alpha,\beta}\delta_{m,-n},
	\end{align*} 
	for all $\alpha,\beta \in \Phi_{+}$ and $m,n \in \Z$, together with the conditions
	\begin{align*}
		&\varphi_{\alpha,m} \vac = 0 \textrm{ for $m \geqslant0$,}\\
		&\varphi^{*}_{\alpha,n} \vac = 0 \textrm{ for $n >0$,}
	\end{align*}
	determine the vertex superalgebra structure. Namely if we set 
	\begin{align*}
		&\varphi_{\alpha} = \varphi_{\alpha,-1}\vac, \\
		&\varphi^{*}_{\alpha} = \varphi^{*}_{\alpha,0}\vac,
	\end{align*}
	and define 
	$$
	Y(\varphi_{\alpha},z) = \sum_{n \in \Z} \varphi_{\alpha,n}z^{-n-1}, \hspace{0.5mm}
	Y(\varphi^{*}_{\alpha},z) = \sum_{n \in \Z} \varphi^{*}_{\alpha,n}z^{-n}
	$$
	so that we have for all $n \in \Z$ the equalities 
	$$
	(\varphi_{\alpha})_{(n)} = \varphi_{\alpha,n}, \hspace{0.5mm} (\varphi^{*}_{\alpha})_{(n)} = \varphi^{*}_{\alpha,n+1}.
	$$
	It comes equipped with two gradations, to define them fix an element of the form 
	$$
	\varphi_{\alpha_{1},m_{1}}...\varphi_{\alpha_{s},m_{s}}\varphi^{*}_{\beta_{1},n_{1}}...\varphi^{*}_{\beta_{t},n_{t}}\vac
	$$ 
	where $s,t \geqslant0$, $\alpha_{1},...,\alpha_{s},\beta_{1},...,\beta_{t} \in \Phi_{+}$, $m_{1},...,m_{s} < 0$ and $n_{1},...,n_{t} \leqslant0$. Then define the degree of this element to be the integer
	$$
	-(m_{1}+...+m_{s}+n_{1}+...+n_{t})
	$$
	and its charge to be the integer
	$$
	t - s.
	$$
	Finally, define the parity of such an element to be its charge modulo $2$ to obtain a structure of a super vector space on $\bigwedge^{\frac{\infty}{2}+\bullet}\Ln_{+}$. For a given $c \in \Z$, denote the subvector space of charge $c$ of $\bigwedge^{\frac{\infty}{2}+\bullet}\Ln_{+}$ to be $\bigwedge^{\frac{\infty}{2}+c}(\Ln_{+})$. 
	
	We now introduce the functor of quantum Drinfeld--Sokolov reduction. Let $V$ be a vertex algebra equipped with a morphism 
	$$
	V^{\kappa}(\Lg) \lra V. 
	$$
	Let $M$ be a $V$-module, define 
	$$
	C(M) = M \otimes \bigwedge^{\frac{\infty}{2}+\bullet}\Ln_{+}.
	$$
	This is naturally a $C(V)=V \otimes \bigwedge^{\frac{\infty}{2}+\bullet}\Ln_{+}$-module which is $\Z$-graded by the charge gradation inherited from the second factor. We define two odd elements of degree $1$ of $C(V)$ as follows 
	\begin{align*}
		Q^{st} &= \left( \sum_{\alpha \in \Phi_{+}} e^{\alpha}_{(-1)}(\varphi^{*}_{\alpha})_{(-1)} - \frac{1}{2} \sum_{\alpha,\beta,\gamma} c_{\gamma}^{\alpha,\beta}(\varphi^{*}_{\alpha})_{(-1)}(\varphi^{*}_{\beta})_{(-1)}(\varphi_{\gamma})_{(-1)}\right)\vac,\\
		Q^{\chi} &= \sum_{\alpha \in \Phi_{+}} \chi(e^{\alpha})\varphi^{*}_{\alpha} = \sum_{i=1}^{r} (\varphi^{*}_{\alpha_{i}})_{(-1)}\vac,
	\end{align*}
	and set 
	$$
	Q = Q^{st} + Q^{\chi}.
	$$
	Let 
	\begin{align*}
		\d^{st} = Q^{st}_{(0)},\quad \d^{\chi} = Q^{\chi}_{(0)}, \quad \d = \d^{st} + \d^{\chi}.
	\end{align*}
	Because $\d$ is a mode of the vertex algebra $C(V)$ it acts on $C(M)$, and it is clear from its definition that it is of degree $1$, \textit{i.e.}, for all $c \in \Z$ we have 
	$$
	\d : M \otimes \bigwedge^{\frac{\infty}{2}+c}(\Ln_{+}) \lra M \otimes \bigwedge^{\frac{\infty}{2}+c+1}(\Ln_{+}). 
	$$
	Moreover 
	$
	\d \circ \d = 0
	$	
	so we can define 
	$
	H^{*}_{DS}(M)
	$
	to be the total cohomology of the complex $(C(M),\d)$. 
	\begin{definition}
		\label{definition w algebra}
		The vertex algebra
		$$
		W^{\kappa}(\Lg) = H^{*}_{DS}(V^{\kappa}(\Lg))
		$$
		is called the universal affine $\mathcal{W}$-algebra of level $\kappa$ for the Lie algebra $\Lg$.
	\end{definition}
	Assume we are given a vertex algebra $V$ together with a morphism 
	$$
	V^{\kappa}(\Lg) \lra V
	$$
	and fix a $V$-module $M$. Then $H^{*}_{DS}(V)$ is naturally a vertex algebra such that we have a morphism of vertex algebras 
	$$
	W^{\kappa}(\Lg) \lra H^{*}_{DS}(V)
	$$
	and moreover, $H^{*}_{DS}(M)$ is naturally a $H^{*}_{DS}(V)$-module. So in fact this procedure defines a functor 
	$$
	H^{*}_{DS} : V\mathrm{-Mod} \lra H^{*}_{DS}(V)\mathrm{-Mod}.
	$$
	
	\begin{theoreme}[Theorem~15.1.9 of \cite{FrenkelBenZvi2004}, Proposition~2 of \cite{frenkelgaitsgoryweylmodules} and Theorem~4.15 of \cite{Arakawa2012}]
		\label{theorem vanishing of cohomology on kazhdan lusztig category}
		Let $\Lg$ be a simple Lie algebra, $\kappa$ be any level and $V^{\kappa}(\Lg)$ the universal affine vertex algebra, then 
		\begin{align*}
			H^{i}_{DS}(V^{\kappa}(\Lg)) &= 0 \textrm{ if $i \neq 0$,}\\
			H^{0}_{DS}(V^{\kappa}(\Lg)) &= W^{\kappa}(\Lg).
		\end{align*}		
		Moreover, if $G$ is a connected reductive algebraic group with Lie algebra $\Lg$ and $M \in \Lgkappacategoryintegrable$ then for all $i \neq 0$, we have 
		$$
		H^{i}_{DS}(M) = 0.
		$$
		In particular, the functor 
		$$
		H^{0}_{DS} : \Lgkappacategoryintegrable \lra W^{\kappa}(\Lg)\mathrm{-Mod}
		$$
		is exact.
	\end{theoreme}

	\subsection{Twisted quantum Hamiltonian reduction}

	The quantum Hamiltonian reduction functor $H^{*}_{DS}$ introduced in the previous paragraph can be twisted by a coweight $\check{\mu} \in \check{P}$ to obtain a new functor
	$$
	H^{*}_{DS,\check{\mu}} : V^{\kappa}(\Lg)\mathrm{-Mod} \lra W^{\kappa}(\Lg)\mathrm{-Mod}.
	$$
	It was studied for instance in \cite{ArakawaFrenkel2019} and \cite{Arakawa2022}. We now recall the construction while collecting some facts that will be useful later.
	
	\begin{note}
		To understand the definition of the functor, it is useful to study the structure of the vertex superalgebra $\bigwedge^{\frac{\infty}{2}+\bullet}\Ln_{+}$ under the Boson-Fermion correspondence. 
	\end{note}
	
	Let 
	$$
	\La = \bigoplus_{\alpha \in \Phi_{+}} \C b_{\alpha}
	$$
	be an abelian Lie algebra with basis $(b_{\alpha})_{\alpha \in \Phi_{+}}$ and $\kappa_{\La}$ be the unique nondegenerate bilinear form on $\La$ such that 
	$$
	\kappa_{\La}(b_{\alpha},b_{\beta}) = \delta_{\alpha,\beta}
	$$
	for all positive root $\alpha,\beta$. Let $(\lambda_{\alpha})_{\alpha \in \Phi_{+}}$ be the basis of $\La^{*}$ dual to the basis $(b_{\alpha})_{\alpha \in \Phi_{+}}$. Let~$V^{\kappa_{\La}}(\La)$ be the associated Heisenberg vertex algebra. The following is easy to prove.
	\begin{proposition}
		\label{proposition morphism heisenberg to fermion}
		There exists an injective morphism of vertex superalgebras
		$$
		V^{\kappa_{\La}}(\La) \lra \bigwedge^{\frac{\infty}{2}+\bullet}\Ln_{+}
		$$ 
		uniquely determined by the condition that it maps $b_{\alpha}$ to 
		$$
		{\varphi^{*}_{\alpha}}_{(-1)}{\varphi_{\alpha}}_{(-1)}\vac
		$$
		for all $\alpha \in \Phi_{+}$.
		Moreover the charge gradation on $\bigwedge^{\frac{\infty}{2}+\bullet}\Ln_{+}$ coincides with the eigenvalues of the endomorphism
		$$
		\sum_{\alpha \in \Phi_{+}} {b_{\alpha}}_{(0)}.
		$$
	\end{proposition}
	
	This endows $\bigwedge^{\frac{\infty}{2}+\bullet}\Ln_{+}$ with the structure of a $V^{\kappa_{\La}}(\La)$-module. The following statement is an instance of a boson-fermion correspondence.
	\begin{theoreme}[Theorem 5.3.2 of \cite{FrenkelBenZvi2004}]
		\label{theorem boson fermion correspondence}
		The vertex superalgebra $\bigwedge^{\frac{\infty}{2}+\bullet}\Ln_{+}$ decomposes as a $V^{\kappa_{\La}}(\La)$-module as 
		\begin{equation}
			\label{equation boson fermion correspondence}
			\bigwedge^{\frac{\infty}{2}+\bullet}\Ln_{+} = \bigoplus_{(n_{\alpha}) \in \Z^{ \Phi_{+} }} \cF^{\kappa_{\La}}_{\sum_{\alpha \in \Phi_{+}}n_{\alpha}\lambda_{\alpha}}.
		\end{equation}
	\end{theoreme}
	
	\begin{remarque}
		The right-hand side of equation \eqref{equation boson fermion correspondence} is naturally equipped with the structure of a vertex superalgebra and is called a lattice vertex superalgebra. In fact, equation \eqref{equation boson fermion correspondence} is an equality of vertex superalgebras. We will not need it in what follows. 
	\end{remarque}
	
	We can define with respect to the morphism defined by Proposition \ref{proposition morphism heisenberg to fermion} a spectral flow group $\cSF\left(\bigwedge^{\frac{\infty}{2}+\bullet}\Ln_{+}\right)$. 
	
	\begin{lemme}
		\label{lemma computation of spectral flow group of charged fermions}
		We have an isomorphism of abelian groups
		$$
		\cSF\left(\bigwedge^{\frac{\infty}{2}+\bullet}\Ln_{+}\right) \simeq \Z^{\Phi_{+}}
		$$
		obtained by mapping each $(n_{\alpha}) \in \Z^{\Phi_{+}}$ to 
		$$
		\sum_{\alpha \in \Phi_{+}} n_{\alpha} b_{\alpha} \in \La \subset V^{\kappa_{\La}}(\La) \subset \bigwedge^{\frac{\infty}{2}+\bullet}\Ln_{+}.
		$$
	\end{lemme}
	\begin{proof}
		This is a straightforward consequence of Theorem \ref{theorem boson fermion correspondence}.
	\end{proof}
	The morphism of vertex algebras 
	$$
	V^{\kappa}(\Lh) \otimes V^{\kappa_{\La}}(\La) \lra C(V^{\kappa}(\Lg))
	$$
	allows us to consider the spectral flow group $\cSF(C(V^{\kappa}(\Lg)))$.
	
	\begin{lemme}
		\label{lemma computation of spectral flow group of c v kappa g}
		The map
		$$
		(\check{\mu},(n_{\alpha})) \in \check{P} \times \Z^{\Phi_{+}} \longmapsto \check{\mu} \otimes 1 + 1 \otimes \sum_{\alpha \in \Phi_{+}} n_{\alpha} b_{\alpha} \in V^{\kappa}(\Lh) \otimes V^{\kappa_{\La}}(\La) \subset C(V^{\kappa}(\Lg))
		$$
		induces an isomorphism of abelian groups 
		$$
		\cSF(C(V^{\kappa}(\Lg))) \simeq \check{P} \times \Z^{\Phi_{+}}.
		$$
	\end{lemme}
	\begin{proof}
		This follows from Proposition \ref{proposition spectral flow group associated to affine vertex algebras}, Lemma \ref{lemma computation of spectral flow group of charged fermions} and the fact that 
		$$
		\cSF(V^{\kappa}(\Lg) \otimes \bigwedge^{\frac{\infty}{2}+\bullet}\Ln_{+}) \simeq \cSF(V^{\kappa}(\Lg)) \times \cSF (\bigwedge^{\frac{\infty}{2}+\bullet}\Ln_{+}).
		$$
	\end{proof}
	
	The canonical pairing between $\Lh$ and its dual $\Lh^{*}$ and the very definition of the group of coweights implies that the following map
	$$
	\iota : \check{\mu} \in \check{P} \longmapsto (\alpha(\check{\mu}))_{\alpha \in \Phi_{+}} \in \Z^{\Phi_{+}}
	$$
	is a well-defined group morphism. In particular, to each coweight $\check{\mu} \in \check{P}$ we can associate the element 
	$$
	(\check{\mu},-\iota(\check{\mu})) \in \check{P} \times \Z^{\Phi_{+}}
	$$
	which, under the isomorphism of Lemma \ref{lemma computation of spectral flow group of c v kappa g}, corresponds to the element 
	\begin{equation}
		\label{equation definition of the spectral flow twist for twisting the drinfeld sokolov functor}
		C(\check{\mu}) := \check{\mu} - \sum_{\alpha \in \Phi_{+}} \alpha(\check{\mu})b_{\alpha} \in \cSF(C(V^{\kappa}(\Lg))).
	\end{equation}
	
	Let $\check{\mu} \in \check{P}$ be a coweight and $M$ a $V^{\kappa}(\Lg)$-module, then
	$$
	C(M) = M \otimes \bigwedge^{\frac{\infty}{2}+\bullet}\Ln_{+}
	$$
	is naturally a $C(V^{\kappa}(\Lg))$-module. Applying a spectral flow twist of parameter $C(\check{\mu})$, we see that 
	$$
	C(\check{\mu}) \cdot C(M) = (\check{\mu} \cdot M) \otimes ( - \iota(\check{\mu}) \cdot \bigwedge^{\frac{\infty}{2}+\bullet}\Ln_{+}) 
	$$
	is again a $C(V^{\kappa}(\Lg))$-module. Recall that, as vector spaces, 
	$$
	C(\check{\mu}) \cdot C(M) = C(M)
	$$
	so we can equip it with the charge gradation and with respect to this gradation, the operator $\d$ acts with degree $1$ and squares to zero. Define 
	$
	H^{*}_{DS,\check{\mu}}(M)
	$ 
	to be the cohomology of the complex 
	$
	(C(\check{\mu}) \cdot C(M),\d).
	$
	It is naturally a $W^{\kappa}(\Lg)$-module and we thus get a well-defined functor 
	$$
	H^{*}_{DS,\check{\mu}} : V^{\kappa}(\Lg)\mathrm{-Mod} \lra W^{\kappa}(\Lg)\mathrm{-Mod}.
	$$

	\begin{note}
		Recall that $\check{P} = \cSF(V^{\kappa}(\Lg))$ and so each $\check{\mu} \in \check{P}$ defines an invertible endofunctor
		$$
		\check{\mu} : M \in V^{\kappa}(\Lg)-\textrm{Mod} \lra \check{\mu}\cdot M \in V^{\kappa}(\Lg)-\textrm{Mod}. 
		$$
		So the first way of twisting the construction is to consider the composition of the functors 
		$$
		V^{\kappa}(\Lg)-\textrm{Mod} \overset{\check{\mu}}{\lra} V^{\kappa}(\Lg)-\textrm{Mod} \overset{H^{*}_{DS}}{\lra} W^{\kappa}(\Lg)-\textrm{Mod}.
		$$
		We now turn to the other construction. Start again with a $V^{\kappa}(\Lg)$-module $M$ and consider the $C(V^{\kappa}(\Lg))$-module $C(M)$. The vertex algebra $C(V^{\kappa}(\Lg))=V^{\kappa}(\Lg) \otimes \bigwedge^{\frac{\infty}{2}+\bullet}\Ln_{+}$ has spectral flow group 
		$$
		\cSF(C(V^{\kappa}(\Lg))) = \cSF(V^{\kappa}(\Lg)) \times \cSF\left(\bigwedge^{\frac{\infty}{2}+\bullet}\Ln_{+}\right) = \check{P} \times \Z^{\dim(\Ln_{+})}.
		$$
		More explicitly if for each $\alpha \in \Delta_{>0}$ we set
		$$
		b_{\alpha}(z) = :\varphi_{\alpha}^{*}(z)\varphi_{\alpha}(z):.
		$$
		Then for all $(\check{\mu},n_{\alpha}) \in \check{P} \times \Z^{\Delta_{>0}}$ the corresponding element of $\cSF(C(V^{\kappa}(\Lg)))$ is given by 
		$$
		\check{\mu} \otimes \vac + \vac \otimes \sum_{\alpha \in \Delta_{>0}}n_{\alpha}b_{\alpha} = \check{\mu} + \sum_{\alpha \in \Delta_{>0}}n_{\alpha}b_{\alpha}.
		$$ 
		So to each $\check{\mu}\in \check{P}$ we can associated the following spectral flow parameter $C(\check{\mu})$ of $\cSF(C(V^{\kappa}(\Lg)))$ :
		$$
		C(\check{\mu}) = \check{\mu} - \sum_{\alpha \in \Delta_{>0}}\alpha(\check{\mu})b_{\alpha}.
		$$
		Consider the module $C(\check{\mu}) \cdot C(M) \in C(V^{\kappa}(\Lg))-\textrm{Mod}$, it looks like
		$$
		C(\check{\mu}) \cdot C(M) = (\check{\mu} \cdot M) \otimes \left(-\sum_{\alpha \in \Delta_{>0}}\alpha(\check{\mu})b_{\alpha}\right) \cdot \bigwedge^{\frac{\infty}{2}+\bullet}\Ln_{+} 
		$$
		Clearly by definition of spectral flow twists, as a vector space $C(\check{\mu}) \cdot C(M)$ is nothing but $C(M)$, so it is again a $\Z$-graded vector space graded by the total charge. To understand how the action of the BRST operator is modified we need to compute a little bit.
		
		Recall that we have for each $\alpha,\beta \in \Delta_{>0}$ 
		\begin{align*}
			[{b_{\alpha}}_{\lambda}\varphi_{\beta}] &= -\delta_{\alpha,\beta}\varphi_{\beta}, \\
			[{b_{\alpha}}_{\lambda}\varphi_{\beta}^{*}] &= \delta_{\alpha,\beta}\varphi_{\beta}^{*}, \\
			[{\check{\mu}}_\lambda e^{\alpha}] &= \alpha(\check{\mu})e^{\alpha}. 
		\end{align*}
		So the following equalities hold in $C(V^{\kappa}(\Lg))[z,z^{-1}]$ (\textcolor{red}{careful the signs here are opposite to the paper on Urod algebras but I think I am right})
		\begin{align*}
			\Delta(C(\check{\mu})) \cdot \varphi_{\beta} &= z^{\beta(\check{\mu})}\varphi_{\beta}, \\
			\Delta(C(\check{\mu})) \cdot \varphi^{*}_{\beta} &= z^{-\beta(\check{\mu})}\varphi^{*}_{\beta}, \\
			\Delta(C(\check{\mu})) \cdot e^{\alpha} &= z^{\alpha(\check{\mu})}e^{\alpha}.
		\end{align*}
		\textcolor{blue}{I am indeed correct but to match the litterature we modify the spectral flow twist by adding a minus (see the subsection I added about spectral flow twist in the litterature, it is not deeper than $ \sigma \cdot M = (\sigma^{-1})^{*}M)$ to get the correct functorialities, so in fact :}
		
		The following equalities hold in $C(V^{\kappa}(\Lg))[z,z^{-1}]$ :
		\begin{align*}
			\Delta(C(\check{\mu})) \cdot \varphi_{\beta} &= z^{-\beta(\check{\mu})}\varphi_{\beta}, \\
			\Delta(C(\check{\mu})) \cdot \varphi^{*}_{\beta} &= z^{\beta(\check{\mu})}\varphi^{*}_{\beta}, \\
			\Delta(C(\check{\mu})) \cdot e^{\alpha} &= z^{-\alpha(\check{\mu})}e^{\alpha}.
		\end{align*}
		So it is easy to see that the standard term $d^{st}$ remains unchanged but that $d^{\chi} = \sum_{\alpha \in \Delta_{>0}}\chi(e^{\alpha})\varphi_{\alpha}^{*}(1)$\footnote{Careful with the normalisation of $\varphi^{*}$ that explains the $1$ and not the zero here.} acts on $C(M)$ as :
		$$
		\sum_{\alpha \in \Delta_{>0}}\chi(e^{\alpha})\varphi_{\alpha}^{*}(\alpha(\check{\mu}) + 1).
		$$
		The operator on the complex $C(M)$ we obtain after spectral flow is denoted by $d_{\check{\mu}}$ in Arakawa--Frenkel. \aref
		
		As $\C$-vector spaces, $C(\check{\mu}) \cdot C(M)$ is $C(M)$ and as such it inherits its charge grading. The operator $d$ (that now acts as $d_{\check{\mu}}$) is still of charge $1$. That makes $(C(\check{\mu}) \cdot C(M),d_{\check{\mu}})$ a $\Z$-graded differential module over $C(V^{\kappa}(\Lg))$ (\textcolor{red}{Spell out what this means someday...}) and so we see that $W^{\kappa}(\Lg)$ acts on the cohomology of the complex which we denote by $H^{*}_{DS,\check{\mu}}(M)$. Again for each $\check{\mu} \in \check{P}$, that defines a functor :
		$$
		H^{*}_{DS,\check{\mu}} : V^{\kappa}(\Lg)-\textrm{Mod} \lra W^{\kappa}(\Lg)-\textrm{Mod}.
		$$
		\textcolor{blue}{Should be written differentily = as an isomorphism of functors. Also add $f$ principal in the definition}
	\end{note}

	\subsection{Compatibility with spectral flow}	
	The twisted quantum Hamiltonian reduction functors and the untwisted one are related by the following result:
	
	\begin{theoreme}
		\label{theorem spectral flow drinfeld sokolov is the same as twisted up to cohomological shift}
		Let $M$ be a $V^{\kappa}(\Lg)$-module and $\check{\mu} \in \check{P}$. Then there exists an isomorphism of $W^{\kappa}(\Lg)$-modules
		$$
		H^{*}_{DS}(\check{\mu}\cdot M) \simeq H^{*-2\rho(\check{\mu})}_{DS,\check{\mu}}(M)
		$$
		where $\rho = \frac{1}{2}\sum_{\alpha \in \Phi_{+}} \alpha$.
	\end{theoreme}
	
	The rest of this section is devoted to the proof of this theorem. The only difference between the two functors 
	$$
	H^{*}_{DS}(\check{\mu}\cdot -) \textrm{ and } H^{*}_{DS,\check{\mu}}(-)
	$$
	is a spectral flow twist on the fermionic part but, in view of Theorem \ref{theorem boson fermion correspondence}, this only amounts to shifting the charge gradation and hence the cohomological degree (see~§9.1.1 of \cite{CreutzigGaiotto2020}). We now give the details.
	\begin{lemme}
		\label{lemma spectral flow twist of charged fermion}
		Let $(m_{\alpha}) \in \cSF\left(\bigwedge^{\frac{\infty}{2}+\bullet}\Ln_{+}\right)$. Then we have an isomorphism 
		\begin{equation}
			\label{equation shift for charged fermions}
			S_{(m_{\alpha})} : 
			\bigwedge^{\frac{\infty}{2}+\bullet}\Ln_{+} \lra (m_{\alpha}) \cdot \bigwedge^{\frac{\infty}{2}+\bullet}\Ln_{+} 
		\end{equation}
		of $\bigwedge^{\frac{\infty}{2}+\bullet}\Ln_{+}$-modules. As vector spaces we have the equality 
		$$
		(m_{\alpha}) \cdot \bigwedge^{\frac{\infty}{2}+\bullet}\Ln_{+} = \bigwedge^{\frac{\infty}{2}+\bullet}\Ln_{+} 
		$$
		and, endowing $(m_{\alpha}) \cdot \bigwedge^{\frac{\infty}{2}+\bullet}\Ln_{+}$ with the charge gradation of the right-hand side, the isomorphism $S_{(m_{\alpha})}$ is of charge $\sum_{\alpha \in \Phi_{+}}m_{\alpha}$.
		
	\end{lemme}
	\begin{proof}
		Theorem \ref{theorem boson fermion correspondence} shows that as $V^{\kappa_{\La}}(\La)$-modules we have the equality 
		$$
		\bigwedge^{\frac{\infty}{2}+\bullet}\Ln_{+} = \bigoplus_{(n_{\alpha}) \in \Z^{ \Phi_{+} }} \cF^{\kappa_{\La}}_{\sum_{\alpha \in \Phi_{+}}n_{\alpha}\lambda_{\alpha}}.
		$$
		By the computation of the spectral flow group in Lemma \ref{lemma computation of spectral flow group of charged fermions}, the element $(n_{\alpha}) \in \Z^{\Phi_{+}}$ is associated to the element 
		$$
		\sum_{\alpha \in \Phi_{+}} n_{\alpha}b_{\alpha} \in V^{\kappa_{\La}}(\La)
		$$
		and so, by Proposition \ref{proposition spectral flow of Fock spaces}, the $\bigwedge^{\frac{\infty}{2}+\bullet}\Ln_{+}$-module 
		$$
		(m_{\alpha}) \cdot \bigwedge^{\frac{\infty}{2}+\bullet}\Ln_{+}
		$$
		decomposes as a $V^{\kappa_{\La}}(\La)$-module as 
		\begin{align*}
			(m_\alpha) \cdot \bigoplus_{(n_{\alpha}) \in \Z^{ \Phi_{+} }} \cF^{\kappa_{\La}}_{\sum_{\alpha \in \Phi_{+}}n_{\alpha}\lambda_{\alpha}}.
			&=
			\bigoplus_{(n_{\alpha}) \in \Z^{ \Phi_{+} }} (m_\alpha)\cdot \cF^{\kappa_{\La}}_{\sum_{\alpha \in \Phi_{+}}n_{\alpha}\lambda_{\alpha}} \\
			&=
			\bigoplus_{(n_{\alpha}) \in \Z^{ \Phi_{+} }} \cF^{\kappa_{\La}}_{\sum_{\alpha \in \Phi_{+}}(n_{\alpha}-m_{\alpha})\lambda_{\alpha}}.
		\end{align*}
		In particular $(m_{\alpha}) \cdot \bigwedge^{\frac{\infty}{2}+\bullet}\Ln_{+}$ contains the vacuum Fock space $\cF^{\kappa_{\La}}_{0} = V^{\kappa_{\La}}(\La)$. It is generated by the vector $|(m_{\alpha})\rangle$. By Lemma \ref{lemma vacuum-like vectors implies there exists a morphism from the vertex algebra} we see that there exists a unique morphism of $\bigwedge^{\frac{\infty}{2}+\bullet}\Ln_{+}$-modules 
		$$
		S_{(m_{\alpha})} : \bigwedge^{\frac{\infty}{2}+\bullet}\Ln_{+} \lra (m_{\alpha}) \cdot \bigwedge^{\frac{\infty}{2}+\bullet}\Ln_{+}
		$$
		mapping $\vac$ to $|\sum_{\alpha\in \Phi_{+}}m_{\alpha}\lambda_{\alpha}\rangle$. In particular this morphism is nonzero. So by the simplicity of $\bigwedge^{\frac{\infty}{2}+\bullet}\Ln_{+}$ it is an isomorphism. The assertion on the charge of $S_{(m_{\alpha})}$ follows from the observation that the charge of the Fock space 
		$$
		\cF^{\kappa_{\La}}_{\sum_{\alpha \in \Phi_{+}}n_{\alpha}\lambda_{\alpha}} \subset \bigwedge^{\frac{\infty}{2}+\bullet}\Ln_{+} 
		$$
		is equal to $\sum_{\alpha \in \Phi_{+}}n_{\alpha}$ and the fact that $S_{(m_{\alpha})}$ is in particular a morphism of $V^{\kappa_{\La}}(\La)$-modules.
	\end{proof}
	
	Let $M$ be a $V^{\kappa}(\Lg)$-module and $\check{\mu} \in \check{P}$, we have an isomorphism of $C(V^{\kappa}(\Lg))$-modules
	$$
	(\id \otimes S_{(-\alpha(\check{\mu}))}): (\check{\mu},0) \cdot C(M) \lra (\check{\mu},-\iota(\check{\mu})) \cdot C(M).
	$$
	This isomorphism is an isomorphism of complexes that shifts the gradation by 
	$$
	\sum_{\alpha \in \Phi_{+}} -\alpha(\check{\mu}) = -(\sum_{\alpha \in \Phi_{+}} \alpha)(\check{\mu}) = -2 \rho(\check{\mu})
	$$
	as explained in Lemma \ref{lemma spectral flow twist of charged fermion}. It induces in cohomology the isomorphism 
	$$
	H^{*}_{DS}(\check{\mu}\cdot M) \simeq H^{*-2\rho(\check{\mu})}_{DS,\check{\mu}}(M)
	$$
	of $W^{\kappa}(\Lg)$-modules, and that finishes the proof of Theorem \ref{theorem spectral flow drinfeld sokolov is the same as twisted up to cohomological shift}.

	\section{Building modules on $\ewalg$}
	\label{section building modules on ewalg}
	
	\subsection{Construction and basic properties of $\ewalg$}
	
	We now introduce the equivariant affine $\mathcal{W}$-algebras following \cite{Arakawa2018}. Assume that $G$ is a simple algebraic group. Recall that we have morphisms of vertex algebras 
	$$
	\xymatrix{
		& V^{\kappa}(\Lg) \otimes V^{\kappa^{*}}(\Lg) \ar^{\pi_{L} \otimes \pi_{R}}[d] &\\
		V^{\kappa}(\Lg) \ar^{\pi_{L}}[r] \ar[ru]& \cdo & V^{\kappa^{*}}(\Lg) \ar^{\pi_{R}}[l] \ar[lu]
	}
	$$
	so we can perform the Hamiltonian reduction with respect to $V^{\kappa}(\Lg)$. We obtain a vertex algebra 
	$$
	\ewalg = H^{*}_{DS}(\cdo),
	$$
	called the equivariant affine $\mathcal{W}$-algebra at level $\kappa$ of the group $G$. It comes equipped with morphisms of vertex algebras 
	$$
	\xymatrix{
		& W^{\kappa}(\Lg) \otimes V^{\kappa^{*}}(\Lg) \ar^{\pi_{L} \otimes \pi_{R}}[d] &\\
		W^{\kappa}(\Lg) \ar^{\pi_{L}}[r] \ar[ru]& \ewalg & V^{\kappa^{*}}(\Lg). \ar^{\pi_{R}}[l] \ar[lu]}
	$$
	In particular, any $\ewalg$-module is naturally a $W^{\kappa}(\Lg) \otimes V^{\kappa^{*}}(\Lg)$-module by restriction \textit{i.e.}, it carries two commuting actions of $W^{\kappa}(\Lg)$ and $V^{\kappa^{*}}(\Lg)$. For all $i \neq 0$ we have  
	$$
	H^{i}_{DS}(\cdo) = 0 
	$$
	and hence 
	$$
	H^{0}_{DS}(\cdo) = \ewalg. 
	$$
	as a direct consequence of Theorem \ref{theorem vanishing of cohomology on kazhdan lusztig category} because $\cdo \in \Lgkappacategoryintegrable$.
	
	It was proven by Arakawa that the associated scheme of $\ewalg$ is the principal equivariant Slodowy slice $\mathcal{S}_{G}$ which is smooth, reduced and symplectic. Using the criterion of \cite{ArakawaMoreau2021} we have the following result:
	\begin{lemme}[Discussion following Corollary 6.2 of \cite{Arakawa2018}]
		\label{lemma simplicity of the equivariant walgebra}
		The vertex algebra $\ewalg$ is simple.
	\end{lemme}
	
	To understand the structure of $\ewalg$ as a $W^{\kappa}(\Lg)\otimes V^{\kappa^{*}}(\Lg)$-module one needs to understand the behaviour of Hamiltonian reduction on the Kazhdan--Lusztig category:
	\begin{proposition}[Corollary 9.1.6. of \cite{Arakawa2007}]
		\label{proposition computation of the Hamiltonian reduction of weyl modules}
		Assume that $k \in \C \setminus \Q$ and let $\lambda \in P_{+}$ then
		$$
		H^{0}_{DS}(V_{\lambda}^{\kappa}):= T_{\lambda,0}^{\kappa}		
		$$
		is a simple $W^{\kappa}(\Lg)$-module.
	\end{proposition}
	
	\begin{remarque}
		In fact, $T_{\lambda,0}^{\kappa}$ is a highest weight module of $W^{\kappa}(\Lg)$. See §5 of~\cite{Arakawa2007} for what it means to be of highest weight in this context.
	\end{remarque}

	\begin{lemme}
		\label{lemma chiral peter weyl decomposition for equivariant walgebra}
		We have a decomposition 
		$$
		\ewalg = \bigoplus_{\lambda \in X^{*}(T)_{+}} T^{\kappa}_{\lambda,0} \otimes V^{\kappa^{*}}_{-\omega_{0}\lambda}
		$$
		as $W^{\kappa}(\Lg) \otimes V^{\kappa^{*}}(\Lg)$-modules.
	\end{lemme}
	\begin{proof}
		Applying the quantum Hamiltonian reduction functor with respect to the left action to the decomposition of Theorem \ref{theorem chiral Peter-Weyl} we obtain that, as $W^{\kappa}(\Lg) \otimes V^{\kappa^{*}}(\Lg)$-modules, 
		$$
		\ewalg = \bigoplus_{\lambda \in X^{*}(T)_{+}} H^{0}_{DS}(V^{\kappa}_{\lambda}) \otimes V^{\kappa^{*}}_{-\omega_{0}\lambda}
		$$ 
		and the result now follows from Proposition \ref{proposition computation of the Hamiltonian reduction of weyl modules}.
	\end{proof}
	As explained before we have a morphism of vertex algebras 
	$$
	V^{\kappa^{*}}(\Lg) \lra \ewalg.
	$$
	It allows us to define the spectral flow group $\cSF(\ewalg)$ of $\ewalg$. The following is an easy consequence of Lemma \ref{lemma chiral peter weyl decomposition for equivariant walgebra}: 
	\begin{corollaire}
		\label{corollaire computation of spectral flow group of eqwalg}
		We have the equality of abelian groups 
		$$
		\cSF(\ewalg) = X_{*}(T).
		$$
	\end{corollaire}

	The following property is a very useful trick to show the simplicity of certain modules. It should be compared with Lemma \ref{lemma weights that can appear in fock spaces in cdo for the torus}.
	\begin{proposition}
		\label{proposition stability of highest weight starting from the vacuum}
		Let $M$ be a $\ewalg$-module and assume that the module $M$ contains $V^{\kappa^{*}}(\Lg)$ as a  $V^{\kappa^{*}}(\Lg)$-submodule. Then for all $\lambda \in X^{*}(T)_{+}$, $M$ contains $V^{\kappa^{*}}_{-\omega_{0}\lambda}$ as a $V^{\kappa^{*}}(\Lg)$-submodule.
	\end{proposition}
	\begin{proof}
		Let $|-,0\rangle \in M$ denote a highest weight vector of weight $0$ that spans the given copy of $V^{\kappa^{*}}(\Lg)$. 
		In view of the decomposition of Lemma \ref{lemma chiral peter weyl decomposition for equivariant walgebra}, we can choose for each~$\lambda \in X^{*}(T)_{+}$ an element $|-,-\omega_{0}\lambda\rangle \in T^{\kappa}_{\lambda,0} \otimes V^{\kappa^{*}}_{-\omega_{0}\lambda}$ that is a highest weight vector of weight $-\omega_{0}\lambda$. Because $\ewalg$ is simple by Lemma \ref{lemma simplicity of the equivariant walgebra} it follows from Proposition \ref{proposition annihilator is an ideal} that there exists a unique $m \in \Z$ such that 
		\begin{align*}
			&|-,-\omega_{0}\lambda\rangle_{(m)}|-,0\rangle \neq 0, \\
			&|-,-\omega_{0}\lambda\rangle_{(k)}|-,0\rangle = 0 \textrm{ for all $k>m$.}
		\end{align*}
		It is straightforward to check that $|-,-\omega_{0}\lambda\rangle_{(m)}|-,0\rangle$ is a highest weight vector for $\Lgkappastar$ of weight $ -\omega_{0}\lambda$. One then checks that Proposition \ref{proposition Weyl modules are highest weight modules} applies.		
	\end{proof}

	\subsection{Building simple modules}
	Recall from §\ref{section reminders on quantum Hamiltonian reduction} that we have a well-defined functor
	$$
	H^{*}_{DS} : \cdo\mathrm{-Mod} \lra \ewalg\mathrm{-Mod}.
	$$ 
	The techniques developed in §\ref{section spectral flow group of cdo} show that for each $(\gamma,x) \in \cSF(\cdo)$ we have a simple $\cdo$-module $((\gamma,x) \cdot \cdo)_{\gamma \in X_{*}(T),x\in \check{P}}$ as a direct application of Theorem~\ref{theorem properties of spectral flow} together with Proposition \ref{proposition cdo is a simple vertex algebra}.

	The goal now is, given $\gamma \in X_{*}(T)$ and $x \in \check{P}$, to understand the module
	$$
	H^{*}_{DS}((\gamma,x) \cdot \cdo) \in \ewalg\mathrm{-Mod}.
	$$
	As $V^{\kappa}(\Lg) \otimes V^{\kappa^{*}}(\Lg)$-modules we have the equality  
	$$
	(\gamma,x) \cdot \cdo =  \bigoplus_{\lambda \in X^{*}(T)_{+}} (\gamma+x) \cdot V^{\kappa}_{\lambda} \otimes (\omega_{0}^{T}x)\cdot V^{\kappa^{*}}_{-\omega_{0}\lambda}
	$$
	so that, as $W^{\kappa}(\Lg) \otimes V^{\kappa^{*}}(\Lg)$-modules, we have the equality 
	\begin{equation}
		\label{equation decomposition of reduction of spectral flow of cdo}
		H^{*}_{DS}((\gamma,x) \cdot \cdo) = \bigoplus_{\lambda \in X^{*}(T)_{+}} H^{*}_{DS}((\gamma+x) \cdot V_{\lambda}^{\kappa}) \otimes (\omega_{0}^{T}x)\cdot V^{\kappa^{*}}_{-\omega_{0}\lambda}.
	\end{equation}
	
	To understand the structure of these modules as $W^{\kappa}(\Lg)\otimes V^{\kappa^{*}}(\Lg)$-module one needs to understand the behaviour of twisted Hamiltonian reduction on the Kazhdan--Lusztig category. The following was conjectured by Creutzig and Gaiotto (see Conjecture 1.5 of \cite{CreutzigGaiotto2020}) and is now a theorem of Arakawa and Frenkel (\cite{ArakawaFrenkel2019}).
	\begin{lemme}
		\label{lemma twisted Hamiltonian reduction of weyl modules}
		Let $\lambda \in P_{+}$ and $\check{\mu} \in \check{P} = \cSF(V^{\kappa}(\Lg))$. 
		Then if $\check{\mu} \notin \check{P}_{+}$ we have 
		$$
		H^{*}_{DS}(\check{\mu}\cdot V_{\lambda}^{\kappa}) = 0.
		$$
		If $\check{\mu} \in \check{P}_{+}$, then 
		\begin{align*}
			&H^{i}_{DS}(\check{\mu} \cdot V^{\kappa}_{\lambda}) = 0 \textrm{ if $i \neq 2\rho(\check{\mu}) $},\\
			\intertext{and we have that}
			&H^{2\rho(\check{\mu})}_{DS}(\check{\mu} \cdot V^{\kappa}_{\lambda}) := T^{\kappa}_{\lambda,\check{\mu}}
		\end{align*}
		is a simple $W^{\kappa}(\Lg)$-module.
	\end{lemme}
	\begin{proof}
		In view of Theorem \ref{theorem spectral flow drinfeld sokolov is the same as twisted up to cohomological shift} we have 
		$$
		H^{*}_{DS}(\check{\mu}\cdot V_{\lambda}^{\kappa}) = H_{DS,\check{\mu}}^{*-2\rho(\check{\mu})}(V_{\lambda}^{\kappa})
		$$
		and the result then follows from Theorems 2.1 and 4.3 of \cite{ArakawaFrenkel2019}.
	\end{proof}

	\begin{remarque}
		Again, $T_{\lambda,\check{\mu}}^{\kappa}$ is a highest weight module of $W^{\kappa}(\Lg)$. This family of $W^{\kappa}(\Lg)$-modules enjoys remarkable duality properties with respect to the Feigin--Frenkel duality of $\mathcal{W}$-algebras (see Theorem 2.2 of \cite{ArakawaFrenkel2019}) that were independently conjectured by Creutzig--Gaiotto and Gaitsgory.  
	\end{remarque}
	
	\begin{proposition}
		\label{proposition decomposition of reduction of specrtal flow for cdo}
		Let $\gamma \in X_{*}(T)$. We have the equalities as $W^{\kappa}(\Lg) \otimes V^{\kappa^{*}}(\Lg)$-modules
		\begin{align*}
			&H^{*}_{DS}(\gamma \cdot \cdo) = 0, &&\textrm{ if $\gamma \notin X_{*}(T)_{+}$,}\\
			&H^{*}_{DS}(\gamma \cdot \cdo) = H^{2\rho(\gamma)}_{DS}(\gamma \cdot \cdo) = \bigoplus_{\lambda \in X^{*}(T)_{+}} T^{\kappa}_{\lambda,\gamma} \otimes V^{\kappa^{*}}_{-\omega_{0}\lambda}, &&\textrm{ if $\gamma \in X_{*}(T)_{+}$}.
		\end{align*}
		Let $x \in \check{P}$, more generally we have the equalities as $W^{\kappa}(\Lg) \otimes V^{\kappa^{*}}(\Lg)$-modules 
		\begin{align*}
			&H^{*}_{DS}((\gamma,x) \cdot \cdo) = 0 &&\textrm{ if $\gamma+x \notin \check{P}_{+}$,}\\
			&H^{n}_{DS}((\gamma,x) \cdot \cdo) = 0 &&\textrm{ if $n \neq 2\rho(\gamma+x)$}.\\
			&H^{2\rho(\gamma+x)}_{DS}((\gamma,x) \cdot \cdo) = \bigoplus_{\lambda \in X^{*}(T)_{+}} T^{\kappa}_{\lambda, \gamma+x} \otimes (\omega_{0}^{T}x)\cdot V^{\kappa^{*}}_{-\omega_{0}\lambda} &&\textrm{ if $\gamma+x \in \check{P}_{+}$}.
		\end{align*}
	\end{proposition}
	\begin{proof}
		This follows directly from equation \eqref{equation decomposition of reduction of spectral flow of cdo} and Lemma \ref{lemma twisted Hamiltonian reduction of weyl modules}.
	\end{proof}
	
	Define 
	$$
	\ewalg\mathrm{-Mod}^{\jetinf G}
	$$
	to be the full subcategory of $\ewalg$-modules such that they belong to the Kazhdan--Lusztig category $V^{\kappa^{*}}(\Lg)\mathrm{-Mod}^{\jetinf G}$ when seen as $V^{\kappa^{*}}(\Lg)$-modules by restriction. 
	
	We can now state a conjecture which is a version of the fundamental local equivalence of the quantum geometric Langlands program in our language:
	
	\begin{conjecture}
		\label{conjecture FLE}
		For generic $\kappa$, the category $\ewalg \mathrm{-Mod}^{\jetinf G}$ is semisimple with simple objects given by the $H^{2\rho(\gamma)}_{DS}(\gamma \cdot \cdo)$ for $\gamma \in X_{*}(T)_{+}$. That is, we have an equivalence of abelian categories 
		$$
		\ewalg \mathrm{-Mod}^{\jetinf G} \simeq \check{G}\mathrm{-Mod}
		$$
		that sends for all $\gamma \in X_{*}(T)_{+}$ the simple $\ewalg$-module $H_{DS}^{2\rho(\gamma)}(\gamma \cdot \mathcal{D}_{G}^{\kappa})$ to the simple $\check{G}$-module $V_{\gamma}$.
	\end{conjecture}
	
	\begin{theoreme}
		\label{theorem simplicity of modules for equivariant w algebra}
		The family of $\ewalg$-modules $(H_{DS}^{2\rho(\gamma)}(\gamma \cdot \cdo))_{\gamma \in X_{*}(T)_{+}}$ consists of simple, pairwise non-isomorphic objects of $\ewalg\mathrm{-Mod}^{\jetinf G}$.
	\end{theoreme}
	\begin{proof}	
		It is clear from Proposition \ref{proposition decomposition of reduction of specrtal flow for cdo} that these modules belong to $\ewalg\mathrm{-Mod}^{\jetinf G}$ and are pairwise non-isomorphic. Let $M$ be a $\ewalg$-submodule of $H^{2\rho(\gamma)}_{DS}(\gamma \cdot \cdo)$ and define
		$$
		M' = H^{2\rho(\gamma)}_{DS}(\gamma \cdot \cdo)/M.
		$$	
		We want to show that one of $M$ or $M'$ is the zero module. For all $\lambda \in X^{*}(T)_{+}$ the $W^{\kappa}(\Lg) \otimes V^{\kappa^{*}}(\Lg)$-module 
		$
		T^{\kappa}_{\lambda,\gamma} \otimes V^{\kappa^{*}}_{-\omega_{0}\lambda}
		$
		is simple by Theorem \ref{theorem Kazhdan--Lusztig categories for irrational levels} and Lemma~\ref{lemma twisted Hamiltonian reduction of weyl modules}. 
		
		Proposition \ref{proposition decomposition of reduction of specrtal flow for cdo} expresses $H^{2\rho(\gamma)}_{DS}(\gamma \cdot \cdo)$ as a direct sum of simple $W^{\kappa}(\Lg) \otimes V^{\kappa^{*}}(\Lg)$-modules each appearing with multiplicity one. It follows that there exists a subset $S~\subset X^{*}(T)_{+}$ such that 
		\begin{equation}
			\label{equation decomposition M}
			M = \bigoplus_{\lambda \in S} T^{\kappa}_{\lambda,\gamma} \otimes V^{\kappa^{*}}_{-\omega_{0}\lambda}
		\end{equation}
		and 
		\begin{equation}
			\label{equation decomposition cokernel M}
			M' = \bigoplus_{\lambda \in X^{*}(T)_{+} \setminus S} T^{\kappa}_{\lambda,\gamma} \otimes V^{\kappa^{*}}_{-\omega_{0}\lambda}
		\end{equation}
		as $W^{\kappa}(\Lg) \otimes V^{\kappa^{*}}(\Lg)$-modules. 
		It is obvious that $0 \in S$ or $0 \in X^{*}(T)_{+}\setminus S$. Let us focus on the second case, this means that 
		$$
		T^{\kappa}_{0,\gamma}\otimes V^{\kappa^{*}}_{0} = T^{\kappa}_{0,\gamma}\otimes V^{\kappa^{*}}(\Lg) \subset M'.
		$$
		Proposition \ref{proposition stability of highest weight starting from the vacuum} applies and shows that for all $\lambda \in X^{*}(T)_{+}$ there exists a highest weight vector for $\Lgkappastar$ of weight $-\omega_{0}\lambda$. Using equation \eqref{equation decomposition cokernel M} it means that $-\omega_{0}\lambda \in X^{*}(T)_{+}\setminus S$ for all $\lambda \in X^{*}(T)_{+}$. This proves that $S = \emptyset$ and hence $M$ is the zero module. The other case is treated similarly.
	\end{proof}
	
	In fact, we can say a bit more, the situation is especially nice when we assume the group $G$ to be simple and of adjoint type, that is, the center of $G$ is trivial, or equivalently, $X_{*}(T) = \check{P}$. A typical example is $G = \mathrm{PSL}_{2}$.

	\begin{proposition}
		Let $G$ be simple of adjoint type. For all~$(\gamma,x) \in X_{*}(T) \times \check{P} \simeq \cSF(\cdo)$ such that $\gamma + x \in X^{*}(T)_{+}$, the $\ewalg$-module $H^{2\rho(\gamma + x)}_{DS}((\gamma,x) \cdot \cdo)$ is simple.
	\end{proposition} 
	\begin{proof}
		By Corollary \ref{corollaire computation of spectral flow group of eqwalg} we have that $\cSF(\ewalg) = X_{*}(T)$. But since $G$ is of adjoint type we have $X_{*}(T) = \check{P}$ so that we can consider the $\ewalg$-module 
		$$
		(- \omega_{0}^{T}x) \cdot H^{2\rho(\gamma+x)}_{DS}((\gamma,x) \cdot \cdo).
		$$
		Clearly it decomposes as a $W^{\kappa}(\Lg) \otimes V^{\kappa^{*}}(\Lg)$-module as 
		$$
		\bigoplus_{\lambda \in X^{*}(T)_{+}} T^{\kappa}_{\lambda, \gamma+x} \otimes V^{\kappa^{*}}_{-\omega_{0}\lambda}.
		$$
		Now the proof of Theorem \ref{theorem simplicity of modules for equivariant w algebra} applies to prove that $(- \omega_{0}^{T}x) \cdot H^{2\rho(\gamma+x)}_{DS}((\gamma,x) \cdot \cdo)$ is simple and hence so is $H^{2\rho(\gamma+x)}_{DS}((\gamma,x) \cdot \cdo)$.
	\end{proof}

	\begin{remarque}
		\label{remark restriction to the simple case and generalization to reductive case}
		We made the choice when considering quantum Hamiltonian reduction to restrict our attention to the case where $\Lg$ (and hence the underlying group $G$) is simple. The generalization of those concepts and of the proof of Theorem \ref{theorem simplicity of modules for equivariant w algebra} to the full reductive case is straightforward. Namely on the abelian part this reduction procedure is trivial and on the semisimple part of $\Lg$ it works separately simple factor by simple factor. Conjecture \ref{conjecture FLE} then makes sense for any reductive group $G$. In the same spirit, the further perspectives discussed in Section \ref{section further perspectives} can be generalized to the reductive case.
	\end{remarque}

	\section{The case of an algebraic torus}
	\label{section the case of an algebraic torus}
	
	Let $T = \G_{m}^{r}$ be an algebraic torus. Its Lie algebra $\Lh$ is a finite dimensional abelian Lie algebra of dimension $r$. Let $\kappa$ be a nondegenerate level. Recall that in this case $\kappa^{*} = -\kappa$. For all $\alpha \in X^{*}(T)=X^{*}(T)_{+}$, we let $e^{\alpha} \in \cO_{T}$ be a function of highest weight $(\alpha,-\alpha)$. In this setting, Theorem \ref{theorem chiral Peter-Weyl} says that we have a decomposition
	$$
	\cD_{T}^{\kappa} = \bigoplus_{\alpha \in X^{*}(T)} \cF^{\kappa}_{\alpha} \otimes \cF^{-\kappa}_{-\alpha}
	$$
	as $V^{\kappa}(\Lh) \otimes V^{\kappa^{*}}(\Lh)$-modules.
		
	It can be deduced for instance from Theorem \ref{theorem rigidity of chiral Peter-Weyl} that $\mathcal{D}_{T}^{\kappa}$ is a half-lattice vertex algebra (in the sense of \cite{BermanDongTan2002}). We learned the theory of lattice vertex algebras from \cite{DongLiMason1997}, \cite{BermanDongTan2002} and §6.4 and §6.5 of \cite{LepowskyLi1994}. It is more convenient for our purposes to reproduce fully the results of the previously cited references. Our main input is a reformulation of certain results using the spectral flow group introduced in Section \ref{section spectral flow in the theory of vertex algebras}, while systematically keeping track of the level $\kappa$.

	\begin{lemme}
		Let $M$ be a $V^{\kappa}(\Lh)$-module. Then for all $h \in \Lh$ the formula
		$$
		E^{-}(h,z) = \exp\left(\sum_{n<0}\frac{h_{(n)}}{-n}z^{-n}\right) = \sum_{r \leqslant0} \left( \sum_{k \geqslant0} \frac{1}{k!} \sum_{\stackrel{i_{1}+\dots+i_{k}=r}{i_{1},\dots,i_{k} < 0}} \frac{h_{(i_{1})}}{-i_{1}} \dots \frac{h_{(i_{k})}}{-i_{k}} \right) z^{-r}
		$$
		defines an element of $\End(M)[[z]]$. Moreover, for all $x \in \Lh$ and $m\in \mathbb{Z}$, the following commutation relations hold 
		\begin{align}
			[x_{(m)},E^{-}(h,z)] &= 0, &&\textrm{ if $m \leqslant0$}\\
			[x_{(m)},E^{-}(h,z)] &= \kappa(x,h)z^{m}E^{-}(h,z) &&\textrm{ if $m \geqslant1$}.
		\end{align}
		
	\end{lemme}
	\begin{proof}
		The fact that this is well-defined is easy. Let 
		$$
		H^{-}(z) = \sum_{n<0}\frac{h_{(n)}}{-n}z^{-n} \in \End(M)[[z]].
		$$
		By definition, we have 
		$$
		E^{-}(h,z) = \sum_{k \geqslant0} \frac{1}{k!}H^{-}(z)^{k}.
		$$
		The case $m \leqslant0$ is straightforward, so assume that $m \geqslant1$. Then 
		$$
		[x_{(m)},H^{-}(z)] = \sum_{n<0}[x_{(m)},h_{(n)}]\frac{z^{-n}}{-n}
		= \sum_{n<0} m\delta_{m,-n}\kappa(x,h)\frac{z^{-n}}{-n}
		= \kappa(x,h)z^{m},
		$$
		hence
		\begin{align*}
			[x_{(m)},E^{-}(h,z)] = \sum_{ k \geqslant0} \frac{1}{k!} [x_{(m)},H^{-}(z)^{k}] &= \sum_{k \geqslant0} \frac{1}{k!}kH^{-}(z)^{k-1}[x_{(m)},H^{-}(z)] \\
			&= \kappa(x,h)z^{m}E^{-}(h,z).
		\end{align*}
		
	\end{proof}
	
	\begin{lemme}
		Let $M$ be a $\cD_{T}^{\kappa}$-module. In particular this is a $V^{\kappa \oplus \kappa^{*}}(\Lh \oplus \Lh)$-module and for all $\alpha \in X^{*}(T)$, set 
		$$
		A(\alpha)(z) = E^{-}((-\alpha,-\alpha),z)e^{\alpha}(z)
		$$
		and write 
		$$
		A(\alpha)(z) = \sum_{n \in \Z}A(\alpha)_{n}z^{-n-1}.
		$$
		Then $A(\alpha)(z)$ is a well-defined element of $\End(M)[[z,z^{-1}]]$ and for all $x \in \Lh$, and $m,n \in \Z$, we have:
		\begin{align}
			&[\pi_{L}(x)_{(m)},A(\alpha)_{n}] = \kappa(x,\alpha)A(\alpha)_{m+n}, &&\textrm{ if $m \leqslant0$}  \label{equation commutation action a gauche et a alpha}\\
			&[\pi_{R}(x)_{(m)},A(\alpha)_{n}] = -\kappa(x,\alpha)A(\alpha)_{m+n}, &&\textrm{ if $m \leqslant0$} \label{equation commutation action a droite et a alpha}\\
			&[\pi_{L}(x)_{(m)},A(\alpha)_{n}] = 0, &&\textrm{ if $m \geqslant1$} \label{equation commutation action a gauche et a alpha modes strictement positifs} \\
			&[\pi_{R}(x)_{(m)},A(\alpha)_{n}] = 0, &&\textrm{ if $m \geqslant1$}\label{equation commutation action a droite et a alpha modes strictement positifs}  \\
			&(-n-1)A(\alpha)_{n} = \sum_{i \geqslant0}A(\alpha)_{n-i}(\pi_{L}(\alpha)_{(i)} + \pi_{R}(\alpha)_{(i)}). \label{equation equa diff a alpha}
		\end{align}
	\end{lemme}
	\begin{proof}
		It is easy to see that $A(\alpha)(z)$ is well-defined because $E^{-}((-\alpha,-\alpha),z) \in \End(M)[[z]]$ and $e^{\alpha}(z)$ is a field. Recall that for all $m \in \Z$ the following equalities hold 
		\begin{align}
			[\pi_{L}(x)_{(m)},e^{\alpha}(z)] &= \kappa(x,\alpha)z^{m}e^{\alpha}(z), \label{equation commutation piL and ealpha}\\
			[\pi_{R}(x)_{(m)},e^{\alpha}(z)] &= -\kappa(x,\alpha)z^{m}e^{\alpha}(z). \label{equation commutation piR and ealpha}
		\end{align}
		Assume for instance that $m \leqslant0$. Then 
		\begin{align*}
			[\pi_{R}(x)_{(m)},A(\alpha)(z)] &= [\pi_{R}(x)_{(m)},E^{-}((-\alpha,-\alpha),z)]e^{\alpha}(z) + E^{-}((-\alpha,-\alpha),z)[\pi_{R}(x)_{(m)},e^{\alpha}(z)]  \\
			&= 0 - \kappa(x,\alpha)z^{m}E^{-}((-\alpha,-\alpha),z)e^{\alpha}(z) \\
			&= -\kappa(x,\alpha)z^{m}A(\alpha)(z),
		\end{align*}
		which proves the claimed equality by looking at the coefficients. The one for $\pi_{L}(x)$ is proven similarly. Now assume that $m \geqslant1$. Then 
		\begin{align*}	
			[\pi_{R}(x)_{(m)},A(\alpha)(z)] &= -\kappa(x,-\alpha)z^{m}E^{-}((-\alpha,-\alpha),z)e^{\alpha}(z) - \kappa(x,\alpha)z^{m}E^{-}((-\alpha,-\alpha),z)e^{\alpha}(z) \\
			&= 0,
		\end{align*}
		which proves the equality. Again, the one for $\pi_{L}(x)$ is proven similarly. We now compute 
		\begin{align*}
			\frac{\partial}{\partial z} A(\alpha)(z) &= \frac{\partial}{\partial z} (E^{-}((-\alpha,-\alpha),z)e^{\alpha}(z)) \\
			&= \frac{\partial}{\partial z}(E^{-}((-\alpha,-\alpha),z))e^{\alpha}(z) + E^{-}((-\alpha,-\alpha),z)\frac{\partial}{\partial z}(e^{\alpha}(z)).
		\end{align*}
		But we have 
		\begin{align*}
			\frac{\partial}{\partial z}E^{-}((-\alpha,-\alpha),z) &= (-\sum_{n<0}(\pi_{L}(\alpha)_{(n)} + \pi_{R}(\alpha)_{(n)})z^{-n-1})E^{-}((-\alpha,-\alpha),z) \\
			&= E^{-}((-\alpha,-\alpha),z)(-\sum_{n<0}(\pi_{L}(\alpha)_{(n)} + \pi_{R}(\alpha)_{(n)})z^{-n-1}),
		\end{align*}
		and 
		$$
		\frac{\partial}{\partial z}e^{\alpha}(z) = (\pi_{L}(\alpha)(z) + \pi_{R}(\alpha)(z))e^{\alpha}(z).
		$$
		Hence we have 
		\begin{align*}
			\frac{\partial}{\partial z} A(\alpha)(z) 
			&= E^{-}((-\alpha,-\alpha),z)e^{\alpha}(z)(-\sum_{n<0}(\pi_{L}(\alpha)_{(n)} + \pi_{R}(\alpha)_{(n)})z^{-n-1})\\
			&\quad+E^{-}((-\alpha,-\alpha),z)e^{\alpha}(z)(\pi_{L}(\alpha)(z) + \pi_{R}(\alpha)(z))\\
			&= A(\alpha)(z)(\sum_{n\geqslant0}(\pi_{L}(\alpha)_{(n)}+\pi_{R}(\alpha)_{(n)})z^{-n-1}),
		\end{align*}
		where the first equality follows from the fact that for all $n \in \Z$, $h \in \Lh$ and $\alpha \in X^{*}(T)$ we have 
		$$
		[\pi_{L}(h)_{(n)}+\pi_{R}(h)_{(n)},e^{\alpha}(z)] = 0
		$$
		by summing equations \eqref{equation commutation piL and ealpha} and \eqref{equation commutation piR and ealpha}.
		Hence we have proven that  
		$$
		\frac{\partial}{\partial z}A(\alpha)(z) = A(\alpha)(z)(\sum_{n\geqslant0}(\pi_{L}(\alpha)_{(n)} + \pi_{R}(\alpha)_{(n)})z^{-n-1})
		$$
		which, by looking at the coefficients, is the announced equality.
	\end{proof}

	\begin{lemme}
		\label{lemma weights that can appear in fock spaces in cdo for the torus}
		Let $M$ be a $\cD_{T}^{\kappa}$-module. Let $\lambda,\mu \in \Lh^{*}$ and assume that $M$ contains a Fock space $\cF_{\lambda}^{\kappa} \otimes \cF_{\mu}^{\kappa^{*}}$ with respect to its $\Lhkappakappastar$-module structure. Then for all $\alpha \in X^{*}(T)$, $M$ contains a Fock space of weight $(\lambda + \alpha, \mu-\alpha)$
		and
		$$
		\lambda + \mu \in \kappa(X_{*}(T)),
		$$ 
		where we think of $\kappa$ as an isomorphism 
		$
		\kappa : \Lh \lra \Lh^{*}.
		$
	\end{lemme}
	
	\begin{proof}
		Let $|\lambda,\mu\rangle$ be a highest weight vector that spans the given Fock space in $M$. Because $\cD_{T}^{\kappa}$ is a simple vertex algebra by Proposition \ref{proposition cdo is a simple vertex algebra}, Proposition \ref{proposition annihilator is an ideal} applies and shows that there exists a unique $m \in \Z$ such that
		\begin{align*}
			e^{\alpha}_{(m)}|\lambda,\mu\rangle &\neq 0, \\
			e^{\alpha}_{(k)}|\lambda,\mu\rangle &= 0 \textrm{ for all $k>m$.}
		\end{align*}
		It is easy to check that $e^{\alpha}_{(m)}|\lambda,\mu\rangle$ is a highest weight vector of weight $(\lambda+\alpha,\mu-\alpha)$, we denote it by 
		$$
		|\lambda+\alpha,\mu-\alpha\rangle.
		$$
		It is clear by its definition that $|\lambda+\alpha,\mu-\alpha\rangle$ lies in the $\cD_{T}^{\kappa}$-submodule generated by~$|\lambda,\mu\rangle$. Applying Lemma \ref{lemma conformal weights are in the same Z coset when you are monogeneous}, we see that for all $\alpha \in X^{*}(T)$, the difference of the conformal weight of $|\lambda,\mu\rangle$ and that of $|\lambda+\alpha,\mu-\alpha \rangle$ should be an integer. That is 
		$$
		\frac{\kappa(\lambda+\alpha,\lambda+\alpha)}{2} - \frac{\kappa(\mu-\alpha,\mu-\alpha)}{2} - \left( \frac{\kappa(\lambda,\lambda)}{2} - \frac{\kappa(\mu,\mu)}{2} \right) = \kappa(\lambda+\mu,\alpha) \in \Z
		$$
		\textit{i.e.},
		$$
		\alpha(\kappa^{-1}(\lambda+\mu)) \in \Z.
		$$
		As this holds for an arbitrary $\alpha \in X^{*}(T)$, we get the result.
	\end{proof}
	\begin{note}
		\begin{proof}[Proof as in Lepowsky-Li]
			Let $|\lambda,\mu\rangle$ be a highest weight vector  that generates the given Fock space in $M$. Let $\alpha \in X^{*}(T)$, because $\cD_{T}^{\kappa}$ is simple and $|\lambda,\mu\rangle \neq 0$ by assumption, $e^{\alpha}(z)|\lambda,\mu\rangle \neq 0$ and it implies that $A(\alpha)(z) |\lambda,\mu\rangle \neq 0$. Moreover since $A(\alpha)(z) |\lambda,\mu\rangle \in M((z))$ there exists an integer $k \in \Z$ maximal such that $A(\alpha)_{k}|\lambda,\mu\rangle \neq 0$. Now applying \eqref{equation equa diff a alpha} to $n=k$ we obtain :
			$$
			(-k-1)A(\alpha)_{k}|\lambda,\mu\rangle = A(\alpha)_{k}(\pi_{L}(\alpha)_{(0)}+ \pi_{R}(\alpha)_{(0)})|\lambda,\mu\rangle = (\lambda + \mu)(\alpha)A(\alpha)_{k}|\lambda,\mu\rangle.
			$$
			Since $A(\alpha)_{k}|\lambda,\mu\rangle \neq 0$ that implies $\alpha(\kappa^{-1}(\lambda+\mu)) \in \Z$ for all $\alpha \in X^{*}(T)$, that is $\kappa^{-1}(\lambda + \mu) \in X_{*}(T)$.
		\end{proof}
	\end{note}
	\begin{proposition}
		\label{proposition there are weight vectors in any module for the torus cdo}
		Let $M$ be an object of $\cD_{T}^{\kappa}\mathrm{-Mod}^{1 \times \jetinf T}$. If $M \neq 0$ then 
		$$M^{t\Lh[[t]] \oplus t\Lh[[t]]} \neq 0.
		$$
		Moreover $M^{t\Lh[[t]]\oplus t\Lh[[t]]}$ contains a weight vector for $\Lh \oplus \Lh$ of weight $(\lambda + \kappa(\gamma,\cdot),-\lambda)$, where $\lambda \in X^{*}(T)$ and $\gamma \in X_{*}(T)$. In particular $M$ contains the Fock space $\cF^{\kappa}_{\lambda + \kappa(\gamma,\cdot)}\otimes \cF^{\kappa^{*}}_{-\lambda}$ as a $V^{\kappa}(\Lh)\otimes V^{\kappa^{*}}(\Lh)$-submodule.
	\end{proposition}
	\begin{proof}
		Because $M$ is nonzero and belongs to the Kazhdan--Lusztig category with respect to the right action, it follows from Theorem \ref{theorem Kazhdan--Lusztig categories for irrational levels} that $M^{0 \oplus t\Lh[[t]]} \neq 0$. For all $m \in M$, since $M$ is a smooth $\Lhkappakappastar$-module and $\Lh$ is finite dimensional, it is clear that the vector space, say $S(m)$, spanned by $\{(\pi_{L}(x)_{(n)}+ \pi_{R}(x)_{(n)}) m \hspace{0.5mm}|\hspace{0.5mm} x\in \Lh, n > 0\}$ is finite dimensional. We denote by $d(m)$ its dimension. Let $m \in M^{0 \oplus t\Lh[[t]]}$ be a nonzero vector such that $d(m)$ is minimal among nonzero elements of $M^{0 \oplus t\Lh[[t]]}$. If $d(m)$ is zero we are done. Otherwise there exists $k \geqslant1$ such that there exists $\alpha \in X^{*}(T)$ satisfying $(\pi_{L}(\alpha)_{(k)} + \pi_{R}(\alpha)_{(k)}) m \neq 0$ and $(\pi_{L}(h')_{(n)} + \pi_{R}(h')_{(n)})m = 0$ for all $h' \in \Lh$ and $n >k$ (recall that $X^{*}(T)$ spans $\Lh^{*}$ as a $\C$-vector space).
		
		Now because $\cD_{T}^{\kappa}$ is a simple vertex algebra, it follows from Proposition \ref{proposition annihilator is an ideal} that~$e^{\alpha}(z)m \neq 0$, and hence 
		$$
		A(\alpha)(z)m = E^{-}((-\alpha,-\alpha),z)e^{\alpha}(z)m \neq 0.
		$$
		But since $e^{\alpha}(z)m \in M((z))$ and $E^{-}((-\alpha,-\alpha),z) \in \End(M)[[z]]$ we see that~$A(\alpha)(z)m \in M((z))$. So there exists an integer $k' \in \Z$ such that $A(\alpha)_{k'}m \neq 0$ and for all $n>k'$, $A(\alpha)_{n}m = 0$. Let us now apply equation \eqref{equation equa diff a alpha} with $n = k + k'$, we have, using equations \eqref{equation commutation action a gauche et a alpha}-\eqref{equation commutation action a droite et a alpha modes strictement positifs} the following equalities
		\begin{align*}
			0 &= (-k-k'-1)A(\alpha)_{k+k'}m \\
			&= \sum_{i=0}^{k} A(\alpha)_{k+k'-i}(\pi_{L}(\alpha)_{(i)} + \pi_{R}(\alpha)_{(i)}) m\\
			&= \sum_{i=0}^{k} (\pi_{L}(\alpha)_{(i)} + \pi_{R}(\alpha)_{(i)})A(\alpha)_{k+k'-i} m\\
			&= (\pi_{L}(\alpha)_{(k)} + \pi_{R}(\alpha)_{(k)})A(\alpha)_{k'}m\\
			&= A(\alpha)_{k'}(\pi_{L}(\alpha)_{(k)}+\pi_{R}(\alpha)_{(k)})m.
		\end{align*}
		Set $m' = A(\alpha)_{k'}m$, it is nonzero by the choice of $k'$. Clearly it is again in $M^{0 \oplus t\Lh[[t]]}$ but since 
		$
		A(\alpha)_{k'} : S(m) \ra S(m')
		$ 
		is a surjection with non-trivial kernel by the previous computation, we see that $d(m')<d(m)$. By contradiction, this proves that there exists a nonzero $m \in M^{0 \oplus t\Lh[[t]]}$ such that $d(m) =0$ \textit{i.e.}, such that for all $x\in \Lh$ and $n>0$
		$$
		(\pi_{L}(x)_{(n)} + \pi_{R}(x)_{(n)})m = 0,
		$$
		but by assumption $\pi_{R}(x)_{(n)}m = 0$ so that implies $\pi_{L}(x)_{(n)}m=0$. This proves the claim. 
		
		Because $M^{t\Lh[[t]] \oplus t\Lh[[t]]}$ is a $\Lh \oplus \Lh$-submodule and $M$ belongs to the Kazhdan--Lusztig~category with respect to the right action, we can choose a nonzero vector $m\in M^{t\Lh[[t]] \oplus t\Lh[[t]]}$ of weight $\lambda \in X^{*}(T)$ with respect to the right action. Let $\alpha \in X^{*}(T)$, as before we can find~$k \in \Z$ such that $A(\alpha)_{k}(m) \neq 0$, now applying equation $\eqref{equation equa diff a alpha}$ with $n=k$ we find that 
		\begin{align*}
			(-k-1)A(\alpha)_{k}m &= A(\alpha)_{k}(\pi_{L}(\alpha)_{(0)}+ \pi_{R}(\alpha)_{(0)})m \\
			&= A(\alpha)_{k} \pi_{L}(\alpha)_{(0)}m + A(\alpha)_{k}\pi_{R}(\alpha)_{(0)}m \\
			&= A(\alpha)_{k}\pi_{L}(\alpha)_{(0)}m + \lambda(\alpha)A(\alpha)_{k}m.
		\end{align*}
		Recall from equation \eqref{equation commutation action a gauche et a alpha} that $[\pi_{L}(\alpha)_{(0)},A(\alpha)_{k}] = \kappa(\alpha,\alpha)A(\alpha)_{k}$ so that the previous equality becomes 
		$$
		\pi_{L}(\alpha)_{(0)}A(\alpha)_{k}m = (-k-1-\lambda(\alpha)+\kappa(\alpha,\alpha))A(\alpha)_{k}m.
		$$
		Notice that $A(\alpha)_{k}m$ is still a nonzero element of $M^{t\Lh[[t]]\oplus t\Lh[[t]]}$ (recall equations~\eqref{equation commutation action a gauche et a alpha modes strictement positifs} and~\eqref{equation commutation action a droite et a alpha modes strictement positifs}) but is now an eigenvector of $\pi_{L}(\alpha)_{(0)}$. It is also clear from equation \eqref{equation commutation action a droite et a alpha} that~$A(\alpha)_{k}m$ is still a weight vector for the right action and that if $m$ is an eigenvector for a certain operator $x_{(0)}$ where $x\in \Lh$, so is $A(\alpha)_{k}m$ by equation \eqref{equation commutation action a gauche et a alpha}. Pick a $\Z$-basis~$\alpha_{1},...,\alpha_{r}$ of $X^{*}(T)$. It is a basis of $\Lh^{*}$ as $\C$-vector space. Iterating the previous construction starting from $m$ and having $\alpha$ range through $\{\alpha_{1},\dots,\alpha_{r}\}$ yields a common eigenvector for the action of $\Lh \oplus \Lh$ \textit{i.e.}, the desired weight vector. The fact that its weight is of the form $(\lambda + \kappa(\gamma,\cdot),-\lambda)$ where $\lambda \in X^{*}(T)$ and $\gamma \in X_{*}(T)$ follows from Lemma~\ref{lemma weights that can appear in fock spaces in cdo for the torus}.
	\end{proof}
	
	\begin{corollaire}
		\label{corollary simple objects of the cdo category for the Torus}
		The only simple objects of the category $\cD_{T}^{\kappa}\mathrm{-Mod}^{1\times \jetinf T}$, up to isomorphism, are the $\gamma \cdot \cD_{T}^{\kappa}$ where $\gamma$ ranges through $X_{*}(T) \subset \cSF(\cD_{T}^{\kappa})$.
	\end{corollaire}
	\begin{proof}
		Let $M$ be a simple object of $\cD_{T}^{\kappa}\mathrm{-Mod}^{1 \times \jetinf T}$. Thanks to Proposition~\ref{proposition there are weight vectors in any module for the torus cdo} we know that there exists $\lambda \in X^{*}(T)$ and $\gamma \in X_{*}(T)$ such that the Fock space~$\cF^{\kappa}_{\lambda + \kappa(\gamma,\cdot)}\otimes \cF^{\kappa^{*}}_{-\lambda}$ is contained in $M$. It follows from Proposition \ref{proposition spectral flow of Fock spaces} that the module $\gamma\cdot M$ contains the Fock space~$\cF^{\kappa}_{\lambda} \otimes \cF^{\kappa^{*}}_{-\lambda}$. Lemma \ref{lemma weights that can appear in fock spaces in cdo for the torus} shows that $\gamma \cdot M$ contains the vacuum Fock space $\cF^{\kappa}_{0}\otimes \cF^{\kappa^{*}}_{0}$. Since $\gamma \cdot M$ is again a simple $\cD_{T}^{\kappa}$-module, Lemma~\ref{lemma criterion for beinn vacuum-like in the case of the cdo} applies to show that $\gamma \cdot M$ is isomorphic to $\cD_{T}^{\kappa}$ as $\cD_{T}^{\kappa}$-modules so that $M$ is isomorphic to $-\gamma \cdot \cD_{T}^{\kappa}$.
	\end{proof}
	
	\begin{theoreme}
		\label{theorem FLE pour le tore}
		The category $\cD_{T}^{\kappa}\mathrm{-Mod}^{1 \times \jetinf T}$ is semisimple with simple objects given by the $\gamma \cdot \cD_{T}^{\kappa}$ where $\gamma \in X_{*}(T) \subset \cSF(\cD_{T}^{\kappa})$.
	\end{theoreme}
	\begin{proof}
		The statement concerning the simple objects is Corollary \ref{corollary simple objects of the cdo category for the Torus}. Let $M$ be an object of~$\cD_{T}^{\kappa}\mathrm{-Mod}^{1 \times \jetinf T}$ and denote by $M'$ the sum of all simple modules contained in~$M$. By definition we have an exact sequence of $\cD_{T}^{\kappa}$-modules so, in particular, of $\Lhkappakappastar$-modules
		$$
		0 \lra M' \lra M \lra M/M' \overset{\pi}{\lra} 0.
		$$
		We want to show that $M/M'$ is zero. Let us assume for a contradiction it is not. It follows from Corollary \ref{corollary simple objects of the cdo category for the Torus} that $M'$ is a direct sum of Fock spaces, so using for example (the proof of) Proposition 3.6 of \cite{LepowskyWilson1982} we have an exact sequence of $\Lh \oplus \Lh$-modules 
		$$
		0 \lra (M')^{t\Lh[[t]] \oplus t\Lh[[t]]} \lra M^{t\Lh[[t]] \oplus t\Lh[[t]]} \lra (M/M')^{t\Lh[[t]] \oplus t\Lh[[t]]} \lra 0.
		$$ 
		Let $x \in (M/M')^{t\Lh[[t]]\oplus t\Lh[[t]]}$ be a weight vector with respect to $\Lh \oplus \Lh$ (it exists by Proposition \ref{proposition there are weight vectors in any module for the torus cdo}). By assumption, the right factor $0 \oplus \Lh \subset \Lh \oplus \Lh$ acts semisimply on every object so we can find a lift $\widetilde{x} \in M^{t\Lh[[t]]\oplus t\Lh[[t]]}$ which is of weight $\lambda \in \Lh^{*}$ with respect to the right factor $0 \oplus \Lh \subset \Lh \oplus \Lh$. 
		
		Let $\alpha \in X^{*}(T)$, as in the proof of Proposition \ref{proposition there are weight vectors in any module for the torus cdo}. There exists  an integer $k$ such that~$A(\alpha)_{k}x \neq 0$. Applying equation \eqref{equation equa diff a alpha} with $n = k$ we have (recall that~$\widetilde{x} \in M^{t\Lh[[t]] \oplus t\Lh[[t]]})$
		\begin{align*}
			(-k-1)A(\alpha)_{k}\widetilde{x} = A(\alpha)_{k}(\pi_{L}(\alpha)_{(0)}+ \pi_{R}(\alpha)_{(0)})\widetilde{x}  
			= A(\alpha)_{k}\pi_{L}(\alpha)_{(0)}\widetilde{x} + \lambda(\alpha)A(\alpha)_{k}\widetilde{x}.
		\end{align*}
		So by equation \eqref{equation commutation action a gauche et a alpha} we have 
		$$
		\pi_{L}(\alpha)_{(0)}A(\alpha)_{k}\widetilde{x} = (-k-1-\lambda(\alpha)+\kappa(\alpha,\alpha))A(\alpha)_{k}\widetilde{x}.
		$$
		Now, $A(\alpha)_{k}\widetilde{x}$ does not belong to $M'$ because $\pi(A(\alpha)_{k}\widetilde{x}) = A(\alpha)_{k}x \neq 0$ by our choice of $k$. It is still in $M^{t\Lh[[t]] \oplus t\Lh[[t]]}$ by equations \eqref{equation commutation action a gauche et a alpha modes strictement positifs} and \eqref{equation commutation action a droite et a alpha modes strictement positifs}. Thanks to equations \eqref{equation commutation action a gauche et a alpha} and \eqref{equation commutation action a droite et a alpha}, $A(\alpha)_{k}\widetilde{x}$ is again a common eigenvector for the right action of $\Lh$ and of the left action of $\alpha \in \Lh$. Iterating the previous construction starting from $\widetilde{x}$ as $\alpha$ runs through a $\Z$-basis of $X^{*}(T)$ yields a vector $m \in M^{t\Lh[[t]] \oplus t\Lh[[t]]}$ such that $m \notin M'$ and that is a weight vector for the action of $\Lh \oplus \Lh$. As such the $\Lhkappakappastar$-module generated by $m$ in $M$ is a Fock space of the form $\cF^{\kappa}_{\lambda + \kappa(\gamma,\cdot)}\otimes \cF^{\kappa^{*}}_{-\lambda}$ where $\lambda \in X^{*}(T)$ and $\gamma \in X_{*}(T)$ by Lemma~\ref{lemma weights that can appear in fock spaces in cdo for the torus}. As in the proof of Corollary \ref{corollary simple objects of the cdo category for the Torus}, we see that there exists a morphism of $\cD_{T}^{\kappa}$-module from $-\gamma \cdot \cD_{T}^{\kappa}$ to $M$ such that $m \neq 0$ belongs to the image. But because $-\gamma \cdot \cD_{T}^{\kappa}$ is simple, it means in particular that the $\cD_{T}^{\kappa}$-submodule generated by $m$ inside of $M$ is simple and is not contained in $M'$, because $m$ is not. This contradicts the definition of $M'$ and finishes the proof.
	\end{proof}
	
	\begin{remarque}
		\label{remarque right hypothesis implies left decomposition tore}
		A striking feature of this result is that our hypothesis enforces that the module we start with is a direct sum of Fock spaces with respect to the right action but requires nothing with respect to the left action. But \textit{a posteriori} the module is indeed a direct sum of Fock spaces with respect to the left action. This is even more remarkable in the case of a simple Lie algebra (see Theorem~\ref{theorem FLE for adjoint type ADE}).
	\end{remarque}
	
	Recall that, inside $\Lh$, lies the lattice $X_{*}(T)$ of cocharacters, and inside $\Lh^{*}$, lies the lattice $X^{*}(T)$ of characters. Let us identify $\Lh$ with $\Lh^{*}$ using $\kappa$. Denote by $\kappa : \Lh \lra \Lh^{*}$ the isomorphism and by $\kappa^{-1}$ its inverse. It is clear that $\kappa^{-1}(X^{*}(T)) \subset \Lh$ is a full rank lattice of $\Lh$. We define 
	$$
	Y = X_{*}(T)\cap \kappa^{-1}(X^{*}(T)).
	$$
	It is a free abelian group of rank possibly lower than $r$. 
	\begin{corollaire}
		\label{corollaire of theorem fle tore}
		The simple objects of the category $\cD_{T}^{\kappa}\mathrm{-Mod}^{\jetinf T \times \jetinf T}$ are in bijection with the subgroup $Y$ of $X_{*}(T)$ where for each $\gamma \in Y \subset X_{*}(T) \subset \cSF(\cD_{T}^{\kappa})$, the corresponding simple object is given by $\gamma \cdot \cD_{T}^{\kappa}$. Moreover the category $\cD_{T}^{\kappa}\mathrm{-Mod}^{\jetinf T \times \jetinf T}$ is semisimple, that is, every object of $\cD_{T}^{\kappa}\mathrm{-Mod}^{\jetinf T \times \jetinf T}$ is a direct sum of simple objects. 
	\end{corollaire}
	\begin{proof}	
		This is a direct consequence of Theorem \ref{theorem FLE pour le tore} and Corollary \ref{corollary sufficient condition for integrability}.
	\end{proof}
	
	\begin{remarque}
		\label{remark degeneration geometric Satake torus}
		For almost every nondegenerate $\kappa$, we have $Y=0$. In this case, the category $\cD_{T}^{\kappa}\mathrm{-Mod}^{\jetinf T \times \jetinf T}$ is equivalent to the category of $\C$-vector spaces (compare with Theorem \ref{theorem degeneration of the geometric Satake correspondence}). For general values of $\kappa$, we believe this result should be compared with the twisted geometric Satake equivalence of \cite{FinkelbergLysenko2007} and with the metaplectic geometric Langlands theory (\cite{Reich_Paper},\cite{GaitsgoryLysenkoMetaplectic}).

	\end{remarque}

	\section{The case of a simple adjoint group of classical simply laced type}
	\label{section the case of a simple simply laced group of adjoint type}	
	We now assume that $G$ is a simple adjoint group of type A or D. That is, $G$ equals $\mathrm{PSL}_{n}$ or the quotient of $\mathrm{SO}_{2n}$ by its center. In particular, we have the equalities of abelian groups $X^{*}(T) = Q$ and $X_{*}(T)=\check{P}$.

	\subsection{Shifted chiral differential operators on a simple group}
	\label{section shifted chiral differential operators on a simple group}
	The chiral Peter--Weyl decomposition is at the technical heart of our arguments so far. We can even go further, in a sense, all that mattered was the existence of a simple vertex algebra structure on $$
	\bigoplus_{\lambda \in X^{*}(T)_{+}} V_{\lambda}^{k} \otimes V_{-\omega_{0}\lambda}^{k^{*}}
	$$ 
	which is a conformal extension of $V^{k}(\Lg) \otimes V^{k^{*}}(\Lg)$. Even the relation between $k$ and $k^{*}$ is irrelevant to some extent. This motivates the following question, given $k,l \in \C\setminus \Q$, does there exist a simple vertex algebra structure on 
	$$
	\bigoplus_{\lambda \in X^{*}(T)_{+}} V_{\lambda}^{k} \otimes V_{-\omega_{0}\lambda}^{l}
	$$
	which is a conformal extension of $V^{k}(\Lg) \otimes V^{l}(\Lg)$? 
	
	It is easily seen that $k$ and $l$ may not be arbitrary if we require our vertex algebras to be $\Z$-graded. We can rewrite the relation between $k$ and $k^{*}$ as 
	\begin{equation}
		\label{equation k and kstar}
		\frac{1}{k+\check{h}} + \frac{1}{k^{*}+\check{h}} = 0.
	\end{equation}
	Given $n \in \Z$ and $k \in \C\setminus \Q$ it is natural to allow $l$ to be defined by the relation 
	\begin{equation}
		\label{equation k and l}
		\frac{1}{k+\check{h}} + \frac{1}{l+\check{h}} = n.
	\end{equation}
	We call $l$ the $n$-shifted level of $k$ and denote it by $k[n]$. It is easy to see that if $k$ is irrational then so is $k[n]$.

	The following was conjectured by Creutzig and Gaiotto (see Conjecture 1.1 of~\cite{CreutzigGaiotto2020}): 	
	\begin{theoreme}[Main Theorem B of \cite{Moriwaki2022}, see also Theorem 2.2 of \cite{CreutzigLinshawNakatsukaSato2025}]
		\label{conjecture existence and unicity of shifted cdo's}
		Let $k \in \C\setminus\Q$ and $n\in \Z$ then there exists a simple vertex algebra structure on 
		\begin{equation}
			\label{equation shifted peter weyl}
			\bigoplus_{\lambda \in X^{*}(T)_{+}} V_{\lambda}^{k} \otimes V_{-\omega_{0}\lambda}^{k[n]}
		\end{equation}
		that is a conformal extension of $V^{k}(\Lg) \otimes V^{k[n]}(\Lg)$. 
	\end{theoreme}
	\begin{remarque}
		If the previous theorem can be generalized to the exceptional type E, our techniques apply as well to prove Conjecture \ref{conjecture FLE} in this case. 
	\end{remarque}
	We will prove there is a unique such vertex algebra structure (see Theorem~\ref{proposition unicity of shifted cdos}). So it makes sense to denote this vertex algebra $\cD_{G}^{k}[n]$ and call it the $n$-shifted vertex algebra of chiral differential operators on the group $G$ at level~$k$. 
	
	We have morphisms of vertex algebras 
	$$
	\xymatrix{
		& V^{k}(\Lg) \otimes V^{k[n]}(\Lg) \ar[d] &\\
		V^{k}(\Lg) \ar[r] \ar[ru]& \cD_{G}^{k}[n] & V^{k[n]}(\Lg) \ar[l] \ar[lu]}
	$$
	so we can perform the Hamiltonian reduction with respect to $V^{k}(\Lg)$. We obtain a vertex algebra
	$$
	\mathcal{W}_{G}^{k}[n] = H^{*}_{DS}(\cD_{G}^{k}[n]),
	$$
	that we call the $n$-shifted equivariant affine $\mathcal{W}$-algebra on the group $G$ at level $k$. 
	It comes equipped with morphisms of vertex algebras 
	
	$$
	\xymatrix{
		& W^{k}(\Lg) \otimes V^{k[n]}(\Lg) \ar[d] &\\
		W^{k}(\Lg) \ar[r] \ar[ru]& \mathcal{W}_{G}^{k}[n] & V^{k[n]}(\Lg).\ar[l]\ar[lu]}
	$$
	
	In particular any $\mathcal{W}_{G}^{k}[n]$-module is naturally a $W^{k}(\Lg) \otimes V^{k[n]}(\Lg)$-module by restriction \textit{i.e.}, it carries two commuting actions of $W^{k}(\Lg)$ and $V^{k[n]}(\Lg)$. We define, as in the non-shifted case (see Definition \ref{definition module categories for the cdo}), the categories 
	$$
	\cD_{G}^{k}[n]\mathrm{-Mod}^{\jetinf G \times \jetinf G},\hspace{0.5mm} \cD_{G}^{k}[n]\mathrm{-Mod}^{1 \times \jetinf G}, \hspace{0.5mm} \mathcal{W}_{G}^{k}[n]\mathrm{-Mod}^{\jetinf G}.
	$$
	\begin{remarque}
		\label{remark open problem deformable family}
		It is an open problem whether the vertex algebras $\cD_{G}^{k}[n]$ and $\mathcal{W}_{G}^{k}[n]$ come from a deformable family of vertex algebras where the structure constants are continuous functions of $k$. This question is settled in the case of $\mathrm{PSL}_{2}$ for $n=1$ in~\cite{CreutzigGaiotto2020} and for $n=2$ in \cite{creutziggaiottolinshaw}.
	\end{remarque}

	\subsection{Convolution operation}
	\label{section convolution operation}		
	Let $V$ and $W$ be two vertex algebras such that there exist two morphisms of vertex algebras 
	$$
	V \longleftarrow V^{\kappa^{*}}(\Lg),\hspace{0.5mm}V^{\kappa}(\Lg) \lra W.
	$$
	By definition, we have the equality 
	$$
	\kappa + \kappa^{*} = -\kappa_{\Lg}
	$$ 
	so that, with respect to the diagonal action,
	$
	V \otimes W 
	$
	is a vertex algebra that comes equipped with a morphism from $V^{-\kappa_{\Lg}}(\Lg)$. This is exactly the level for which the relative semi-infinite cohomology 
	$$
	V \circ W = H^{\frac{\infty}{2}+\bullet}(\Lg^{-\kappa_{\Lg}},\Lg,V \otimes W)
	$$
	is defined, and by construction, $V \circ W$ is a vertex algebra. If we are given $M \in V\mathrm{-Mod}$ and $N \in W\mathrm{-Mod}$ then by viewing $M \otimes N$ as a $V^{-\kappa_{\Lg}}(\Lg)$-module, we can consider the relative semi-infinite cohomology 
	$$
	M \circ N = H^{\frac{\infty}{2}+\bullet}(\Lg^{-\kappa_{\Lg}},\Lg,M \otimes N),
	$$
	and it is naturally a $V \circ W$-module. This convolution operation, known as the quantized Moore--Tachikawa operation, is especially well-behaved on the Kazhdan--Lusztig category.
	
	\begin{theoreme}[Theorem 3.3 of \cite{CreutzigLinshawNakatsukaSato2025} and  \cite{FrenkelGarlandZuckerman1986}]
		Assume that $V$ (resp. $W$) belongs to $V^{\kappa^{*}}(\Lg)\mathrm{-Mod}^{\jetinf G}$ (resp. $V^{\kappa}(\Lg)\mathrm{-Mod}^{\jetinf G}$) then 
		\begin{equation}
			\label{equation vanishing of convolution on category KL}
			H^{\frac{\infty}{2}+k}(\Lg^{-\kappa_{\Lg}},\Lg,V \otimes W) = 0 \textrm{ if $k \neq 0$.}
		\end{equation}
		More precisely, we have the equality of vertex algebras 
		$$
		V^{\kappa^{*}}(\Lg) \circ V^{\kappa}(\Lg) = \C 
		$$
		and if $\lambda,\mu \in P_{+}$, then we have the equality as $\C$-modules
		\begin{equation}
			\label{equation convolution of weyl modules}
			V^{\kappa^{*}}_{\lambda} \circ V^{\kappa}_{\mu} = \delta_{\lambda, -\omega_{0}\mu} \C.
		\end{equation} 
	\end{theoreme}
	
	It makes sense to consider the functors 
	\begin{align*}
		&\cdo \circ - : V^{\kappa}(\Lg)\mathrm{-Mod} \lra (\cdo \circ V^{\kappa}(\Lg))\mathrm{-Mod},\\
		&- \circ \cdo : V^{\kappa^{*}}(\Lg)\mathrm{-Mod} \lra (V^{\kappa^{*}}(\Lg) \circ \cdo)\mathrm{-Mod}.
	\end{align*}
	The following property is fundamental and expresses that $\cdo$ is the unit for the convolution operation:
	\begin{proposition}[Theorem 6.2 of \cite{arkhipov_gaitsgory_differential_operators} and discussion following Proposition~5.6 of~\cite{Arakawa2018}]
		\label{proposition cdo is the identity for convolution}
		We have the equalities of vertex algebras 
		\begin{equation}
			\cdo \circ V^{\kappa}(\Lg) = V^{\kappa}(\Lg),\quad 
			V^{\kappa^{*}}(\Lg) \circ \cdo = V^{\kappa^{*}}(\Lg).
		\end{equation}
		More generally, if $V$ (resp. $W$) is a vertex algebra that belongs to $V^{\kappa^{*}}(\Lg)\mathrm{-Mod}^{\jetinf G}$ (resp. $V^{\kappa}(\Lg)\mathrm{-Mod}^{\jetinf G}$) then, as vertex algebras,
		\begin{equation}
			V \circ \cdo  = V,\quad \cdo \circ W = W. 
		\end{equation}
		Moreover, if $M$ (resp. $N$) is a $V$-module (resp. $W$-module) in $V^{\kappa^{*}}(\Lg)\mathrm{-Mod}^{\jetinf G}$ (resp.~$V^{\kappa}(\Lg)\mathrm{-Mod}^{\jetinf G}$) then, as a $V$-module (resp. $W$-module),
		\begin{equation}
			\label{equation cdo is the identity of convolutions on modules}
			M \circ \cdo = M,\quad \cdo \circ N =  N.
		\end{equation}
	\end{proposition}
	\begin{remarque}
		\label{remark results on convolution extend to reductive groups}
		The definitions and results given so far in this paragraph extend to the case of a reductive algebraic group for generic levels (this is the setting of \cite{CreutzigLinshawNakatsukaSato2025}). In particular when $G$ is an algebraic torus, equation \eqref{equation convolution of weyl modules} holds for Fock spaces (see~§6.1 of \cite{Correspondencesofcategoriesforsubregular}). We also note that the semisimplicity of the Kazhdan--Lusztig category $V^{\kappa}(\Lg)\otimes V^{\kappa^{*}}(\Lg)\mathrm{-Mod}^{\jetinf G \times \jetinf G}$ combined with equalities \eqref{equation convolution of weyl modules} and \eqref{equation cdo is the identity of convolutions on modules} gives another proof of the chiral Peter--Weyl Theorem \ref{theorem chiral Peter-Weyl}.
	\end{remarque}
	
	The vertex algebra $\ewalg$ comes equipped with a morphism
	$$
	\ewalg \longleftarrow V^{\kappa^{*}}(\Lg).
	$$
	As such, it defines a functor 
	$$
	\ewalg \circ - : V^{\kappa}(\Lg)\mathrm{-Mod} \lra (\ewalg \circ V^{\kappa}(\Lg) )\mathrm{-Mod}
	$$
	The following result expresses the fact that taking the convolution with $\ewalg$ represents the quantum Hamiltonian reduction functor on the Kazhdan--Lusztig category.
	\begin{proposition}[Theorem 6.7 of \cite{Arakawa2018}]
		\label{proposition equivariant W algebra represents quantum Hamiltonian reduction functor}
		We have the equality of vertex algebras
		\begin{equation}
			\ewalg \circ V^{\kappa}(\Lg)   = W^{\kappa}(\Lg).
		\end{equation}
		More generally, if $V$ is a vertex algebra that belongs to $V^{\kappa}(\Lg)\mathrm{-Mod}^{\jetinf G}$ then we have the equality of vertex algebras 
		\begin{equation}
			\ewalg \circ	V   = H^{0}_{DS}(V).
		\end{equation}
		Moreover, if $M$ is a $V$-module in $V^{\kappa}(\Lg)\mathrm{-Mod}^{\jetinf G}$ then we have the equality as $H^{*}_{DS}(V)$-modules
		\begin{equation}
			\ewalg \circ M   = H^{0}_{DS}(M).
		\end{equation} 
	\end{proposition}

	\subsection{The strategy}
	\label{section the strategy}
	Conjecture \ref{conjecture FLE} gives a description of the category 
	$$
	\mathcal{W}^{k}_{G}\mathrm{-Mod}^{\jetinf G}
	$$
	in terms of the Langlands dual group of $G$.
	The study of this category feels a bit beyond the scope of usual vertex algebraic techniques. For instance $\mathcal{W}^{k}_{G}$ is a $\Z$-graded vertex algebra that is not lower bounded. So it is unclear at face value how the Zhu algebra may be useful. In the case of a torus (cf §\ref{section the case of an algebraic torus}), some explicit computations in Fock spaces solve the problem. Here, the lack of knowledge of the relations between strong generators of the principal $\mathcal{W}$-algebra makes this approach seemingly impossible: one needs a new idea to proceed. 
	
	Let us forget momentarily about Hamiltonian reduction and consider $\cD^{k}_{G}$ again. This is a somewhat complicated vertex algebra, for instance its $0$-degree space is $\cO_{G}$ which is infinite dimensional. In view of Theorem \ref{conjecture existence and unicity of shifted cdo's} we see that $\cD^{k}_{G} = \cD^{k}_{G}[0]$ is a member of a $\Z$-family of simple vertex algebras. Members of this family have very different behaviours as vertex algebras, for instance $\mathcal{D}^{k}_{G}[1]$ is positively graded and $(\mathcal{D}^{k}_{G}[1])_{0} = \C$ while $\mathcal{D}^{k}_{G}[-1]$ is a $\Z$-graded, not lower bounded vertex algebra with every conformal degree of infinite dimension. Yet, as we will see shortly, the categories $\mathcal{D}^{k}_{G}[0]\mathrm{-Mod}^{\jetinf G}$ and $\mathcal{D}^{k}_{G}[1]\mathrm{-Mod}^{\jetinf G}$ are naturally equivalent under the convolution operation. The equivalence is well behaved enough to induce an equivalence between $\mathcal{W}^{k}_{G}[0]\mathrm{-Mod}^{\jetinf G}$ and $\mathcal{W}^{k}_{G}[1]\mathrm{-Mod}^{\jetinf G}$. The hope is that the latter category is easier to relate to representations of the Langlands dual group.
	
	Let us be more precise, let $m,n \in \Z$ and $k,l \in \C\setminus \Q$, the vertex algebras $\cD_{G}^{k}[m]$ and~$\cD_{G}^{l}[n]$ each come equipped with two morphisms
	$$
	\xymatrix{
		& \cD_{G}^{k}[m] &  &   & \cD_{G}^{l}[n] &   \\
		V^{k}(\Lg) \ar[ur] &   & V^{k[m]}(\Lg) \ar[ul]& V^{l}(\Lg) \ar[ur] &   & V^{l[n]}(\Lg) \ar[ul]
	}.
	$$
	
	In particular the convolution $\cD_{G}^{k}[m] \circ \cD_{G}^{l}[n]$ makes sense if and only if $l = k[m][0]$. An easy computation using relation \eqref{equation k and l} shows that 
	$$
	k[m][0][n] = k[m+n].
	$$
	Hence we have two morphisms of vertex algebras 
	\begin{equation*}
		\xymatrix{
			&  & \cD_{G}^{k}[m] \circ  \cD_{G}^{k[m][0]}[n] & &  \\
			V^{k}(\Lg) \ar[urr] & &  & & V^{k[m+n]}(\Lg). \ar[llu] 
		}
	\end{equation*}
	More precisely, in view of equations \eqref{equation shifted peter weyl} and \eqref{equation convolution of weyl modules}, as a $V^{k}(\Lg) \otimes   V^{k[m+n]}(\Lg)$-module, we have the decomposition
	\begin{equation}
		\label{equation convolution of two shifted cdo has peter weyl decomposition}
		\cD^{k}_{G}[m] \circ \cD_{G}^{k[m][0]}[n] = \bigoplus_{\lambda\in X^{*}(T)_{+}} V_{\lambda}^{k} \otimes V_{-\omega_{0}\lambda}^{k[m+n]}.
	\end{equation}
	This strongly suggests the following:
	\begin{proposition}
		\label{proposition convolution of shifted cdo}
		We have the following equality of vertex algebras
		$$
		\cD^{k}_{G}[m] \circ \cD_{G}^{k[m][0]}[n] = \cD^{k}_{G}[m+n].
		$$
	\end{proposition}
	We begin by collecting some consequences. The following technical lemma is essential:
	\begin{lemme}[Lemma 4.8 of \cite{BeemNair2022} and Theorem 10.11 of \cite{Arakawa2018}]
		\label{lemma associativity in KL}
		Let $A$,$B$ and $C$ be vertex algebras. Assume that there exist morphisms of vertex algebras
		\begin{align*}
			&A \longleftarrow V^{k^{*}}(\Lg) \\
			V^{k}(\Lg) \lra &B \longleftarrow V^{k^{*}}(\Lg) \\
			V^{k}(\Lg) \lra &C
		\end{align*}
		such that $A$,$B$ and $C$ belong to the Kazhdan--Lusztig category. Then, as vertex algebras, 
		\begin{equation}
			\label{equation associativity in KL}
			(A \circ B) \circ C \simeq A \circ (B \circ C).
		\end{equation} 
		Moreover if we are given, with obvious notations,
		\begin{align*}
			U \in A\mathrm{-Mod}^{\jetinf G},\quad
			V \in B\mathrm{-Mod}^{\jetinf G \times \jetinf G},\quad W \in C\mathrm{-Mod}^{\jetinf G},
		\end{align*}
		then, as modules over $A \circ B \circ C$ (which is unambiguous by \eqref{equation associativity in KL}), we have 
		$$
		(U \circ V) \circ W \simeq U \circ (V \circ W).
		$$
	\end{lemme}
	
	\begin{remarque}
		The proof of the associativity property given in \cite{BeemNair2022} is a spectral sequence argument that relies on the fact that in the Kazhdan--Lusztig category, the cohomology is concentrated in degree $0$. Because the vanishing also holds in the case of an algebraic torus (recall equation \eqref{equation convolution of weyl modules} and Remark \ref{remark results on convolution extend to reductive groups}), their proof also works for an abelian Lie algebra. 
	\end{remarque}

	Let $n \in \Z$ and $V$ be any vertex algebra that comes equipped with a vertex algebra morphism
	$$
	V \longleftarrow V^{k^{*}}(\Lg)
	$$
	such that $V$ is in the Kazhdan--Lusztig category $V^{k^{*}}(\Lg)\mathrm{-Mod}^{\jetinf G}$. Set 
	$$
	V[n] = V \circ \mathcal{D}_{G}^{k}[n].
	$$
	
	\begin{remarque}
		Note that this is indeed compatible with the previously introduced notations as we have for all $n \in \Z$ the equality of vertex algebras
		$$
		\mathcal{W}_{G}^{k}[n] = H^{*}_{DS}(\mathcal{D}_{G}^{k}[n])
		$$
		thanks to Proposition \ref{proposition equivariant W algebra represents quantum Hamiltonian reduction functor}.
	\end{remarque}
	Convolution by $\mathcal{D}_{G}^{k[n][0]}[1]$ defines a functor 
	$$
	-\circ \mathcal{D}_{G}^{k[n][0]}[1] : V[n]\mathrm{-Mod}^{\jetinf G} \lra V[n+1]\mathrm{-Mod}^{\jetinf G}.
	$$
	Convolution by $\mathcal{D}_{G}^{k[n+1][0]}[-1]$ defines a functor in the opposite direction in view of Proposition \ref{proposition convolution of shifted cdo} and Lemma \ref{lemma associativity in KL}. As a direct consequence of Proposition~\ref{proposition cdo is the identity for convolution} together with Lemma ~\ref{lemma associativity in KL} we have:
	\begin{proposition}
		\label{proposition convolution by 1 and -1 are inverses}
		These two functors are quasi-inverse to one another.
	\end{proposition}	
	This readily implies that the following diagrams consist of equivalences of categories 
	\[
	\begin{tikzcd}[column sep=6em, row sep=4em]
		\cD_{G}^{k}[-1]\mathrm{-Mod}^{1 \times \jetinf G}
		\arrow[r, bend left=30, "{-\!\circ \mathcal{D}_{G}^{k[-1][0]}[1]}"{above}]&
		\cD_{G}^{k}\mathrm{-Mod}^{1 \times \jetinf G}
		\arrow[r, bend left=30, "{-\!\circ \mathcal{D}_{G}^{k}[1]}"{above}]
		\arrow[l, bend left=30, "{-\!\circ \mathcal{D}_{G}^{k}[-1]}"{below}]&
		\cD_{G}^{k}[1]\mathrm{-Mod}^{1 \times \jetinf G}
		\arrow[l, bend left=30, "{-\!\circ \mathcal{D}_{G}^{k[1][0]}[-1]}"{below}],
	\end{tikzcd}
	\]
	\[
	\begin{tikzcd}[column sep=6em, row sep=4em]
		\mathcal{W}_{G}^{k}[-1]\mathrm{-Mod}^{\jetinf G}
		\arrow[r, bend left=30, "{-\!\circ \mathcal{D}_{G}^{k[-1][0]}[1]}"{above}]&
		\mathcal{W}_{G}^{k}\mathrm{-Mod}^{\jetinf G}
		\arrow[r, bend left=30, "{-\!\circ \mathcal{D}_{G}^{k}[1]}"{above}]
		\arrow[l, bend left=30, "{-\!\circ \mathcal{D}_{G}^{k}[-1]}"{below}]&
		\mathcal{W}_{G}^{k}[1]\mathrm{-Mod}^{\jetinf G}
		\arrow[l, bend left=30, "{-\!\circ \mathcal{D}_{G}^{k[1][0]}[-1]}"{below}].
	\end{tikzcd}
	\]	
	
	\begin{remarque}
		Clearly, the same technique shows that for all $n\in \mathbb{Z}$ the categories $\cD_{G}^{k}[n]\mathrm{-Mod}^{\jetinf G \times \jetinf G}$  are equivalent. So Theorem \ref{theorem degeneration of the geometric Satake correspondence} describes them all. 
	\end{remarque}
	
	We now turn to the proof of Proposition \ref{proposition convolution of shifted cdo}.
	
	\begin{lemme}[Corollary 3.6 of \cite{CreutzigLinshawNakatsukaSato2025}]
		\label{lemma simplicity of convolution}
		Let 
		$$
		V = \bigoplus_{\lambda\in X^{*}(T)_{+}} V_{\lambda}^{k} \otimes V_{-\omega_{0}\lambda}^{k[m]}
		$$
		and 
		$$
		W = \bigoplus_{\lambda\in X^{*}(T)_{+}} V_{\lambda}^{k[m][0]} \otimes V_{-\omega_{0}\lambda}^{k[m+n]}
		$$
		be simple vertex algebras. Assume that $V$ (resp. $W$) is a conformal extension of $V^{k}(\Lg) \otimes V^{k[m]}(\Lg)$ (resp. $V^{k[m][0]}(\Lg) \otimes V^{k[m+n]}(\Lg)$). Then $V \circ W$ is a simple vertex algebra that is a conformal extension of $V^{k}(\Lg) \otimes V^{k[m+n]}(\Lg)$. 
	\end{lemme} 
	It is clear from this lemma that Proposition \ref{proposition convolution of shifted cdo} follows from equation \eqref{equation convolution of two shifted cdo has peter weyl decomposition} together with the following statement (see §3.1 of \cite{CreutzigGaiotto2020}).
	\begin{theoreme}
		\label{proposition unicity of shifted cdos}
		There exists on the $V^{k}(\Lg)\otimes V^{k[n]}(\Lg)$-module 
		$$
		\bigoplus_{\lambda \in X^{*}(T)_{+}} V_{\lambda}^{k} \otimes V_{-\omega_{0}\lambda}^{k[n]}
		$$
		a unique simple vertex algebra structure that is a conformal extension of $V^{k}(\Lg)\otimes V^{k[n]}(\Lg)$ given by that of $\mathcal{D}_{G}^{k}[n]$.
	\end{theoreme}
	\begin{proof}
		The existence part of the statement is Theorem \ref{conjecture existence and unicity of shifted cdo's}. For the uniqueness, assume we are given two families $\mathcal{A}^{k}[n]$ and $\mathcal{B}^{k}[n]$  depending on $k \in \C\setminus\Q$ and $n \in \mathbb{Z}$ with obvious notations. Let $k \in \C \setminus \Q$ and $n \in \Z$. It follows from Theorem \ref{theorem rigidity of chiral Peter-Weyl}, Lemma~\ref{lemma simplicity of convolution} and equation \eqref{equation convolution of two shifted cdo has peter weyl decomposition} that we have the equalities of vertex algebras:
		\begin{align*}
			\mathcal{D}_{G}^{k[n][0]} &= \mathcal{A}^{k[n][0]}[-n] \circ \mathcal{B}^{k}[n], \\
			\mathcal{D}_{G}^{k} &= \mathcal{A}^{k}[n] \circ \mathcal{A}^{k[n][0]}[-n].
		\end{align*}
		By successive applications of Lemma \ref{lemma associativity in KL}, Proposition \ref{proposition cdo is the identity for convolution} and the previous two equalities we have the equalities of vertex algebras:
		\begin{align*}
			\mathcal{A}^{k}[n] &= \mathcal{A}^{k}[n] \circ \mathcal{D}_{G}^{k[n][0]} \\
			&= \mathcal{A}^{k}[n] \circ (\mathcal{A}^{k[n][0]}[-n] \circ \mathcal{B}^{k}[n])\\
			&= (\mathcal{A}^{k}[n] \circ \mathcal{A}^{k[n][0]}[-n]) \circ \mathcal{B}^{k}[n]\\
			&= \mathcal{D}_{G}^{k} \circ \mathcal{B}^{k}[n]\\
			&= \mathcal{B}^{k}[n].
		\end{align*}
	\end{proof}
	
	\subsection{Proof of the fundamental local equivalence}
	The goal of this section is to prove the following
	\begin{theoreme}
		\label{theorem FLE for adjoint type ADE}
		Conjecture \ref{conjecture FLE} holds for a simple adjoint group of type A or D. That is, for $k \in \C \setminus \Q$ there is an equivalence of abelian categories 
		$$
		\mathcal{W}_{G}^{k}\mathrm{-Mod}^{\jetinf G} \simeq \check{G}\mathrm{-Mod}
		$$
		that sends for all $\gamma \in X_{*}(T)_{+}$ the simple object $H_{DS}^{2\rho(\gamma)}(\gamma \cdot \mathcal{D}_{G}^{k})$ to $V_{\gamma}$.
	\end{theoreme}
	The rest of this section is devoted to the proof of this theorem. By assumption we have
	$
	\Lie(\check{G}) = \Lg
	$ and $\check{G}$ is simply connected so the simple objects of $\check{G}\mathrm{-Mod}$ are in bijection with $P_{+}$.

	We apply the strategy explained in §\ref{section the strategy}. The main ingredient is the following relation between the equivariant affine $\mathcal{W}$-algebra and the Goddard--Kent--Olive coset realization of $\mathcal{W}$-algebra:
	\begin{theoreme}
		\label{theorem computation of shifted cdo's}
		For $k \in \C \setminus \Q$ there is an isomorphism of vertex algebras 
		$$
		\mathcal{W}_{G}^{k}[1] \simeq V^{k[1] - 1}(\Lg) \otimes L_{1}(\Lg)
		$$
		where $L_{1}(\Lg)$ is the simple quotient of $V^{1}(\Lg)$.
	\end{theoreme}
	\begin{proof}
		Because $\mathcal{W}_{G}^{k}$ is simple (Lemma \ref{lemma simplicity of the equivariant walgebra}), it follows from Proposition \ref{proposition convolution by 1 and -1 are inverses} that~$\mathcal{W}_{G}^{k}[1]$ is simple. Now we notice that (the proof of) Theorem 1.1 of \cite{CreutzigNakatsuka2023} (see §4 of \textit{loc.cit.}) applies and shows the desired result. 
	\end{proof}
	
	According to Proposition \ref{proposition convolution by 1 and -1 are inverses} the following functors are equivalences of categories
	\[
	\begin{tikzcd}[column sep=6em, row sep=4em]
		\mathcal{W}_{G}^{k}\mathrm{-Mod}^{\jetinf G}
		\arrow[r, bend left=30, "{-\!\circ \mathcal{D}_{G}^{k}[1]}"{above}]&
		\mathcal{W}_{G}^{k}[1]\mathrm{-Mod}^{\jetinf G}
		\arrow[l, bend left=30, "{-\!\circ \mathcal{D}_{G}^{k[1][0]}[-1]}"{below}].
	\end{tikzcd}
	\]		
	Under the isomorphism of Theorem \ref{theorem computation of shifted cdo's} one checks that 
	$$
	\mathcal{W}^{k}_{G}[1]\mathrm{-Mod}^{\jetinf G} \simeq V^{k[1]-1}(\Lg)\otimes L_{1}(\Lg)\mathrm{-Mod}^{\jetinf G}
	$$
	where the right-hand side consists of the $V^{k[1]-1}(\Lg) \otimes L_{1}(\Lg)$-modules that when seen as~$V^{k[1]}(\Lg)$-modules via the diagonal embedding 
	$$
	V^{k[1]}(\Lg) \lra V^{k[1]-1}(\Lg) \otimes L_{1}(\Lg)
	$$
	belong to $V^{k[1]}(\Lg)\mathrm{-Mod}^{\jetinf G}$. The representation theory of $L_{1}(\Lg)$ is well-understood (see \textit{e.g.}, §6.6 of \cite{LepowskyLi1994}), its simple modules are parametrized by the set of 
	$$
	P^{1}_{+} =\{ \lambda \in P_{+} \hspace{0.5mm}|\hspace{0.5mm} \lambda(\theta) \leqslant1 \}
	$$
	where $\theta$ is the longest root of $\Lg$. Given $\lambda \in P^{1}_{+}$ we denote the corresponding simple module by $L_{1}(\lambda)$, it is the unique simple quotient of the Weyl module $V^{1}_{\lambda}$.	
	
	Because $L_{1}(\Lg)$ is a regular vertex algebra (see \textit{e.g.}, Theorem 3.7 of \cite{DongLiMason1997}), any~$V^{k[1]-1}(\Lg) \otimes L_{1}(\Lg)$-module can be written 
	\begin{equation}
		\label{equation decomposition coset L1lambda}
		\bigoplus_{\lambda \in P^{1}_{+}} M_{\lambda} \otimes L_{1}(\lambda)
	\end{equation}
	with $M_{\lambda} \in V^{k[1]-1}(\Lg)\mathrm{-Mod}$. If we assume moreover that $M \in V^{k[1]-1}(\Lg)\otimes L_{1}(\Lg)\mathrm{-Mod}^{\jetinf G}$ then it is easy to see that for all $\lambda \in P^{1}_{+}$ we must have $M_{\lambda} \in V^{k[1]}(\Lg)\mathrm{-Mod}^{\jetinf \check{G}}$. 
	
	Now if we fix $\mu \in P_{+}$ and $\lambda \in P^{1}_{+}$ we can consider the simple $V^{k[1]-1}(\Lg) \otimes L_{1}(\Lg)$-module 
	$$
	V_{\mu}^{k[1]-1} \otimes L_{1}(\lambda).
	$$ 
	It belongs to $V^{k[1]-1}(\Lg)\otimes L_{1}(\Lg)\mathrm{-Mod}^{\jetinf G}$ if and only if 
	$$
	\mu + \lambda \in Q,
	$$
	\textit{i.e.},
	$$
	\lambda \in -\mu + Q.
	$$
	The set $P^{1}_{+}$ consists of $0$ and the minuscule weights (see Chapter VI, §1, Exercice~24 of \cite{Bourbaki2002}). One can show that minuscule weights are lifts of nonzero classes of the quotient $P/Q$ (Chapter VI, §2, Exercice 5 of \textit{loc.cit.}). This means that given $\mu \in P_{+}$ there exists a unique $\lambda = \lambda(\mu) \in P^{1}_{+}$ such that 
	$$
	\mu + \lambda(\mu) \in Q.
	$$
	This observation paired with equation \eqref{equation decomposition coset L1lambda} and Theorem \ref{theorem Kazhdan--Lusztig categories for irrational levels} shows that the category 
	$$
	V^{k[1]-1}(\Lg)\otimes L_{1}(\Lg)\mathrm{-Mod}^{\jetinf G}
	$$
	is semisimple with simple objects given by the 
	$$
	V_{\mu}^{k[1]-1} \otimes L_{1}(\lambda(\mu))
	$$
	with $\mu \in P_{+}$. Hence as abelian categories we indeed have 
	$$
	V^{k[1]-1}(\Lg)\otimes L_{1}(\Lg)\mathrm{-Mod}^{\jetinf G} \simeq \check{G}\mathrm{-Mod}.
	$$
	Now to see that the simple objects are mapped to one another under this correspondence, it suffices to combine the branching rules (Theorem 11.1 of \cite{ArakawaCreuztigLinshaw2019}) and Arakawa--Frenkel duality (Theorem 2.2 of \cite{ArakawaFrenkel2019}).
	\begin{remarque}
		\label{remarque right hypothesis implies left decomposition simple} 	
		The semisimplicity of the category $\mathcal{W}_{G}^{k}\mathrm{-Mod}^{\jetinf G}$ is very surprising to us. For instance, it does not seem to be known (beyond the $\mathfrak{sl}_{2}$ case) that the full subcategory of $W^{k}(\Lg)\mathrm{-Mod}$ generated by the $T^{k}_{\lambda,0}$ with $\lambda \in P_{+}$ (or the $T^{k}_{0, \check{\mu}}$ with $\check{\mu} \in \check{P}_{+}$) is semisimple. 
	\end{remarque}
	
	\begin{remarque}
		Remarkably, Theorem \ref{theorem computation of shifted cdo's} also shows that $\mathcal{W}_{G}^{k}[1]$ comes from a $1$-parameter family of vertex algebras (compare with Remark \ref{remark open problem deformable family}). The key fact we use from \cite{CreutzigNakatsuka2023} is that there exists exactly one structure of a simple vertex algebra on the $W^{k}(\Lg) \otimes V^{k[1]}(\Lg)$-module 
		$$
		\bigoplus_{\lambda \in X^{*}(T)_{+}} T^{k}_{\lambda,0}\otimes V^{k[1]}_{-\omega_{0}\lambda}
		$$
		that is a conformal extension of $W^{k}(\Lg)\otimes V^{k[1]}(\Lg)$ given by that of $V^{k[1]-1}(\Lg) \otimes L_{1}(\Lg)$. By the proof of Theorem \ref{proposition unicity of shifted cdos} we see that for all $n \in \Z$ there exists a unique structure of a simple vertex algebra on the $W^{k}(\Lg)\otimes V^{k[n]}(\Lg)$-module 
		$$
		\bigoplus_{\lambda \in X^{*}(T)_{+}} T^{k}_{\lambda,0}\otimes V^{k[n]}_{-\omega_{0}\lambda}
		$$
		that is a conformal extension of $W^{k}(\Lg)\otimes V^{k[n]}(\Lg)$ given by that of $\mathcal{W}_{G}^{k}[n]$. In particular, for $n=0$, it shows that the vertex algebra structure on $\mathcal{W}_{G}^{k}$ is determined by the fact that it is a simple conformal extension of $W^{k}(\Lg)\otimes V^{k^{*}}(\Lg)$ together with its Peter--Weyl decomposition (Lemma \ref{lemma chiral peter weyl decomposition for equivariant walgebra}). It seems reasonable to expect the same kind of uniqueness statement beyond the simply laced case and principal cases, as well as for the chiral universal centralizer (see §\ref{subsection beyond the principal case} and §\ref{subsection the chiral universal centralizer}).
	\end{remarque}

	\section{Further perspectives}  
	\label{section further perspectives}
	
	\subsection{Beyond the principal case}
	\label{subsection beyond the principal case}
	Assume that $G$ is a simple algebraic group and that $k \in \C \setminus \Q$. In this work we considered the quantum Hamiltonian reduction functor $H^{*}_{DS} = H^{*}_{DS,f_{prin}}$ associated with a principal nilpotent element $f_{prin}$. From there we defined the principal affine $\mathcal{W}$-algebra $W^{k}(\Lg) = W^{k}(\Lg,f_{prin})$ and the principal equivariant affine $\mathcal{W}$-algebra $\mathcal{W}_{G,f_{prin}}^{k}$ on~$G$.
	
	More generally, given any nilpotent element $f$ there exists a quantum Hamiltonian reduction functor $H^{*}_{DS,f}$ associated with $f$ by the work of \cite{kac-roan-wakimoto-quantum-reduction}. One may define in a similar fashion the affine $\mathcal{W}$-algebra associated with $(\Lg,f)$ at level $k$ to be  
	$$
	W^{k}(\Lg,f) := H^{*}_{DS,f}(V^{k}(\Lg))
	$$ 
	and the equivariant affine $\mathcal{W}$-algebra associated with $(G,f)$ at level $k$ by
	$$
	\mathcal{W}_{G,f}^{k} := H^{*}_{DS,f}(\cD_{G}^{k}).
	$$
	This vertex algebra comes equipped with two morphisms 
	$$
	W^{k}(\Lg,f) \lra \mathcal{W}_{G,f}^{k} \longleftarrow V^{k^{*}}(\Lg)
	$$
	so that we can define the category $\mathcal{W}_{G,f}^{k}\mathrm{-Mod}^{\jetinf G}$. It is natural to ask for a description of this category. Unlike in the principal case, there is to the best of the author's knowledge no clear quantum geometric Langlands interpretation of this category. An affine version of the Skryabin equivalence was studied in the principal case in \cite{Raskin2021} and this was a motivation to state Conjecture \ref{conjecture FLE}. Since the Skryabin equivalence holds for any nilpotent element $f$ (\cite{Skryabin2002}), it seems natural to expect an affine analogue involving $W^{k}(\Lg,f)$.
	
	Our construction still provides certain (possibly) interesting modules. Given a cocharacter $\gamma \in X_{*}(T)$ we can once again consider the family of $\mathcal{W}_{G,f}^{k}$-modules $H^{*}_{DS,f}(\gamma \cdot \mathcal{D}_{G}^{k})$. If the work of Arakawa and Frenkel \cite{ArakawaFrenkel2019} can be generalized beyond the principal case, one may carry over what we did in §\ref{section building modules on ewalg}. For instance if one can show that for $\lambda \in P_{+}$ the $W^{k}(\Lg,f)$-module 
	$$
	T_{\lambda,\gamma}^{k} := H^{*}_{DS,f}(\gamma \cdot V_{\lambda}^{k}) 
	$$ 
	is either zero or simple, then the proof of Theorem \ref{theorem simplicity of modules for equivariant w algebra} applies to show that the modules $H^{*}_{DS,f}(\gamma \cdot \mathcal{D}_{G}^{k})$ are either zero or simple. This would then allow for the formulation of a conjecture in the spirit of Conjecture \ref{conjecture FLE}. This might in turn be related to the work of Creutzig and Linshaw on the trialities of $\mathcal{W}$-algebras (\cite{creutziglinshawtrialities}, \cite{creutziglinshawtrialitiesortho}). We plan to explore these directions in the future in collaboration with Thomas Creutzig. 
	
	The question of generalizing Arakawa--Frenkel duality beyond the principal case seems particularly interesting to us. 
	
	\subsection{The chiral universal centralizer}
	\label{subsection the chiral universal centralizer}
	We keep the previously introduced notations and consider a pair $f,f'$ of nilpotent elements of $\Lg$. Recall that we have morphisms of vertex algebras
	\begin{align*}
		V^{k}(\Lg) \lra &\mathcal{D}_{G}^{k} \longleftarrow V^{k^{*}}(\Lg),\\
		W^{k}(\Lg,f) \lra &\mathcal{W}_{G,f}^{k} \longleftarrow V^{k^{*}}(\Lg).
	\end{align*}
	We may further apply, following §7 of \cite{Arakawa2018}, the functor $H^{*}_{DS,f'}$ with respect to the right embedding to obtain the chiral universal centralizer associated with $(G,f,f')$ at level $k$
	$$
	\mathcal{I}_{G,f,f'}^{k} := H^{*}_{DS,f'}(\mathcal{W}_{G,f}^{k}).
	$$
	This vertex algebra comes equipped with two morphisms 
	$$
	W^{k}(\Lg,f) \lra \mathcal{I}^{k}_{G,f,f'} \longleftarrow W^{k^{*}}(\Lg,f').
	$$
	
	Let us comment that our construction builds (possibly) interesting modules for~$\mathcal{I}^{k}_{G,f,f'}$. Recall from Proposition \ref{proposition computation of the spectral flow group for the cdo} that 
	$$
	\cSF(\cdo) \simeq X_{*}(T) \times \check{P}.
	$$
	We mostly exploited the subgroup $X_{*}(T) \times \{0\} \subset \cSF(\cD_{G}^{k})$, because these are the only parameters that produce objects in $\mathcal{W}_{G,f}^{k}\mathrm{-Mod}^{\jetinf G}$. Let $(\gamma,x) \in X_{*}(T) \times \check{P}$, we can consider the $\mathcal{I}_{G,f,f'}$-module
	$$
	H^{*}_{DS,f'}(H^{*}_{DS,f}((\gamma,x)\cdot \cdo)).
	$$
	
	For instance if $f=f'=f_{prin}$ are principal nilpotent elements, and $\mathcal{I}_{G}^{k} := \mathcal{I}^{k}_{G,f_{prin},f_{prin}}$ then the category 
	$
	\mathcal{I}_{G}^{k}\mathrm{-Mod}
	$
	should be related to $\kappa$-twisted $\mathcal{D}$-modules on the loop group satisfying Whittaker equivariance conditions with respect to both the left and right actions of $\mathcal{L}G$ on itself. We are unaware of any published references mentioning conjectures on this category or its relevance to the quantum geometric Langlands program.
	
	Note that when the level $\kappa$ is critical, \textit{i.e.}, $\kappa=\kappa^{*}$, and $f,f'$ are principal or zero, then $\mathcal{I}^{k}_{G,f,f'}$ has been introduced and studied by Arakawa from a vertex algebraic point of view in his work on the class $\mathcal{S}$ (\cite{Arakawa2018}). This level is very special as it allows for the definition of an inverse quantum Hamiltonian reduction procedure using the Feigin--Frenkel center (see §9 of \textit{loc.cit.}).

	For irrational values of $k$, in view of the chiral Peter--Weyl decomposition (Theorem~\ref{theorem chiral Peter-Weyl}), we have the equality as $V^{k}(\Lg) \otimes V^{k^{*}}(\Lg)$-modules:
	$$
	(\gamma,x) \cdot \cdo = \bigoplus_{\lambda \in X^{*}(T)_{+}} ((\gamma + x) \cdot V_{\lambda}^{k}) \otimes (\omega_{0}^{T}(x) \cdot V_{-\omega_{0}\lambda}^{k^{*}}).
	$$
	So that by taking the left and right quantum Hamiltonian reduction we obtain by Proposition \ref{proposition computation of the Hamiltonian reduction of weyl modules} the equalities as $W^{k}(\Lg) \otimes W^{k^{*}}(\Lg)$-modules:
	\begin{align*}
		H^{*}_{DS}(H^{*}_{DS}((\gamma,x)\cdot \mathcal{D}_{G}^{k})) &= \bigoplus_{\lambda \in X^{*}(T)_{+}} T_{\lambda, \gamma + x}^{k} \otimes T^{k^{*}}_{-\omega_{0}\lambda, \omega_{0}^{T}(x)} &&\textrm{ if $\gamma + x,\omega_{0}^{T}(x) \in \check{P}_{+}$,}\\
		H^{*}_{DS}(H^{*}_{DS}((\gamma,x)\cdot \mathcal{D}_{G}^{k})) &= 0 &&\textrm{ else.}
	\end{align*}
	We do not know how to prove that these modules are simple, but we conjecture that they are. 
	
	\begin{remarque}
		The free field realization of the chiral universal centralizer given in~\cite{On_The_Beem_Nair} might be useful in the study of its representation theory. We note that for $G = \mathrm{SL}_{2}$ and generic $\kappa$ the chiral universal centralizer $\mathcal{I}_{G}^{k}$ was introduced by Frenkel and Styrkas in~\cite{FrenkelStyrkas2006} and Frenkel and Zhu in~\cite{frenkelzhusl2centralizer}. McRae and Yang have studied the representation theory of $\mathcal{I}_{SL_{2}}^{k}$ for generic values of $k$ (see Remark 8.4 of \cite{McRaeYang}) and for $k=-1$ (see Theorem 8.2 of \textit{loc.cit.}). In particular, in their result for $k=-1$, the Langlands dual group $\mathrm{PSL}_{2}$ of $\mathrm{SL}_{2}$ appears. 
	\end{remarque}
	
	Of course, a similar construction holds for arbitrary pairs of nilpotent elements~$(f,f')$ and the study of the representation theory of $\mathcal{I}_{G,f,f'}^{k}$ appears to us as a natural and wide-open question.

	\subsection{Studying the vertex algebras of shifted chiral differential operators}
	
	The $\Z$-family of vertex algebras $\cD_{G}^{k}[n]$ introduced in §\ref{section shifted chiral differential operators on a simple group} played a central role in our work. Their existence is still conjectural in general and when they have been shown to exist, they are rather mysterious vertex algebras. For instance, their associated schemes do not seem to be known. We have answered these questions in the case of an algebraic torus (see Chapter 10 of \cite{my_phd}).
	
	It would be interesting to have more examples beyond the multiplicative group case, for instance in the case of $G=\mathrm{PSL}_{2}$. More generally this raises the general question of understanding the interaction between the convolution operation and the operation of passing to the associated scheme.

	\subsection{Beyond irrational levels}
	
	Another natural question is that of rational levels. One first step is to understand what happens to the Peter--Weyl decomposition of $\cdo$, and this has been studied for simple groups by Zhu (\cite{Zhu2011}). In the generic case, the theory of quantum groups was somehow hidden all along, it seems that for the rational case it should play a deep role (recall also Corollary \ref{corollaire of theorem fle tore} in which the combinatorics that appear seem to be related to the duality Lusztig introduces in §2.2.4 of \cite{Introduction_to_quantum_groups}). Some of the technology we used in our work, for instance the spectral flow group and the twisted quantum Hamiltonian reduction, makes sense for all levels. It is unclear to us what the shifted algebras of chiral differential operators should be in this case and what role they should play. Understanding the interplay between twisted principal Hamiltonian reduction and the Feigin--Frenkel duality beyond irrational levels also seems interesting but subtle (see~§4.4 of \cite{ArakawaFrenkel2019}). 
	
	The most general question our work raises is the following: given a reductive group~$G$, a level $\kappa$ and two nilpotent elements $f,f' \in \Lg$, how does one define and study a suitable subcategory of 
	$
	\mathcal{I}_{G,f,f'}^{\kappa}\mathrm{-Mod}
	$?
	\printbibliography	
\end{document}